\documentclass[10pt]{amsart}

\usepackage[utf8]{inputenc}

\usepackage{MnSymbol}
\usepackage{amsmath}
\usepackage{amsthm}
\usepackage{physics}
\usepackage{graphics}
\usepackage{amsfonts}
\usepackage{verbatim}
\usepackage{booktabs}
\usepackage[english]{babel}
\usepackage{graphicx}
\usepackage{algorithm}
\usepackage{algpseudocode}
\usepackage{float}
\usepackage{xcolor}
\usepackage{comment}
\usepackage{todonotes}
\usepackage{enumitem}

\usepackage{listings}
\usepackage{xcolor}
\usepackage{hyperref}
\usepackage{cleveref}
\def\im{\mathrm{im\,}}

\def\kk{\mathbb{K}}

\def\PP{\mathbb{P}}
\def\NN{\mathbb{N}}

\def\ub{\mathbf{u}}

\def\mfm{\mathfrak{m}}

\def\eval{\mathfrak{e}}

\def\hk{{\mathfrak{h}}}

\def\Sk{{\mathfrak{S}}}

\def\KK{\mathbb{K}}

\def\<{\langle}
\def\>{\rangle}
\newcommand{\invsys}[1]{\langle\langle{#1}\rangle\rangle}

\newcommand{\dual}[1]{#1^{\ast}}

\newcommand{\apply}[2]{ {#1}\textcolor{\revcolor}{\,\raisebox{-1mm}{\scalebox{2}{$\cdot$}}\,}{#2}}

\newcommand{\vspan}[1]{\langle #1 \rangle}

\newcommand{\set}[1]{\{ #1 \}}
\newcommand{\reg}{\mathrm{reg}}

\def \ann {\mathrm{Ann}}

\def\eval{{\bm e}}

\newcommand{\KKrng}{\KK[\xaff]}

\newcommand{\Okdn}{\cl O_{d}^{k}}
\newcommand{\N}{\mathbb N}

\newcommand{\C}{\mathbb C}

\newcommand{\K}{\mathbb K}

\def\eval{\mathbf{e}}

\def\PolExp{\mathcal{P}ol\mathcal{E}xp}
\newcommand{\cl}[1]{\mathcal{#1}}
\newcommand{\bm}[1]{\mathbf{#1}}
\newcommand{\mf}[1]{\mathfrak{#1}}
\newcommand{\bms}[1]{\boldsymbol{#1}}

\def\xaff{{\bm x}}
\def\xprj{\underline{\bm x}}

\def\Iprj{\underline{I}}
\def\Aprj{\underline{\cl A}}
\def\zaff{{\bm z}}

\def\zprj{\underline{\bm z}}
\def\xiprj{\underline{\xi}}

\def\phiprj{\underline{\varphi}}

\def\Ann{\textup{Ann}}
\def\eval{\bm{e}}

\def\Sone{\S_{1}}

\def\nub{\bms \nu}

\def\S{\mathcal{S}}
\def\Sd{\S_d}
\def\Sk{\S_k}
\def\Sdk{\S_{d-k}}
\def\V{\mathcal{V}}
\def\O{\textnormal{O}}
\def\GL{\textnormal{GL}}
\def\End{\textnormal{End}}
\usepackage{MnSymbol}
\newcommand{\contract}{{\righthalfcup}\,}
\newcommand{\gadrank}{\rank_{\scriptstyle\mathrm{GAD}}}
\newcommand{\cactusrank}{\rank_{\scriptstyle\mathrm{cactus}}}
\newcommand{\Hom}{\mathrm{Hom}}

\usepackage{graphicx}
\usepackage{subcaption} 
\usepackage{caption}

\lstdefinelanguage{Julia}{
  morekeywords={function,return,if,else,for,while,end,using},
  sensitive=true,
  morecomment=[l]\#,
  morestring=[b]",
  alsoletter=0123456789,
  morekeywords=[2]{@polyvar},
  keywordstyle=\color{blue},
  keywordstyle=[2]\color{teal},
  commentstyle=\color{gray}\ttfamily,
  stringstyle=\color{red},
  basicstyle=\ttfamily\footnotesize,
  breaklines=true,
}
\usepackage{tcolorbox}
\tcbuselibrary{listingsutf8}
\usepackage{lipsum}

\def\revcolor{black}
\newcommand{\brev}{\color{\revcolor}}
\newcommand{\erev}{\color{black}}

\title{Generalized Additive Decompositions of Symmetric Tensors}

\author{Enrica Barrilli}
\address{Enrica Barrilli \\ Inria at Université Côte d'Azur \\ 2004 route des Lucioles, 06902 Sophia Antipolis, France}
\email{enrica.barrilli@inria.fr}

\author{Bernard Mourrain}
\address{Bernard Mourrain \\ Inria at Université Côte d'Azur \\ 2004 route des Lucioles, 06902 Sophia Antipolis, France}
\email{bernard.mourrain@inria.fr}

\author{Daniele Taufer}
\address{Daniele Taufer \\ Max Planck Institute for Mathematics in the Sciences \\ Inselstraße 22, 04103 Leipzig, Germany}
\email{daniele.taufer@gmail.com}

\date{}



\newtheorem{theorem}{Theorem}[section]
\newtheorem{corollary}[theorem]{Corollary}
\newtheorem{lemma}[theorem]{Lemma}
\newtheorem{proposition}[theorem]{Proposition}

\theoremstyle{definition}
\newtheorem{definition}[theorem]{Definition}
\newtheorem{remark}[theorem]{Remark}
\newtheorem{example}[theorem]{Example}

\begin{document}

\maketitle

\vspace{-0.6cm}
\begin{abstract}
This article addresses the Generalized Additive Decomposition (GAD) of symmetric tensors, that is, degree-$d$ forms $f \in \Sd$.
From a geometric perspective, a GAD corresponds to representing a point on a secant of osculating varieties to the Veronese variety, providing a compact and structured description of a tensor that captures its intrinsic algebraic properties.

We provide a linear algebra method for measuring the GAD size and prove that the minimal achievable size, which we call the GAD-rank of the considered tensor, coincides with the rank of suitable Catalecticant matrices, under certain regularity assumptions.
We provide a new explicit description of the apolar scheme associated with a GAD as the annihilator of a polynomial-exponential series.
We show that if the Castelnuovo-Mumford regularity of this scheme is sufficiently small, then both the GAD and the associated apolar scheme are minimal and unique.

Leveraging these results, we develop a numerical GAD algorithm for symmetric tensors that effectively exploits the underlying algebraic structure, extending existing algebraic approaches based on eigen computation to the treatment of multiple points. 
We illustrate the effectiveness and numerical stability of such an algorithm through several examples, including Waring and tangential decompositions.
\end{abstract}
%\tableofcontents

\section{Introduction}

Tensors are fundamental mathematical objects that generalize matrices to higher dimensions and play a crucial role in various domains. While matrices, namely second-order tensors, are well-known and widely used, higher-order tensors offer a more natural representation for data involving more than two interacting variables.
These objects provide an effective means to organize and model such data, capturing intricate relationships across these multiple dimensions.
Consequently, understanding and efficiently manipulating tensors is essential for real-world problems that involve multidimensional data.

Tensor decomposition, i.e., expressing a given tensor as a sum of simpler tensors, has emerged as a key technique in the analysis of higher-order tensors 
%This process is important because it reveals the underlying structure of the data, reduces its complexity, and facilitates more efficient computations.
%Tensor decomposition is a fundamental tool 
for uncovering latent structure in multidimensional data, impacting diverse domains.
%as signal processing, psychometrics, medical imaging, telecommunications, and machine learning.
In blind source separation, higher-order cumulant tensors provide algebraic structures that enable the identification of independent sources, even in underdetermined mixtures \cite{cardoso1998blind, DeLathauwer2007}, while in independent component analysis (ICA), they support consistent recovery of independent components in overcomplete regimes \cite{hyvarinen2000independent}.
In psychometrics, tensor methods were introduced early through the $N$-way generalization of the Eckart–Young decomposition \cite{carroll1970analysis}, and in medical imaging, spectral decompositions of covariance tensors are central to diffusion MRI for characterizing tissue microstructure \cite{basser2007spectral}.
In telecommunications, multilinear decompositions underpin blind receivers and CDMA separation algorithms \cite{chevalier1999optimal, de2007tensor, sidiropoulos2002blind, van2002analytical}, while more broadly, tensors provide a natural language for sparse representations and uncertainty principles in signal processing \cite{donoho2001uncertainty}.
In statistics and machine learning, moment tensors are now a standard tool for learning latent variable models, enabling efficient algorithms for Gaussian mixtures \cite{ge2015learning, hsu2013learning, pereira2022tensor}, nonparametric mixture models \cite{zhang2023moment}, and hidden Markov models \cite{allman2009identifiability, anandkumar2014tensor}.

Crucially, many of these applications rely not on arbitrary tensors but on symmetric tensors: cumulant tensors, covariance tensors, and moment tensors are all permutation-invariant by construction.
%Their decomposition is thus not only a means of reducing dimensionality but the key to extracting interpretable latent structure.
While general-purpose methods such as the Canonical Polyadic Decomposition (CPD) are widely used, they do not exploit the additional structure encoded by symmetry.
In contrast, symmetric tensor decompositions capture the intrinsic algebraic and geometric properties of these problems.
Moreover, techniques from algebraic geometry provide deep insights into the geometric structure of tensors, enabling the development of decomposition algorithms.
This paper contributes to this line of work by presenting an algorithm for the Generalized Additive Decomposition (GAD) of symmetric tensors, which extends the Waring framework.
This generalization offers new techniques for computing compact tensor representations and revealing intrinsic information, thereby broadening their scope and applicability.

\smallskip
\paragraph{\emph{State of the art.}}
The decomposition of a symmetric tensor into powers of linear forms, i.e., the classical Waring decomposition problem, has been deeply investigated from numerical, algebraic, and geometric viewpoints.
Rooted in the works of G. Riche de Prony \cite{baron_de_prony_essai_1795} and J. J. Sylvester \cite{sylvester1851lx}, foundational contributions based on apolarity theory and catalecticant constructions \cite{Iarrobino1999, landsberg2015new} have provided the basic framework to interpret tensor decompositions in terms of apolar ideals, and secant varieties of the Veronese \cite{Bernardi2007}.
This geometric setting has been refined through the study of generalized eigenvectors and roots of algebraic equations \cite{BernardiReigFite2025, Bernardi20201, Brachat2010, oeding2013eigenvectors, qi2018tensor}, as well as through the explicit determination of equations defining secant and cactus varieties \cite{buczynski2024cactus,landsberg2013equations}.

Parallel developments have focused on computational efficiency, leveraging simultaneous diagonalization \cite{comon2009tensor}, FOOBI-type and eigen decomposition algorithms \cite{anandkumar2014tensor, de2000multilinear}, and the moment matrix extension framework \cite{Bernardi2013, Brachat2010, shi2025efficient, Telen2022}, which connects apolarity to structured linear algebra.
In addition, several numerical and geometric approaches have been explored, including alternating least squares \cite{kolda2009tensor}, power iteration schemes \cite{zhang2001rank}, and homotopy continuation methods \cite{bernardi2017tensor}.

The problem of computing GADs of a given symmetric tensor, as introduced in \cite{Iarrobino1999}, is closely related to the study of Hilbert schemes of points \cite{buczynski2024cactus,Buczyska2013}, since every GAD canonically defines a (possibly non-reduced) zero-dimensional projective scheme apolar to the considered tensor \cite{bernardi2018polynomials,Bernardi2024}.
The Hilbert function of such schemes is what governs the GAD complexity \cite{Bernardi2024,Buczyska2013,TonyMemoir}, and understanding these functions for prescribed tensors remains an open and significant research \brev{}task, which has proved to be challenging even for local cases \cite{BernardiReigFite2025,reigtaufer2026}\erev{}. %was this intended? \cite{bernardi2017tensor}
%More recent advances have clarified the structure of apolar schemes and their Hilbert functions \cite{bernardi2018polynomials} and have investigated the theory of Generalized Additive Decompositions (GADs) introduced in \cite{Iarrobino1999}, relating additive schemes to their regularity and geometric realization \cite{Bernardi2024} via the moment extension approach \cite{BernardiReigFite2025}.

\smallskip
\paragraph{\emph{Novel contributions.}}
In this paper, we provide a detailed study of GADs of symmetric tensors $f$ of order $d$ or, equivalently, degree-$d$ forms $f\in \Sd$. 
We provide a geometric interpretation of GADs as a representation of points on secants of the smooth part of osculating varieties.
We prove that the GAD-size computation can be performed entirely via linear algebra, improving previous quasi-linear constructions, by evaluating the dimensions of inverse systems associated with prescribed forms.
We define the minimum of such sizes as the \emph{GAD-rank} $\gadrank(f)$.
We show that a minimal GAD of $f$, i.e., a GAD realizing $\gadrank(f)$, can be computed from certain Catalecticant matrices, under regularity assumptions.
We provide a new explicit description of the corresponding apolar scheme as annihilators of polynomial-exponential series (see \Cref{thm:rk gad,thm:rk cactus} for more details):

\smallskip
\noindent \textbf{Theorem.} {\it Let $f\in \Sd$ with a GAD of the form $f= \sum_i \omega_i \ell_i^{d-k_i} $ such that $\omega_i \in \S_{k_i}$ and $\ell_i \in \Sone$.
If the Castelnuovo-Mumford regularity of $\S/\underline{I}$ is not greater than $\lfloor\frac{d-1}{2}\rfloor$, where $\underline{I}$ is the ideal associated with the GAD, then 
$$
\gadrank(f) = \cactusrank(f)= \rank H, 
$$
where $\cactusrank(f)$ is the minimal length of a scheme apolar to $f$, and $H$ is the Catalecticant matrix of $f$ in degree $(d-c,c)$ with $c= \lfloor \frac{d-1}{2} \rfloor$.
Moreover, in this case, the considered GAD is the unique minimal one, and defines the unique minimal apolar scheme to $f$.
}

\smallskip
\brev{}It generalizes \cite[Theorem 1.6]{Buczyska2013}, replacing the rank by the regularity.\erev{} 
As a consequence, when $\omega_i \in \kk=\S_0$ and the regularity of the quotient by the vanishing ideal of the points $\ell_i\in \Sone$ is not greater than $\lfloor\frac{d-1}{2}\rfloor$, we deduce that $f$ has a unique Waring decomposition.

This leads to a numerical GAD algorithm for symmetric tensors, which runs as if the regularity of the unknown decomposition were at most $\lfloor \frac{d-1}{2} \rfloor$ and can verify the hypothesis a posteriori.
This algorithm refines \cite[Algorithm 3]{Bernardi20201}, establishing terminating conditions that prevent the issue associated with high nil-indices (equivalently, high socle degrees for the associated schemes).
Its algorithmic realization significantly extends algebraic approaches based on eigen computation by providing a flexible and efficient decomposition technique that numerically handles tensor decompositions corresponding to multiple points.
The GAD algorithm exhibits approximately linear reconstruction-error growth in the tested Waring decompositions, when the regularity of the associated ideal is at most $\lfloor \frac{d-1}{2} \rfloor$, and it likewise demonstrates robust performance for symmetric tensor decompositions with multiple points, which is tested on several examples.
%Our approach combines geometric techniques from algebraic geometry with classical linear algebra methods to enhance the efficiency and scalability of tensor decomposition. 
%We illustrate, on several examples, the numerical stability of the algorithm.
By addressing the limitations of existing algorithms, we provide a method that is more general and applicable to a broader range of problems, and could pave the way for computing minimal apolar schemes and cactus ranks using extension techniques such as \cite{Bernardi20201,Brachat2010}.

\medskip
\paragraph{\emph{Paper structure.}}
In \Cref{sec:Preliminaries}, we recall the mathematical background and key results about apolarity, which will be used throughout the paper.
\Cref{sec:GAD} is devoted to defining GADs, their rank, and their geometrical interpretation.
General properties of Artinian algebras are discussed in \Cref{sec:algop-GAD}, and are used to prescribe algebraic properties of the operators and algebras constructed from a GAD.
These properties are then employed in \Cref{sec:ApolarSchemes} to define zero-dimensional apolar schemes from GADs, which are used to prove the uniqueness of GADs and minimal apolar schemes under regularity assumptions.
Finally, in \Cref{sec:AlgoExp}, we employ the results of the previous sections to design an algorithm for recovering such a unique minimal GAD, prove its effectiveness, and discuss its numerical robustness.

%Section 5 describes the Generalized Additive Decomposition algorithm in detail, presenting its mathematical foundations and computational steps. In Section 6, we demonstrate the effectiveness of the proposed method through several illustrative examples and tests. 

%\section{Apolarity and Generalized Additive Decomposition}
 %Definition/properties/reformulation of GAD/geometric point of view
 
%\section{Structure of Artinian Algebras and their dual}
 %Recall properties: sum of local algebras,  operators of multiplication, diagonal blocs, 
 %nil-index, inverse systems, degree

%\section{Low rank Hankel matrices}

%\section{Decomposition algorithm}
%\section{Tests}

\section{Preliminaries} \label{sec:Preliminaries}
%\subsection{Notation BM} 
Let $\kk$ be \brev{}an \emph{algebraically closed} field\erev{} of characteristic $0$.
\brev{}Hereafter $n>0$, $d>1$ will always denote integers. \erev 
We employ the standard multi-index notation: for $\alpha = (\alpha_{0}, \ldots, \alpha_{n})\in \NN^{n+1}$, we denote
$$
|\alpha| := \alpha_0+ \cdots + \alpha_n, % \brev{}= d \erev{},
\quad \alpha! := \alpha_{0}! \cdots \alpha_{n}!, \quad \textnormal{and} \quad {d \choose \alpha} := \frac{d!}{\alpha!}.
$$
%Let $\xprj= (x_0, \ldots, x_n)$ be a basis of $\cl S_1(\kk^{n+1})$.

For any $\kk$-vector space $V$, \brev{}let $\Sd(V)$ denote the $d$-th symmetric power of $V$\erev{}.
When $V = \kk^{n+1}$, we identify this space with $\kk[x_0, \ldots, x_n]_{d} =\kk[\xprj]_{d}$, the space of homogeneous degree-$d$ polynomials in the variables $\xprj = (x_0, \ldots, x_n)$, and we shortly denote it by $\S_d$.
A non-zero element in $\Sd$ is also called a \emph{form} of degree $d$.
We denote the full graded algebra of forms (of any degree) by $\S := \oplus_{d = 0}^{\infty} \S_d$. 
\brev For $\bm v = \set{v_1, \dots, v_m} \subset V$, we denote by $\vspan{\bm v}$ the vector space spanned by $\bm v$ and $\Sd(\vspan{\bm v})$ by $\Sd(\bm v)$, hence $\S(\bm v) = \kk[v_1, \dots, v_m]$.
We denote the ideal generated by a subset $K$ of a ring by $(K)$. \erev

\brev{}Throughout the paper, the underlined bold symbols $\underline{\mathbf{x}}$ and $\underline{\mathbf{z}}$, denote a collection of $n+1$ variables, e.g., $\underline{\mathbf{x}}=(x_0,\ldots,x_n)$, as typically used in the homogeneous setting, whereas the corresponding \emph{non-underlined} bold symbol denotes a collection of $n$ variables, e.g., $\mathbf{x}=(x_1,\ldots,x_n)$, as typically used in the \emph{inhomogeneous} setting.\erev{}

\subsection{Projective duality}
The dual space $\dual{\Sd} := {\brev{}\Hom\erev{}}_{\kk}(\Sd,\kk)$ is the set of $\kk$-linear functionals from $\Sd$ to $\kk$.

For a given $\varphi \in \dual{\Sd}$, we denote its application to $f \in \Sd$ by $\apply{\varphi}{f} := \varphi(f)$.
An element $\varphi \in \dual{\Sd}$ is represented uniquely as a \emph{differential polynomial}
$$
\varphi(\zprj) = \sum_{|\alpha|=d} \varphi_{\alpha} \frac {\zprj^{\alpha}} {\alpha!},
$$
where $\varphi_{\alpha}:=\apply{\varphi}{\xprj^{\alpha}}\in \kk$ is the value of the linear functional $\varphi$ on $\xprj^{\alpha}\in \Sd$, and
$$
\zprj^{\alpha} : \S \to \kk, \quad p \mapsto \partial_{\xprj}^{\alpha}(p)(0) := ( \partial_{x_0}^{\alpha_0} \cdots \partial_{x_n}^{\alpha_n} p)(0).
$$
In this representation, $(\frac {\zprj^{\alpha}} {\alpha!})_{|\alpha|=d}$ is the dual basis in $\dual{\Sd}$ of the monomial basis $( {\xprj^{\alpha}})_{|\alpha|=d}$ of $\Sd$.
We verify for monomials and by linearity for all $g(\zprj) \in \dual{\Sk}$, $h(\zprj)\in \dual{\cl S_{d-k}}$, and $f\in \Sd$, that we have 
\begin{equation}{\label{eq:proddual}}
\apply{gh(\zprj)}{f} = \apply{h(\zprj)}{g(\partial_{\xprj})(f)} := \apply{h(\zprj)}{g(\partial_{x_0}, \ldots, \partial_{x_n})(f)}.
\end{equation}

%Let $\ell=(\xiprj,\xprj) \in \Sone$ with $\xiprj\in \kk^{n+1}\setminus\set{0}$.
For any $\bm \xiprj = (\xiprj_0, \ldots, \xiprj_n) \in \kk^{n+1} $, we consider the associated linear form $ \ell = (\bm \xiprj, \xprj) := \sum_{i=0}^{n} \xiprj_i\, x_i\in \Sone$, and we denote the corresponding exponential series by
$$
   \eval_{\xiprj}(\zprj) := \eval_{\ell}(\zprj) = e^{\ell(\zprj)} = \sum_{k \in \NN} \frac{\ell(\zprj)^k}{k!} = \sum_{\alpha \in \NN^{n+1}} \frac{\xiprj^{\alpha} \zprj^{\alpha}}{\alpha!} \in \kk[[\zprj]]= \dual{\S}.
$$
We check from the definitions that $\eval_{\xiprj}(\zprj)$ represents the evaluation-in-$\xiprj$ linear functional 
$\eval_{\xiprj}: p\in \cl S\mapsto \apply{\eval_{\xiprj}}{p}= p(\xiprj)$. 
More generally, by \cref{eq:proddual} and linearity, for $\omega(\zprj) \in \kk[\zprj] \subset \dual{\cl S}$ and $\xiprj \in \kk^{n+1}$, we verify that the linear functional $\omega(\zprj) \eval_{\xiprj}(\zprj)\in \kk[[\zprj]]= \dual{\cl S}$ represents
\begin{equation}\label{eq:polyexprep}
\omega(\zprj) \, \eval_{\xiprj} : \S \to \kk, \quad p \mapsto \apply{\omega(\zprj) \eval_{\xiprj}(\zprj)}{p} = \omega(\partial_{\xprj})(p)(\xiprj).
\end{equation}

For any $p \in \S$ and $\varphi\in \dual{\S}$, we define $p \star \varphi \in \dual{\S}$ as the linear functional such that, $\forall q \in \S$,
$$
\apply{(p \star \varphi)}{q } := \apply{\varphi}{p q}.
$$
This \emph{transposed} product is also known as \emph{co-product} or \emph{contraction}.
Note that the operator $\S^* \to \S^*,\ \varphi \mapsto p \star \varphi$ is the transposed operator of $\S \to \S,\ q \mapsto pq$.
We verify that, for any $p\in \S$ and $\varphi\in \dual{\S}$, the operator $p\star \varphi$ is represented by the differential polynomial
$$
 p \star \varphi = p(\partial_{\zprj})\big(\varphi(\zprj)\big) := p(\partial_{z_0}, \ldots, \partial_{z_n}) \big(\varphi(\zprj)\big),
$$
 where $p(\partial_{z_0}, \ldots, \partial_{z_n})$ is the differential polynomial obtained by substituting the variable $x_i$ by the derivation $\partial_{z_i}$ with respect to the (formal) variable $z_i$.

\brev{}
\begin{definition}\label{def:apolar duality}
For any $f = \sum_{|\alpha|=d} f_{\alpha}\, \xprj^{\alpha} \in \Sd$, we define
$$
    f^*= \frac 1 {d!} \sum_{|\alpha|=d} f_{\alpha}\, \zprj^{\alpha} = \frac 1 {d!} f(\zprj) \in \Sd^*.
$$
\end{definition}
We verify that for any forms $f\, =\, \sum_{\vert\alpha\vert=d} f_{\alpha} \xprj^{\alpha}\in \Sd$ and $g\, =\, \sum_{\vert \alpha\vert=d} g_{\alpha} \xprj^{\alpha}\in \Sd$,
\begin{align}\label{eq:Bomb}
  \apply{f^*}{g} = \frac{1} {d!} f(\partial_{\xprj})(g) = \frac{1} {d!} g(\partial_{\xprj})(f) = \ \sum_{\vert\alpha\vert=d} {d \choose \alpha}^{-1} f_{\alpha} g_{\alpha}.
\end{align}
Moreover, we have the following properties:
%This apolar product depends on the choice of the coordinates $\xprj$ in $\cl S_1$. 
\begin{lemma}\label{lem:apol prop}
For every $f \in \Sd$, $g \in \Sk$, $h \in \Sdk$, and $\bm \xiprj \in \kk^{n+1}$, we have% the apolar product satisfies the following properties:
\begin{enumerate}[label=\textup{(\roman*)}, ref=\roman*]
    \item\label{apol-i} $\apply{(g\,h)^*}{f}\ =\ \frac{(d-k)!}{d!} \big(\apply{h^*}{g(\partial_{\xprj})(f)}\big)$,
    \item\label{apol-ii} $\apply{\big((\bm{\xiprj}, \xprj)^d\big)^*}{f}\ =\ f(\bm{\xiprj})$,
    \item\label{apol-iii} $\apply{\big(g\, (\bm \xiprj,\xprj)^{d-k}\big)^*}{f}\ =\ \frac{(d-k)!}{d!} g(\partial_{\xprj}) (f)(\bm{\xiprj})$.
\end{enumerate}
\label{properties}
\end{lemma}
\begin{proof}
Equality (\ref{apol-i}) follows from \cref{eq:Bomb}, since
$$
\apply{(g h)^*}{f} =\frac 1 {d!} g(\partial_{\xprj}) h(\partial_{\xprj}) (f)
= \frac 1 {d!}  h(\partial_{\xprj})  
\big(g(\partial_{\xprj})(f)\big)= \frac {(d-k)!} {d!}
\big(\apply{h^*}{g(\partial_{\xprj})(f)}\big),
$$
while (\ref{apol-ii}) follows by straightforward computation.
By combining (\ref{apol-i}) and (\ref{apol-ii}), we immediately obtain (\ref{apol-iii}).
\end{proof}
The bilinear form defined in \cref{eq:Bomb}
is known as the \emph{apolar or Bombieri-Weyl product} (with respect to $\xprj$). 
It is symmetric, non-degenerate, and invariant under the orthogonal group $\O(\kk^{n+1})$ \cite[Section 12.1]{unitary}, but not under the linear group $\GL(\kk^{n+1})$.
The monomial basis $\{\xprj^{\alpha}\}_{|\alpha|=d}$ of $\Sd$ is orthogonal for this apolar product with respect to $\xprj$. 
\erev{}

Given a functional $\varphi = \sum_{\alpha \in \NN^{n+1}} \varphi_{\alpha} \zprj^{\alpha} \in \S^*$, we denote its \emph{degree-d part} by the functional $\varphi^{[d]} := \sum_{|\alpha|=d}\varphi_{\alpha}\zprj^{\alpha} \in \Sd^*$.

%For $\xi\in \kk^{n+1}$, let $\eval_{\xiprj}^{k}(\zprj) = \sum_{d\ge k} \frac{(\xiprj, \zprj)^{d-k}} {d!} (\xiprj, \zprj)^{d-k}$.

Let $\ell = \brev{} (\xiprj, \xprj) \erev{}$ %\scp{\xiprj,\xprj}
$\in \Sone$ and let $\bm v= \set{v_1, \ldots, v_n}\subset \Sone$ such that $\set{\ell, v_1, \ldots, v_n}$ is a basis of 
$\Sone$.
Any $\omega \in \Sk$ can be decomposed uniquely as 
$\omega = \sum_{j=0}^{k} \omega_{j} \ell^{k-j}$ with $\omega_j \in \S_{j}(\bm v)$.
We define 
\begin{equation}\label{eq:omega v}
    \omega^{d, \ell, {\bm v}} := \frac{1}{d!}\sum_{j=0}^{k}  {(d-j)!}\, \omega_j(\bm v) \in \S(\bm v).
\end{equation}
Note that $\omega^{d, \ell, {\bm v}} \in \S(\bm v) \subset \S$ can be seen as a polynomial $\omega^{d, \ell, {\bm v}}(\xprj)$, since every $v_i \in \bm v$ is a linear form in $\S_1$.
As usual, we write $\omega^{d,\ell, \bm v}(\zprj)$ to denote its evaluation in the $\zprj$-variables.

\begin{lemma}\label{lem:w xi dual}
Let $f = \omega\, \ell^{d-k}\in \Sd$ with $d\ge k$, $\omega \in \Sk$, and $\ell=(\xiprj, \xprj)$ with $\xiprj\in \kk^{n+1}$.
Then, for any ${\bm v}$ as above, we have
$$
f^* =\big(\omega^{d,\ell, \bm v}(\zprj) \eval_{\ell}(\zprj)\big)^{[d]} 
%= (\omega(\zprj) \eval_{\xiprj}^{k}(\zprj))^{[d]}
\in \Sd^*.
$$
\end{lemma}
\begin{proof}
For $\omega = \sum_{j=0}^k \omega_j \ell^{k-j}$ with $\omega_j \in \S_j(\bm v)$, we have
\begin{align*}
f^* & 
= \frac{1}{d!} \omega(\zprj) \, \ell(\zprj)^{d-k} 
= \frac{1}{d!}  \sum_{j=0}^{k}\omega_j(\zprj) \ell(\zprj)^{d-j} 
= \frac{1}{d!} \sum_{j=0}^{k}\omega_j(\zprj) (d-j)! \, \frac{\ell(\zprj)^{d-j}}{(d-j)!} \\
& 
= 
\big(
 \frac{1}{d!} \sum_{j=0}^{k} (d-j)!\,\omega_j(\zprj)  \eval_{\ell}(\zprj) \big)^{[d]}
= \big(
 \omega^{d,\ell,\bm v}(\zprj) \eval_{\ell}(\zprj) \big)^{[d]},
\end{align*}
which proves the lemma.
\end{proof}

\subsection{Affine duality}\label{sec:affine duality}
For the decomposition of forms, we will study Artinian algebras, defined in an \emph{affine} setting, by localization of projective spaces.
In this section, we recall the notation and properties of the corresponding duality.

Let $R = \kk[x_1, \ldots, x_n] = \KKrng$ be the polynomial ring over $\kk$ in the variables $\xaff = (x_1, \ldots, x_n)$.
Its dual space is $\dual{R} = {\brev{}\Hom\erev{}}_{\kk}(R, \kk)$.
Similarly to the homogeneous setting, any linear functional $\varphi \in \dual{R}$ can be represented as a formal power series in $\zaff=(z_1,\ldots, z_n)$:
$$
 \varphi(\bm z)= \sum_{\alpha \in \NN^n} \varphi_{\alpha} \, \frac{\bm z^{\alpha}}{\alpha!} \in \kk[[\bm z]]\brev{}, \erev{}
$$ 
where $\varphi_{\alpha} := \brev{} \apply{\varphi}{\xaff^{\alpha}} \erev{} %\scp{\varphi,\xaff^{\alpha}}
= \varphi(\xaff^{\alpha}) $ is the value of the linear functional $\varphi$ on the monomial $\xaff^{\alpha}$, %$\zaff^{\alpha}:= z_1^{\alpha_1} \cdots z_n^{\alpha_n}$,  $\alpha! := \alpha_1! \times \cdots \times \alpha_n!$ for $\alpha=(\alpha_1, \ldots, \alpha_n)\in \NN^{n}$. Here
and $\zaff^{\alpha}$ represents the linear functional 
$$
\zaff^{\alpha}: R \to \kk, \quad p \mapsto \partial_{\xaff}^{\alpha}(p)(0) := ( \partial_{x_1}^{\alpha_1} \cdots \partial_{x_n}^{\alpha_n} p)(0).
$$
%We verify, for monomials and by bilinearity for all elements, that for $g(\zaff) \in \kk[\zaff]$ and $\varphi(\zaff)\in \dual{R}$, $p\in R$,
As in the homogeneous case, for $g(\zaff) \in \kk[\zaff]$, $\varphi(\zaff)\in \dual{R}$, and $p\in R$, we have
%\begin{equation}\label{eq:dual prop}
$$
\brev{} \apply{g(\zaff)\, \varphi(\zaff)}{p} = \apply{\varphi}{g(\partial_{\xaff})(p)} \erev{}
%\scp{g\, \varphi, p}= \scp{\varphi, g(\partial_{\xaff})(p)}. 
$$
%\end{equation}
The dual space $\dual{R}$ is again equipped with a \emph{co-product}:
$p \star \varphi: R \to \kk, \ q \mapsto \varphi(p q)$, which is represented by the series
\begin{equation}\label{eq:star prop}
p\star \varphi = p(\partial_{\zaff}) \varphi(\bm z) \in \kk[[\bm z]].
\end{equation}

For an ideal $I \subset R$, we define 
$$
I^{\perp} := \{\varphi \in \dual{R} \mid \forall p\in I, \varphi (p)=0 \}.
$$
We identify $I^{\perp}$ with the dual  $\dual{\cl A}$ of $\cl A = R/I$.
One easily sees that for all $p \in R$ and $\varphi \in I^{\perp}$, we have $p \star \varphi \in I^{\perp}$, namely $I^{\perp} \subset \kk[[\zaff]]$ is stable by derivations with respect to $\zaff$, as in \cref{eq:star prop}. 
Given $\varphi \in \dual{R}$, we define the \emph{inverse system} spanned by $\varphi$ as the vector space 
$$
\invsys{\varphi} := R \star \varphi 
= \{ p \star \varphi \mid p \in R \} 
= \{ p(\partial_{\zaff}) \varphi(\zaff) \mid p \in R \}
\subset \dual{R}.
$$
We notice that $\invsys{\varphi}^{\perp} := \{ p \in R \mid \forall \varphi'\in \invsys{\varphi}, \varphi'(p)= 0 \}$ is an ideal of $R$.
For $\omega \in R$, we also define the \emph{inverse system} of $\omega$ as 
$$
\invsys{\omega} := \vspan{ p(\partial_{\xaff})(\omega)\ | \ p \in R}\subset R.
$$

We define the set of \emph{polynomial-exponential series} as
$$
\PolExp := \left\{ \sum_{i= 1}^r \omega_i(\zaff) \, \eval_{\xi_i}(\zaff) \in \kk[[\zaff]] \ \big\mid \ \omega_i(\zaff)\in \kk[\zaff], \, \xi_i \in \kk^{n}\right\} \subset \kk[[\zaff]],
$$
where $\eval_{\xi}(\zaff) := \brev{} e^{(\xi, \zaff)} \erev{}
%e^{\scp{\xi, \zaff}}
= e^{\xi_1 z_1 + \dots + \xi_n z_n} =\sum_{\alpha\in \NN^n} \xi^{\alpha} \frac{\zaff^{\alpha}} {\alpha!}$ represents the \emph{affine evaluation} map: $\brev{} \apply{\eval_{\xi}}{p} = p(\xi)\erev{}$. %$\scp{\eval_{\xi},p} = p(\xi)$.

%The polynomials $\omega_i (\zaff)$ are called the \emph{weights}, while the points $\xi_{i}$ are called the \emph{frequencies}.
Consequently, a straightforward computation gives
\begin{equation} \label{eq:exprule}
x_k \star \big(\omega(\zaff) \eval_{\xi}(\zaff)\big) = \big(\partial_{z_k} \omega(\zaff) + \xi_{k} \omega(\zaff) \big) \eval_{\xi}(\zaff),
\end{equation}
from which we deduce that the inverse system spanned by $\omega(\zaff) \eval_{\xi}(\zaff)$ is
$$
\invsys{\omega(\zaff) \eval_{\xi}(\zaff)}
= \big\langle g(\partial_{\zaff})\big(\omega(\zaff)\big) \mid g\in R \big\rangle\, \eval_{\xi}(\zaff)
= \invsys{\omega}(\zaff)\, \eval_{\xi}(\zaff).
$$
%where $\vspan{\partial_{\zaff}^\alpha(\omega(\zaff)) \mid \alpha\in \NN^n}$ is the vector space spanned by all derivatives of $\omega(\zaff)$.
It is a finite-dimensional vector space, since $\omega \in \kk[\xaff]$ is a polynomial.
Its dimension $\dim \vspan{g(\partial_{\xaff})(\omega) \mid g\in R}$ is called the \emph{multiplicity} of $\omega$.

We define the degree-$d$ \emph{homogenization} map (with respect to $x_0$) by
$$
\mathfrak{h}_0 : \KKrng_{\le d} \to \Sd, \quad 
p \mapsto \mathfrak{h}_{0}(p) := x_0^{d} p\left(\frac{x_1}{x_0} , \ldots, \frac{x_n}{x_0} \right).
$$
Its transposed map will therefore be
$$
\dual{\mathfrak{h}_{0}}: \dual{\Sd} \to \dual{\KKrng_{\le d}}, \quad \varphi \mapsto \dual{\mathfrak{h}_0}(\varphi) := \varphi \circ \mathfrak{h}_{0}.
$$
For any $f= \sum_{\alpha \in \NN^{n+1},|\alpha|=d} f_{\alpha} \, \xprj^{\alpha}\in \Sd$, we have $f^{*} = \frac{1}{d!} \sum_{\alpha \in \NN^{n+1},|\alpha|=d} f_{\alpha} \zprj^{\alpha} \in   \dual{\Sd}$ and we define the linear functional $\check{f} \in \dual{\KKrng_{\le d}}$ as
\begin{equation} \label{eq:fcheck}
\check{f} := \dual{\mathfrak{h}_0}(f^*) = \frac{1}{d!} \sum_{\alpha \in \NN^{n+1}, |\alpha|=d} \alpha_0! f_{\alpha}\,  z_1^{\alpha_1} \cdots z_n^{\alpha_n} = \frac{1}{d!} \sum_{\beta\in \NN^{n},|\beta|\le d} (d-|\beta|)!\,  f_{\underline{\beta}}\,  \zaff^{\beta},
\end{equation}
where $\underline{\beta}=(d-|\beta|, \beta_1, \ldots, \beta_n)\in \NN^{n+1}$ for $\beta=(\beta_1,\ldots, \beta_n)\in \NN^{n}$.
%By definition, $\forall p\in \KKrng_{\le d},\  \scp{\check{f},p}= \scp{f^*, \mathfrak{h}_0(p)}= \scp{f, \mathfrak{h}_{0}(p)}_{a}$.

%We explicitly verify that for any $f = \sum_{\alpha \in \NN^{n+1},|\alpha|=d} f_{\alpha} \, \xprj^{\alpha}\in \Sd$ and $g = \sum_{\beta \in \NN^{n},|\beta| \leq d} g_{\beta} \xaff^{\beta} \in \kk[\xaff]_{\leq d}$, we have
%$$
%\check{f}(g) = %\frac{1}{d!} \eval_0( g \contract f_{\textnormal{dp}}(1,x_1,\dots,x_n) =
%\frac{1}{d!} \sum_{\beta \in \NN^{n},|\beta| \leq d} (d-|\beta|)!\beta! f_{(d-|\beta|,\beta)}g_{\beta}.
%$$
\begin{remark} \label{rmk:check_vs_dp}
The functional $\check{f}$ of \cref{eq:fcheck} may be interpreted in terms of polynomial contraction $\zaff^{\beta} \contract \xaff^{\alpha} := \xaff^{\alpha-\beta}$. In fact, if we denote the \emph{divided powers} expression of $f$ by $f_{\textnormal{dp}} := \sum_{\alpha \in \NN^{n+1},|\alpha|=d} f_{\alpha} \alpha! \, \xprj^{\alpha}\in \Sd$ \cite{Bernardi2024}, then
\[ \apply{\check{f}}{g} = \frac{1}{d!}\big( g \contract f_{\textnormal{dp}}\big)(1,0, \dots,0) . \]
\brev{}When $f=\omega \ell^{d-k}$, the dehomogenization of $\frac{1}{d!}f_{\textnormal{dp}}$ by $x_0$ agrees with the divided power expression of $\omega^{d, x_0, \xaff}$ \cite[Proposition 2.4]{reigtaufer2026}.\erev{}
Furthermore, by inspecting (\ref{eq:fcheck}), we observe that the coefficient of $\zaff^\beta$ in $\check{f}$ equals the coefficient of $\xaff^\beta$ in $\frac{1}{d!}\textnormal{Tw}^{-1}(f)(1,x_1,\dots,x_n)$, where $\textnormal{Tw}(f)$ is the \emph{twist} of $f$ as defined in \cite{SymPows}.
\end{remark}

The following lemma is the affine version of \Cref{lem:w xi dual}.

\begin{lemma}\label{lem:polexp trunc}
Let $f = \omega\, \ell ^{d-k} \in \Sd$, where $\ell = x_0 + (\xi, \xaff) \in \Sone$ with $\xi\in \kk^{n}$.
Then we have
$$
\check{f} = \big(\omega^{d,\ell,\xaff}(\zaff)\, \eval_{\xi}(\zaff)\big)^{[\le d]}.
$$
\end{lemma}
\begin{proof} Let $\omega = \sum_{j=0}^k\omega_j\, \ell^{k-j} \in \Sk$ with $\omega_j \in \S_j(\xaff)$, so $\omega^{d,\ell, \xaff}(\zaff) =\frac 1 {d!} \sum_{j=0}^{k} (d-j)!\, \omega_j(\zaff)$ as defined in \cref{eq:omega v}.
For $p \in \KKrng_{\le d}$, let us shortly denote $p^h = \hk_{0}(p)$. By \Cref{lem:apol prop}, % and \Cref{eq:dual prop}, %Why do we need that equation?
we have
\begin{align*}
\apply{\check{f}}{p}  & = \apply{f^*}{p^h} 
= \sum_{j=0}^{k} \apply{(\omega_j \ell^{d-j})^*}{p^h}= \frac 1 {d!} \sum_{j=0}^{k} {(d-j)!}\, \omega_j(\partial_{\xprj})(p^h)(1, \xi_1, \dots, \xi_n) \\
 & = \omega^{d,\ell, \xaff}(\partial_{\xprj})(p^h)(1, \xi_1, \dots, \xi_n).
\end{align*}
Since $\omega_j \in \S_j(\xaff)$, then also $\omega^{d,\ell, \xaff} \in \S(\xaff)$, thus
$$
 \omega^{d,\ell, \xaff}(\partial_{\xprj})(p^h)(1, \xi_1, \dots, \xi_n) = \omega^{d,\ell, \xaff}(\partial_{\xaff})(p)(\xi_1, \dots, \xi_n) = \brev{} \apply{\omega^{d,\ell,\xaff}(\zaff)\, \eval_{\xi}(\zaff) }{p}, \erev{}
$$
which shows the equality between the linear functionals $\check{f}$ and $\big(\omega^{d,\ell,\xaff}(\zaff)\, \eval_{\xi}(\zaff)\big)^{[\le d]}$.
\end{proof}

%We denote by $\bigg(\sum_{i=1}^{r} \omega_i(\zprj)\eval_{\xi_i}(\zprj)\bigg)^{[k]}$ the \textit{degree-k} part of the series.

\section{Generalized Additive Decompositions} \label{sec:GAD}
%\todo{intro}
In this section, we define the Generalized Additive Decompositions (GAD) of a form in $\Sd$, and detail how we measure the size of such decompositions.

\subsection{GADs} 
\begin{definition}
    Let $f \in \Sd$ and $0 \leq k_i \leq d$ for any $i \in \{1,\dots,s\}$. 
    A \textit{Generalized Additive Decomposition (GAD)} of \textit{f} is an expression of the form:
    \begin{equation}
        f = \sum_{i=1}^{s}  \omega_i\, \ell_i^{d-k_i},
        \label{eq:gad}
    \end{equation}
    where $\xi_i \in \kk^{n+1}$ and $\ell_i:= (\xi_i,\xprj)$ are not proportional and do not divide $\omega_i \in \S_{k_i}$.
    The linear forms $\{\ell_i\}_{1 \leq i \leq s}$ are called the \emph{supports} of the GAD \eqref{eq:gad}.
    When $s = 1$, we call the GAD \emph{local}.
\end{definition}
Clearly such a decomposition exists for any $f\in \Sd$, since we can write $f=f\, \ell^0$ for any $\ell \in \Sone$ not dividing $f$.
To measure the size of such a decomposition, we use a dual approach and introduce so-called inverse systems. 
%% \begin{remark}
%%     Each summand $\omega_i\, \ell_i^{d-k_i}$ can be interpreted as the truncation of the Taylor expansion of $f$ around the point $\ell_i$ up to order $k_i$. 
%% \end{remark}

%Given $\ell\in \Sone$, which is \emph{non-isotropic} for the apolar product (i.e. $\apol{\ell, \ell}\ne 0$), any basis $\{v_0,\dots,v_n\}$ of $\ell^{\perp}=\set{v\in \Sone\mid \apol{\ell, v}=0}$ is such that $\set{\ell, v_1, \ldots, v_n}$ is a basis of $\mathcal{S}_1$. 
%Notice that by a generic change of coordinates, we can always assume that each element of a finite collection $\ell_1, \ldots, \ell_s\in \Sone$ is non-isotropic.

%$\nub= \set{\nu_1, \ldots, \nu_n}$ be a basis of $\ell^\perp\subset \Sone$.
%The basis $\{v_0, v_1, \dots, v_n\}$ basis is referred to as the \textit{decomposition variables}.
%\todo[inline]{update with new def. Check hereafter}

A GAD $f = \sum_{i=1}^{s} \omega_i \ell_i^{d-k_i}\in \S_d$ can be reinterpreted in the dual space as a decomposition in terms of (truncated) polynomial-exponential series.
In fact, by \Cref{lem:w xi dual} we have 
\begin{equation*}
f^* %= \bigg(\sum_{i=1}^s \omega_i(\xprj)(\xi_i,\underline{\textbf{x}})^{d-k_i} \bigg)^* 
= \bigg(\sum_{i=1}^s \omega_i^{d,\ell_i, \bm v}(\zprj)\eval_{\ell_i}(\zprj)\bigg)^{[d]},
\end{equation*}
for any $\bm v$ such that $\set{\ell_i, v_1, \ldots, v_n}$ is a basis of $\Sone$ for $i=1, \ldots, s$. 
%Let $\xiprj_i= (1, \xi_i)$ with $\xi_i\in \kk^{n}$.
Similarly, by \Cref{lem:polexp trunc}, if $\xiprj_{\brev{}i,0\erev{}}=1$, then we also have
\begin{equation} \label{eq:fcheckastrunc}
    \check{f} = \bigg( \sum_{i=1}^s {\omega}_i^{d,\ell_i,\xaff}(\zaff) \eval_{\xi_i}(\zaff) \bigg)^{[\le d]} \in \kk[\xaff]^*_{\leq{d}}.
\end{equation}
where $\xi_i=(\xiprj_1, \ldots, \xiprj_n)$ and the $\omega_i^{d, \ell_i, \xaff}$ are defined as in \cref{eq:omega v}.
%, $\omega_i = \sum_{j = 0}^{k_i} \omega_{i,j}\, \ell_i^{k_i-j} \in \S_{k_i}$, 
%$\omega_{i,j} \in \S_j$ and
%$\omega_i^{d, \ell_i, \xaff} := \frac{1} {d!} \sum_{j=0}^{k_i} {(d-j)!}\,\omega_{i,j}$.
Hence, $\check{f}$ is the truncation (in degree $\le d$) of the series $\sum_{i=1}^s {\omega}_i^{d, \ell_i, \xaff}(\zaff) \eval_{\xi_i}(\zaff)\in \PolExp$.

\begin{definition} \label{defn:linvsys}
Let $\ell\in \Sone$ and $\bm v= \set{v_1, \ldots, v_n}\subset \Sone$ such that $\set{\ell, v_1, \ldots, v_n}$ is a basis of $\Sone$, and fix $d \ge k$. We define 
%$\ell^{\perp} := \set{v\in \Sone\mid \apol{\ell, v}=0}$, and
the $\ell$-\emph{inverse system} of $\omega \in \Sk$ as
$$
\invsys{\omega}^{d, \ell, \bm v} :=  \big\langle g(\partial_{\bm v})(\omega^{d,\ell, \bm v})  \ |\ g \in \S(\bm v) \big\rangle.
$$
\end{definition}

\begin{remark} 
The $\ell$-inverse system is made of all the possible derivatives of $\omega^{d, \ell, \bm v}$, with respect to any basis $\nub = \set{\nu_1, \ldots, \nu_n}$ of $\vspan{\bm v}$.
In fact, we have
$$
    \invsys{\omega}^{d, \ell, \bm v} = \invsys{\omega^{d, \ell, \bm v}}= \sum_{|\alpha| \leq k} \big\langle  \partial_{\nu_1}^{\alpha_1} \dots \partial_{\nu_n}^{\alpha_n}(\omega^{d,\ell, \bm v}) \big\rangle,
$$
since $\partial_{\nub}^{\alpha}(\omega^{d,\ell, \bm v})=0$ for $|\alpha| >k$, and  $\set{\nub^{\alpha}}_{\alpha \in \NN^n, |\alpha| \le k}$ is a basis of $\oplus_{i=0}^{k}\S_i(\bm v)$. % by \Cref{lem:change var}.
% This construction is canonical, since by linearity, it does not depend on the choice of the basis $\nub$ of $\vspan{\bm v}$.
%any set of generators of $\ell^{\perp}$ (as $\kk$-vector space) is a set of (ideal) generators for $\mathfrak{m}_{\xiprj}$, the maximal ideal corresponding to $\ell = (\xprj,\xiprj)$. By Lemma \ref{lem:change var}, the vector space $\invsys{\omega}_{\ell}$ is made of all derivations of $\omega$ by homogeneous elements in $\mathfrak{m}_{\xiprj}$, rehomogenized in degree $k$ with suitable powers of $\ell$.\todo{Why independent of $\nub$?}
\end{remark}

The vector space $\invsys{\omega}^{d, \ell, \bm v}$ depends, in general, on the choice of $\bm v$.
However, we now prove that the dimension does not depend on the chosen basis: it is a canonical quantity associated with $\omega$ and $\ell$, \brev{}for any given $d$\erev{}.

\begin{proposition} \label{prop:rankwelldefined}
    Let $\omega \in \S_k$ and $\ell, v_1, \dots, v_n, u_1, \dots, u_n \in \S_1$ such that $\vspan{\ell, \bm v} = \vspan{\ell, \bm u} = \S_1$.
    Then, for all integers $d \geq k$, there is an isomorphism of vector spaces
    \[ \invsys{ \omega }^{d, \ell, \bm v} \simeq \invsys{ \omega }^{\brev{}d\erev{}, \ell, \bm u}, \]
    which respects the polynomial degree.
\end{proposition}
\begin{proof}
\begin{comment}
    It follows directly from the definition of $\omega^{d,\ell,\bm v}$ (\cref{eq:omega v}) that
    $$ \invsys{ \omega }^{d, \ell, \bm v} \to \invsys{ \omega }^{k, \ell, \bm v}, \quad \sum_{j = 0}^k s_j \mapsto \sum_{j = 0}^k \frac{d! (k-j)!}{(d-j)! k!} s_j, $$
    where $s_j \in \S_j(\bm v)$ are the graded parts of a given element of $\invsys{ \omega }^{d, \ell, \bm v}$, is an isomorphism of vector spaces that respects the degrees.
    Thus, it is sufficient to prove the graded isomorphism $\invsys{ \omega }^{k, \ell, \bm v} \simeq \invsys{ \omega }^{k, \ell, \bm u}$.
\end{comment}
Since $\vspan{\ell, \bm v} = \vspan{\ell, \bm u} = \S_1$, and $\omega^{d,\ell, \bm u}$ (resp. $\omega^{d,\ell, \bm v}$) depends only on $\ell$ and $\vspan{\bm u}$ (resp. $\vspan{\bm v}$), we can assume that
$ \bm u = \bm v + \bm a\, \ell$ for some $\bm a=(a_1, \ldots, a_n) \in \kk^n$.

We denote by $\partial_{u_i}$ (resp. $\partial_{v_i}$) the differentiation with respect to the variable $u_i$ (resp. $v_i$) in the monomial basis  of $\S$ in the variables $\set{\ell, \bm u}$ (resp. $\set{\ell, \bm v}$).
Let $\Delta_{\bm a} := \sum_{i=1}^n  a_i \partial_{u_i}$. 
By the Taylor expansion at $\ell = 0$, we have
$$
  \omega = \sum_{i=0}^{k}\omega_i(\bm u) \ell^{k-i} 
  %= \sum_{i=0}^{k} \omega_i(\bm u) \ell^{k-i} 
  = \sum_{i=0}^{k} \omega_i(\bm v + \bm a \ell) \ell^{k-i}
  = \sum_{i=0}^{k} \sum_{j=0}^{i} \frac 1 {j!} \Delta_{\bm a}^j (\omega_i) (\bm v) \ell^{j+k-i} = \sum_{i=0}^{k} \sum_{j=0}^{k-i} \frac 1 {j!} \Delta_{\bm a}^j (\omega_{i+j}) (\bm v)  \ell^{k-i},
  %= \sum_{i=0}^k \omega'_{i} (\bm v)  \ell^{k-i}\brev{}, \erev{}
$$  
%where $\omega'_i= \omega_i+\sum_{j=1}^{k-i} \frac 1 {j!} \Delta_{\bm a}^j (\omega_{i+j})\in \S_{i}(\bm v)$.
\brev{}where, for $p \in \kk[\bm u]$, we denote by $p(\bm v) \in \kk[\bm v]$ the polynomial $p$ after the substitution $\bm u \mapsto \bm v$.
By definition (see \cref{eq:omega v}), we deduce that 
\begin{equation} \label{eq:wdluv}
    \omega^{d,\ell, \bm u} = \sum_{m=0}^{k} \frac{(d-m)!}{d!} \omega_m(\bm u),
    \qquad \textnormal{and} \qquad
    \omega^{d,\ell, \bm v} = \sum_{i=0}^{k} \frac{(d-i)!}{d!} \sum_{j=0}^{k-i} \frac{1}{j!} \Delta_{\bm a}^{j}( \omega_{i+j})(\bm v).
\end{equation}
We define the linear operator\erev{}
$$ T_{\bm a}: \brev{}\kk[\bm u] \to \kk[\bm v]\erev{}, \quad 
p \mapsto \sum_{i,j \geq 0} \brev{} {d-i \choose j} \erev{} \Delta_{\bm a}^{j}( p_{i+j})(\bm v)
\brev{}= \sum_{m \geq 0} \sum_{j=0}^{m} {d-m+j \choose j} \Delta_{\bm a}^{j} (p_{m})(\bm v)\erev{},
$$
\brev{}where %${b \choose j} = \frac{1}{j!}\prod_{r=0}^{j-1}(b-r)$, and 
$p_m$ denotes the homogeneous degree-$m$ part of $p \in \kk[\bm u]$.
These sums are finite on polynomials, and the equality follows by summing over $m=i+j$.
In the monomial bases ordered by degree, the restriction $T_{\bm a}:\kk[\bm u]_{\leq l} \to \kk[\bm v]_{\leq l}$ has an upper triangular matrix with identity blocks on the diagonal. Hence it is invertible and preserves polynomial degree.

%\brev{}We also define the base-change operator
%$$\rho: \kk[\bm u] \to \kk[\bm v], \quad u_i \mapsto v_i.$$\erev{}
By \cref{eq:wdluv}, we have
\begin{align*}
T_{\brev{}\bm a\erev{}} ( \omega^{d,\ell,\bm u} ) 
%& = \sum_{m=0}^k  T_{\brev{}\bm a\erev{}} \big( \frac{(d-m)!}{d!}\omega_m(\bm u) \big)= \sum_{m=0}^k \sum_{j=0}^{m} \frac{(d-m)!}{d!} {d-m+j \choose j} \Delta_{\bm a}^{j}(\omega_{m})(\bm v)\\
& = \sum_{i=0}^k \sum_{j=0}^{k-i}  {d-i \choose j} \Delta_{\bm a}^{j}\big(\frac{(d-i-j)!}{d!}\omega_{i+j}\big)(\bm v) 
= \sum_{i=0}^k \frac{(d-i)!}{d!} \sum_{j=0}^{k-i} \frac{1}{j!} \Delta_{\bm a}^{j}(\omega_{i+j})(\bm v) 
= \omega^{d,\ell,\bm v}. \erev{}
%\omega^{\brev{}d\erev{},\ell,\bm v} = \sum_{i=0}^{\brev{}d\erev{}} \frac{(\brev{}d\erev{}-i)!} {\brev{}d\erev{}!} \omega'_i = T_{\brev{}\bm a\erev{}}\left( \sum_{i=0}^{k} \frac{(\brev{}d\erev{}-i)!} {\brev{}d\erev{}!} \omega_i\right) = T_{\brev{}\bm a\erev{}}(\omega^{\brev{}d\erev{},\ell,\bm u}),
\end{align*}

\brev{}%We note that $(1-\Delta_{\bm a})^{-1} = \sum_{i=0}^k \Delta_{\bm a}^i$ over $\kk[\bm v]_{\leq k}$.
For every $1 \leq h \leq n$, we prove that
\begin{equation} \label{eq:Tdelta}
   T_{\bm a} \partial_{u_h} = \partial_{v_h} T_{\bm a} (1+\Delta_{\bm a}).
\end{equation}
In fact, on every homogeneous polynomial $p_m \in \kk[\bf u]$ of degree $m>0$, we have
%\begin{align*} (1-\Delta_{\bm a})^{-1} \partial_{v_h} T_{\bm a} (p_m) &= \sum_{j=0}^{m-1} \sum_{t=0}^j {d-m+t \choose t} \Delta_{\bm a}^{j}\big( \partial_{v_h} p_m(\bm v) \big) \\
%&= \sum_{j=0}^{m-1} {d-m+1+j \choose j} \Delta_{\bm a}^{j}\big( \partial_{v_h} p_m(\bm v) \big)
%= T_{\bm a} \partial_{u_h}(p_m),
%\end{align*}
\begin{align*}\small
\partial_{v_h}T_{\bm a}(1+\Delta_{\bm a})(p_m)
&= \partial_{v_h}\big(\sum_{j=0}^{m-1}\left[\binom{d-m+j}{j}+\binom{d-m+j}{j-1}\right](\Delta_{\bm a}^jp_m)(\bm v)\big)\\
&= \partial_{v_h}\big(\sum_{j=0}^{m-1}\binom{d-m+1+j}{j}\big(\Delta_{\bm a}^j p_m\big)(\bm v)\big)
= \sum_{j=0}^{m-1}\binom{d-m+1+j}{j}\big(\Delta_{\bm a}^j \partial_{u_h}p_m\big)(\bm v)\\
& = T_{\bm a}(\partial_{u_h}(p_m)),
\end{align*}
where the second equality is Pascal's identity and the third is from the relation: $\partial_{v_h}\big(p(\bm v)\big)= \big(\partial_{u_h}(p)\big)(\bm v)$ for any $p\in \kk[\bm u]$.
For $m=0$, \cref{eq:Tdelta} trivially holds.\erev{}

\brev{}
Since $\invsys{\omega}^{d,\ell,\bm u}$ is closed under
differentiation and $\Delta_{\bm a}$ is nilpotent on it, we have
$$ (1+\Delta_{\bm a}) \bigl(\invsys{\omega}^{d,\ell,\bm u}\bigr) = \invsys{\omega}^{d,\ell,\bm u}. $$
Hence, \cref{eq:Tdelta} gives
$$ \partial_{v_h}T_{\bm a} \bigl(\invsys{\omega}^{d,\ell,\bm u}\bigr) = T_{\bm a}\partial_{u_h} \bigl(\invsys{\omega}^{d,\ell,\bm u}\bigr) \subseteq T_{\bm a}\bigl(\invsys{\omega}^{d,\ell,\bm u}\bigr). $$
Therefore, $T_{\bm a}(\invsys{\omega}^{d,\ell,\bm u})$ is closed under differentiation and contains $T_{\bm a}(\omega^{d,\ell,\bm u})=\omega^{d,\ell,\bm v}$, which implies that $ \invsys{\omega}^{d,\ell,\bm v} \subseteq T_{\bm a}\bigl(\invsys{\omega}^{d,\ell,\bm u}\bigr)$.
Since $T_{\bm a}$ is injective, we then have
$$ \dim\invsys{\omega}^{d,\ell,\bm v} \leq \dim\invsys{\omega}^{d,\ell,\bm u}. $$
Exchanging $\bm u$ and $\bm v$ and replacing $\bm a$ with $-\bm a$ gives the reverse dimension inequality.
Hence the inclusion is an equality, and the restriction of $T_{\bm a}$ is the required degree-preserving isomorphism.
\erev{}
\end{proof}

This allows us to define the $\ell$-rank and GAD-rank of a form:
\begin{definition} \label{defn:invsystrankGAD}
Let $\{\ell, v_1, \dots, v_n\}$ be a basis of $\S_1$.
We define the \textit{$\ell$-rank} of \brev{}$ \omega \ell^{d-k} \in \Sd$\erev{} as %at $\ell$ 
$$
\brev{}\rank_{\ell}(\omega \ell^{d-k}) := \dim \invsys{\omega}^{d, \ell, \bm v}.\erev{}
$$
The \textit{rank} (or \emph{size}) of a GAD of $\brev{}\sum_{i=1}^s \omega_i\, \ell_i^{d-k_i}\erev{} \in \Sd$  is additively defined as
\begin{equation*}
    \rank \left( \sum_{i=1}^s \omega_i\, \ell_i^{d-k_i} \right) := \brev{}\sum_{i=1}^s \rank_{\ell_i}( \omega_i \ell_i^{d-k_i} ).\erev{} % = \sum_{i=1}^s \dim \invsys{\omega_i}^{d, \ell_i, \bm v}
\end{equation*}
We say that a GAD of $f \in \Sd$ is \emph{minimal} if there are no GADs of $f$ with a strictly lower rank.
The rank of a minimal GAD of $f$ will be referred to as the \emph{GAD-rank} of $f$, which we denote by $\gadrank(f)$.
\end{definition}

\begin{remark}
    \Cref{defn:invsystrankGAD} is well-given, as the $\ell$-rank does not depend on the choice of $\bf v$ by \Cref{prop:rankwelldefined}, and it can be evaluated with only linear algebra operations.
    \end{remark}

The inverse system $\invsys{\omega}^{d,\ell,\bm v}$ is covariant with changes of coordinates, as we show now.
Given $\phi = (\phi_{i,j})_{0 \leq i,j \leq n} \in \textup{GL}_{n+1}(\kk)$, we denote the corresponding change of coordinate action on $g \in \S$ as 
$\phi(g)= g(\xprj \cdot \phi)= g\left(\sum_{i=0}^n \phi_{i,0} x_i, \ldots, \sum_{i=0}^n \phi_{i,n} x_i\right)$,
while its transpose action is the same action, with respect to its transposed matrix, i.e., $\phi^{T}(g)= g(\xprj \cdot \phi )= g(\sum_{i=0}^n \phi_{0,i} x_i, \ldots, \sum_{i=0}^n \phi_{n, i} x_i)$.

\begin{lemma} \label{lem:basechagederivation}
    Let $f,g \in \S$ and $\phi \in \textnormal{GL}_{n+1}(\kk)$. Then
    $$
    \phi\big(g(\partial_{\xprj})(f)\big)= \big( \phi^{-T}(g)(\partial_{\xprj}) \big) \big(\phi(f)\big).
    $$
\end{lemma}
\begin{proof}
    By linearity of the involved objects, it is sufficient to prove it on monomials, hence we assume $f = \xprj^{\alpha}$ and $g = \xprj^{\beta}$.
    For every $0 \leq i,j \leq n$, we have
    $$
    \big( \phi^{-T}(x_i)(\partial_{\xprj}) \big) \big(\phi(x_j^{\alpha_j})\big) = 
    \big( \phi^{-T}(x_i)(\partial_{\xprj}) \big) \big(\phi(x_j)^{\alpha_j}\big) = \alpha_j \apply{\phi(x_j)^{\alpha_j-1} ( \phi^{-T}(x_i) )^*}{\phi(x_j)}.
    $$
    We verify that $\phi^{-T}(x_i)$ is represented by the $i$-th row of $\phi^{-1}$, while $\phi(x_j)$ is represented by the $j$-th column of $\phi$ in the basis $(x_0, \ldots, x_n)$ of $\Sone$.
    Thus, we have $\apply{( \phi^{-T}(x_i) )^*}{\phi(x_j)} = \delta_{ij}$, so
    $$
    \big( \phi^{-T}(x_i)(\partial_{\xprj}) \big) \big(\phi(x_j^{\alpha_j})\big) = 
    \delta_{ij} \alpha_j \phi(x_j)^{\alpha_j-1} = \phi\big( \partial_{x_i}(x_j^{\alpha_j}) \big),
    $$
    which proves the statement for $f = x_j^{\alpha_j}$ and $g = x_i$.
    Since $\phi, \phi^{-1}$ and $\phi^{T}$ are all multiplicative, by the chain rule of differentiation, we get
    $$
    \big( \phi^{-T}(x_i)(\partial_{\xprj}) \big) \big(\phi(\xprj^{\alpha})\big) =
    \phi\big(\alpha_i x_i^{\alpha_i-1} \prod_{j \neq i} x_j^{\alpha_j}\big) =
    \phi\big( \partial_{x_i}(\xprj^{\alpha}) \big).
    $$
    The conclusion follows by repeatedly applying the above argument for all the (possibly repeated) linear factors $x_i$ of $\xprj^{\beta}$.
\end{proof}

We note that \Cref{lem:basechagederivation} corrects and generalizes \cite[Lemma 3.6]{Bernardi2024}, \brev{}where the same result was stated for monomials but the right-hand side was missing the coordinate transformation $\phi^{-1}$.
We note that this correction preserves the equivalence of vanishing, and hence the apolarity results in \cite{Bernardi2024} are not affected. \erev{}

\begin{corollary} \label{cor:rankefficientcomp}
    Let $\omega \in \S_k$, $d \geq k$, $\ell \in \Sone$ and $\phi \in \textnormal{GL}_{n+1}(\kk)$.
    Then we have %$\phi$ yields an isomorphism of $\kk$-vector spaces
    $$
    \phi\big( \invsys{ \omega }^{d,\ell, \bm v}\big) =
    \invsys{ \phi(\omega) }^{d,\phi(\ell), \phi(\bm v)} \brev{}. \erev{} 
    $$
\end{corollary}
\begin{proof}
\brev{}For every $g\in\kk[\bm v]$, applying \Cref{lem:basechagederivation} in the bases $\{\ell,\bm v\}$ and $\{\phi(\ell),\phi(\bm v)\}$, we obtain\erev{}
$$ \phi\big(g(\partial_{\bm v})(\omega^{d,\ell,\bm v})\big) = \brev{}g(\partial_{\phi(\bm v)})\erev{} \big(\phi(\omega^{d,\ell,\bm v})\big). $$
\brev{}Here $g(\partial_{\phi(\bm v)})$ denotes the operator obtained by replacing each $v_h$ in $g$ with $\partial_{\phi(v_h)}$.\erev{}
By construction, we also have $\phi(\omega^{d,\ell,\bm v}) = \phi(\omega)^{d,\phi(\ell),\phi(\bm v)}$.
Hence, the above equality shows that the space of all derivatives in $\phi(\bm v)$ applied to $\phi(\omega^{d,\ell,\bm v})$ is precisely the image under $\phi$ of the space of all derivatives in $\bm v$ applied to $\omega^{d,\ell,\bm v}$.
\end{proof}

    %we can measure each term $\omega \ell^{d-k}$ in a GAD with the size of (any of) its $\ell$-inverse system, as follows.

\Cref{cor:rankefficientcomp} implies that the computation of the $\ell$-inverse system of $\omega$ can be performed in arbitrary coordinates.  One can compute the linear map $\phi : \{\ell, \bm v\} \mapsto \{x_0, \xaff\}$, and the inverse system as
$$ \invsys{\omega}^{d, \ell, \bm v} = \phi^{-1} \big( \invsys{\phi(\omega)}^{d, x_0, \bm \xaff} \big). $$
This approach is usually more efficient than a straight application of \Cref{defn:linvsys}, as it only involves variable differentiations instead of directional derivatives.
\brev{}If we are only interested in computing the $\ell$-rank, we can even avoid the last base-change ($\phi^{-1}$): we rename variables so that $\{ \ell, x_1,\ldots,x_n \}$ is a basis of $\cl S_1$, and apply \Cref{prop:rankwelldefined} to get \erev{}
$$ 
\rank_{\ell}(\omega \brev{}\ell^{d-k}\erev{}) = \dim \invsys{\omega}^{d, \ell, \xaff}.
$$

\subsection{Geometric interpretation}

We provide here a brief geometric interpretation of GADs in terms of osculating varieties to the \emph{$d$-th Veronese variety}, which is classically defined as the $d$-th Veronese embedding of $\mathbb{P}^n(\kk)$.
Its affine cone is
    \begin{equation*}
        \mathcal{V}_{d} := \{\omega \,\ell ^d \hspace{0.1cm}|\hspace{0.1cm} \omega \in \kk, \ell \in \cl S_1\}.
    \end{equation*}
\begin{remark}
    A famous challenging problem consists of computing a (minimal) \textit{Waring decomposition} of $f \in \S_d$, that is, a decomposition of the form: 
    \begin{equation*}
        f = \sum_{i=1}^r \omega_i \, \ell_i^d,
    \end{equation*}
    with a minimal $r \in \NN$, which is called the \textit{Waring rank} of $f$.
    Geometrically, it corresponds to expressing $f$ as a minimal linear combination of $r$ points on $\mathcal{V}_{d}$. 
\end{remark}
One can consider broader families of decompositions by employing also elements in the tangent variety to the Veronese \cite{Bernardi20201}, whose affine cone is

\begin{equation*}
    \mathcal{T}_{d} := \{\omega\, \ell^{d-1} \hspace{0.1cm}| \hspace{0.1cm}\omega, \ell\in \S_1\}.
\end{equation*}

\begin{remark}\label{rem:tangential}
 A point $\omega\, \ell^{d-1}$ with $\ell= \ell_0 x_0 + \cdots+ \ell_n x_n, \omega= \omega_0 x_0 + \cdots + \omega_n x_n \in \Sone$ lies in the tangent space to $\mathcal{V}_{d}$, since $\sum_{i=0}^{n}\omega_i \partial_{\ell_i}(\ell^d)= d\, (\sum_{i=0}^{n} \omega_i x_i) \ell^{d-1} = d \omega \ell^{d-1}$.
\end{remark}

More generally, one can consider decompositions involving higher-order osculating varieties to the Veronese.
\begin{definition}
    The affine cone of the \textit{k-th osculating variety} to $\mathcal{V}_{d}$ %, denoted $\mathcal{O}^k_{d}$, 
    is defined as
    \begin{equation*}
        \mathcal{O}^k_{d} := \{\omega\, \ell^{d-k} \hspace{0.1cm} | \hspace{0.1cm} \omega \in \Sk, \ell \in \cl S_1 \}.
    \end{equation*}
\end{definition}

\begin{remark}
Let $\set{\ell, v_1, \ldots, v_n}$ be a basis of $\Sone$. The 
Taylor expansion of $f\in \S_d$ at $\ell$ is of the form:
\begin{equation}
    f = \sum_{j=0}^d f_j(\bm v) \ell^{d-j} ,
    \label{taylor}
\end{equation}
where $f_j \in \mathcal{S}_j(\bm v)$ are the \textit{decomposition coefficients}. 
For $k \le d$, \cref{taylor} reads
\begin{equation*}
    f \in \omega \,\ell^{d-k} + (v_1,..., v_n)^{k+1}\brev{}, \erev{}
    \label{omega}
\end{equation*}
where $\omega \in \Sk$ and $(v_1,\dots,v_n)^{k+1}$ is the ideal generated by all degree-$(k+1)$ monomials in $v_1, \ldots, v_n$.
The term $\omega\, \ell^{d-k}$ is therefore the truncation of Taylor's expansion of $f$ at $\ell$ at order $k$, which explains the name of $k^{\textup{th}}$-\emph{osculating} variety. 
% We notice that, as in \Cref{rem:tangential}, the tangential variety to $\cl O^{k-1}_d$ is $\cl O^{k}_{d}$. 
For more details, see e.g. \cite{Bernardi2007, Segre1946}.
\end{remark}

\begin{comment}
\textcolor{red}{\begin{lemma} \label{newattempt} 
Let $0 \le k \leq d$.
\begin{itemize}
    \item If $k=d$ or $n=1$ and $k=d-1$, then $\Okdn = \Sd$, hence
    $$ \textup{sing}(\Okdn) = \emptyset. $$
    \item If $n \geq 2$ and $(d,k) = (2,1)$, then
    $$ \textup{sing}(\mathcal{O}_{2}^1) = \mathcal{O}_{2}^0. $$ 
    \item In all other cases, we have
    $$
    \textup{sing}(\Okdn) = \cl O^{k-1}_{d} \cup \set{\omega \ell_1^{d-k} \ell_2^{d-k} \ | \ \omega \in \cl S_{2k -d}, \ell_1, \ell_2 \in \cl S_1}.
    $$
\end{itemize}
We use the convention that $\cl O^{-1}_d = \emptyset$.
When $\Okdn \neq \Sd$ its affine dimension is ${k+n \choose k} + n$. 
\end{lemma}}
\end{comment}
\begin{proposition} \label{prop:osc dim sing} Let $0 \le k \leq d$.
The singular locus of $\Okdn$ is
\brev{}$$\textup{sing}(\Okdn) = \begin{cases}
\emptyset & \textnormal{if } \Okdn = \S_d, \\
    \cl O^{k-1}_{d}  \cup \cl D_d^k & \textnormal{otherwise}, %k < \frac d 2 \textnormal{ or } d=2, k=1, \\
    %\cl O^{k-1}_{d}\cup \set{\omega \ell_1^{d-k} \ell_2^{d-k} \ | \ \omega \in \cl S_{2k -d}, \ell_1, \ell_2 \in \cl S_1, \vspan{\ell_1} \neq \vspan{\ell_2}} & \textnormal{otherwise},
\end{cases}$$
where
$$ 
\cl D_d^k = \begin{cases} \emptyset & \textnormal{when } k<\frac{d}{2}  \textnormal{ or } (d,k)=(2,1), \\ 
\set{\omega \ell_1^{d-k} \ell_2^{d-k} \ | \ \omega \in \cl S_{2k -d}, \, \ell_1, \ell_2 \in \cl S_1} & \textnormal{otherwise},
\end{cases}
$$
%$\textup{sing}(\Okdn)= \cl O^{k-1}_{d}$ (with the convention that $\cl O^{-1}_d = \emptyset$) \brev{} if $k < \frac d 2$, and
%$\textup{sing}(\Okdn)= \cl O^{k-1}_{d}\cup \set{\omega \ell_1^{d-k} \ell_2^{d-k} : \omega \in \cl S_{2k -d}, \ell_1, \ell_2 \in \cl S_1, \vspan{\ell_1} \neq \vspan{\ell_2}}$ otherwise. 
with the convention that $\cl O^{-1}_d = \{0\}$. %\emptyset 
%The affine dimension of $\Okdn$ is ${k+n \choose k} + n$. 
When $\Okdn \subsetneq \Sd$, its affine dimension is ${k+n \choose k} + n$. \erev{}
\end{proposition}
\begin{proof} \brev{}
%% Exclude the whole Sd
If $\Okdn = \Sd$, which occurs precisely when $k=d$, or when $n=1$ and $k=d-1$, we have $\textup{sing}(\Okdn) = \emptyset$, since $\Okdn = \Sd$ is a vector space.

%% Define our projective \pi
Henceforth we assume that $\Okdn \subsetneq \Sd$, therefore $d-k>0$, and define
%$ X := \PP(\cl S_k) \times \PP(\cl S_1)$, $Y := \PP(\Okdn)$, and
$$ X := \PP(\cl S_k) \times \PP(\cl S_1), \quad Y := \PP(\Okdn), \qquad \pi : X \to Y, \quad ([\omega],[\ell]) \mapsto [\omega \ell^{d-k}]. $$
The map $\pi$ is a well-defined surjective and regular morphism of projective varieties, since its coordinates are bihomogeneous polynomials of bidegree $(1,d-k)$ and $\omega \ell^{d-k} \neq 0$ on $X$.
%% Smoothness of X, irreducibility of both
The variety $X$ is smooth and irreducible, being a product of projective spaces, and $Y$ is irreducible as its continuous image.
%% Finite fibers
The fiber of $[f] \in Y$ corresponds precisely to linear factors $[\ell] \in \PP(\cl S_1)$ such that $\ell^{d-k} | f$, hence they are finite.
%% Finite fibers & surjective => compute dimension
Since $\pi$ is surjective with finite fibers, \cite[Corollary~11.13]{Harris} yields
$$ \dim Y = \dim X = {n+k \choose k} + n - 1. $$
Hence $\dim \Okdn = \dim Y+1 = {n+k \choose k} + n$.
%% Finite fibers & regular => \pi is finite
Moreover, since $\pi$ is a regular morphism of projective varieties with finite fibers, it is finite by \cite[Lemma 14.8]{Harris}.

%% Matrix cases
If $d=2$ then $k \in \set{0,1}$, and $\cl O_2^k$ is the variety of rank $\le k+1$ quadrics.
Since this variety is proper in $\cl S_2$, we have $k+1 < n+1$.
Its singular locus is the variety $\cl O_2^{k-1}$ of quadrics of rank $\le k$ by \cite[Corollary 2.4]{Kanev1999}, which is $\cl O_2^{k-1}$, with the convention that $\cl O_2^{-1} = \{0\}$.

%% Generically 1-to-1 => birational
We now assume that $d>2$ and prove that the map $\pi$ is birational.
Since $\kk$ has characteristic $0$, this is equivalent to being generically $1$-to-$1$ \cite[Proposition~7.16]{Harris}, which we verify in all possible cases.
A general form $f = \omega \ell^{d-k}$ has a unique such factorization, up to scalars $\lambda \neq 0$: $\omega\, \ell^{d-k} = (\omega \lambda^{- d+k}) \,(\lambda \ell)^{d-k}$.
Suppose that $f= \omega \ell^{d-k} = \omega' (\ell')^{d-k}$, with $\omega'\in \S_k$, $\ell'\in \S_1$ such that $[\ell']\neq [\ell]$. Then
$(\ell')^{d-k}$ divides $\omega\in \S_k$.
If $d-k\ge 2$ or $d-k=1$ and $n>1$, a general $\omega \in \S_k$ is not divisible by $(\ell')^{d-k}$, thus the fiber of $f$ is the single point $([\omega], [\ell])$.
%If $n = 1$, $\Okdn \neq \Sd$ implies that $d-k \geq 2$, and a general $\omega$ is squarefree and not divisible by $\ell$.
%When $n \geq 2$, if $k \geq 2$ then a generic $\omega \in \cl S_k$ is irreducible, while if $k \leq 1$ then $d-k \geq 2$, hence a general $[\omega \ell^{d-k}]$ uniquely determines $[\omega]$ and $[\ell]$.

%% Apply Lemma 3.2
Since $\pi$ is a birational finite surjective morphism between irreducible projective varieties, and $X$ is smooth, \cite[Lemma~3.2]{SecantsBitangentsAndTheirCongruences} or \cite[Theorem 14.9]{Harris} implies that the smooth points $y \in Y$ are  those with $\pi^{-1}(y) = \{x\}$ and $d_x \pi$ injective.
Therefore, $\textup{sing}(Y)$ consists of points where one of these conditions does not hold.

%% Check the differntial
Fix $f = \omega \ell^{d-k} \neq 0$.
%and choose $\ell_0 \in \cl S_1$ such that $\apply{\ell_0^*}{\ell} = 1$.
By identifying $T_{[f]} \PP( \cl S_d ) \simeq S_d / \kk f$, the image of the projective differential is
$$ \im d\pi_{([\omega], [\ell])} =  \vspan{  \theta\, \ell^{d-k} + \omega \,\mu \, \ell^{d-k-1} \ | \ \theta\in \cl S_{k}, \mu \in \S_1
%\ell_0^\perp
} / \kk f. $$
Its dimension is ${k+n \choose k} + n - \dim(W)$, where 
$$
W = \vspan{ \theta \ell^{d-k}  : \theta \in \cl S_{k}} \cap \vspan{ \omega \mu \ell^{d-k-1} : \mu \in \S_1}.
$$
Elements of $W$ satisfy $\theta \ell^{d-k} = \omega \mu \ell^{d-k-1}$, i.e.,
$\theta \ell = \omega \mu$.
If $\ell \nmid \omega$, then $\mu \in \vspan{\ell}$, hence $W = \vspan{\omega \ell^{d-k}}= \vspan{f}$ and $\dim (W)=1$.
If $\omega = \omega' \ell$, every $\mu \in \S_1$ gives $\theta = \omega' \mu$, so $\dim W = n+1$.
Thus $d\pi_{([\omega],[\ell])}$ is injective if and only if $\ell \nmid \omega$, since $\dim X = {n+k \choose k} + n - 1$.
Consequently, the image of the locus where $d\pi$ is not injective is precisely $\PP(\cl O_d^{k-1})$, interpreted as empty when $k=0$.

%% Check multiple points
The fiber $\pi^{-1}([f])$ contains at least two points when there are non-proportional $\ell_1, \ell_2 \in \cl S_1$ such that $\ell_1^{d-k} | f$ and $\ell_2^{d-k} | f$. 
If $k < \frac{d}{2}$, this cannot happen, while for $k \geq \frac{d}{2}$ this gives the singular points
$$ \set{ [\omega \ell_1^{d-k} \ell_2^{d-k}] \ | \ \omega \in \cl S_{2k -d}, \, \ell_1, \ell_2 \in \cl S_1, \, [\ell_1] \neq [\ell_2] }. $$
We note that, when $\ell_2 = \lambda \ell_1$, the above points are represented by points in $\cl O_d^{k-1}$, since
$$ \omega \ell_1^{d-k} \ell_2^{d-k} = ( \lambda^{d-k} \omega \ell_1^{d-k-1} ) \ell_1^{d-(k-1)} \in \cl O_d^{k-1}. $$
Thus, we conclude that $\textup{sing}(Y) = \PP(\cl O_d^{k-1} \cup \cl D_d^k)$.

%% Projective -> affine, and check 0
Except for its vertex, a point on a cone is singular if and only if its projective class on the cone is.
Therefore, the above argument determines all the non-zero singular forms in $\cl O_d^{k}$.
Finally, the vertex of the cone is singular precisely when $\Okdn \neq \Sd$, otherwise it is included in $\cl O_d^{k-1}$, completing the classification of $\textup{sing}(\Okdn)$. \erev{} 
\end{proof}

\brev{}The dimension of $\Okdn$ computed above agrees with the (projectivized) dimension computed in \cite[Lemma 3.3]{Bernardi2007}.
\Cref{prop:osc dim sing} also shows that for $k< \frac d 2$\erev{}, if $\ell$ does not divide $\omega$, then $\omega\, \ell^{d-k} \in (\Okdn)^{smooth}$ is a smooth point of $\Okdn$.
Hence, GADs parametrize generic points of joint varieties of osculating varieties to $\mathcal{V}_{d}$. In particular, \brev{}when $k_i< \frac d 2$, 
           $f = \sum_{i=1}^{s} \omega_i\, \ell_i^{d-k_i} $ is a GAD as in \cref{eq:gad} implies that $f \in \sum_{i=1}^{s} (\mathcal{O}^{k_i}_{d})^{smooth}$.
\erev{}

\begin{remark}
Waring and tangential decompositions are special instances of GADs (for $k_i = 0$ and $k_i \leq 1$, respectively), thus the GAD-rank is always bounded from above by the Waring (resp. tangential) rank.
However, minimal GADs typically differ completely from minimal Waring (or tangential) decompositions.
\end{remark}

\section{Algebras and operators associated with GADs} \label{sec:algop-GAD}
In this section, we recall certain properties of Artinian algebras, describe those associated with a GAD, and prove novel results that will be employed in the next sections.
\brev{}Throughout this section, we denote the abstract operators with bold symbols (e.g., $\bm H, \bm M$), while their matrix representations with respect to specified bases are denoted by the corresponding italic symbols (e.g., $H, M$).\erev{}

\subsection{Properties of an Artinian algebra and its dual} 

Let $I \subset R=\kk[\xaff]$ be a non-unit ideal, and $\cl A := R/I$ be the associated quotient algebra.
In this setting, the following are known to be equivalent:
\begin{itemize}
    \item $I$ is \emph{zero-dimensional}, i.e. the Krull dimension of $\cl A$ is $0$.
    \item $\cl A$ is \emph{Artinian}, i.e. $\cl A$ is a finite-dimensional $\kk$-vector space.
    \item $I$ defines a \emph{scheme of points}, i.e. its vanishing set $\V(I)$ is finite.
\end{itemize}

Artinian algebras admit a unique direct decomposition into local Artinian algebras, as follows.

\begin{theorem}[{\cite[Theorem 4.9]{EM07}}]\label{thm:artinian}
    Let $I \subset \kk[\xaff]$ be a zero-dimensional ideal, whose minimal primary decomposition is $I = I_{\xi_1} \cap \dots \cap I_{\xi_s}$, where $I_{\xi_i}$ is $\mfm_{\xi_i}$-primary and $\xi_i \in \brev{}\kk\erev{}^n${\brev{}, and $\mfm_{\xi_i}$ is the maximal ideal of the point $\xi_i$\erev{}}. Then
    $$
    \cl A = \cl A_1 \oplus \dots \oplus \cl A_s,
    $$
    where $\cl A_i \simeq \kk[\xaff] / I_{\xi_i}$ and $\cl A_i \cdot \cl A_j = 0$ for every $i \neq j$.
\end{theorem}

For every prime $\mfm$ of $\cl A$, the kernel of the localization map $\cl A \to \cl A_{\mfm}$ is the intersection of all $\mfm$-primary ideals of $\cl A$ \cite[Corollary 10.21]{Atiyah2018}.
Hence, every $\cl A_i$ is isomorphic to the localization of $\cl A$ at $\mfm_{\xi_i} \cl A$.
The $\kk$-dimension of the local algebra $\cl A_i$ is referred to as the \emph{multiplicity} of the point $\xi_i$.

We define the \emph{idempotent} $\ub_i$ associated with $\xi_i$ as the projection of $1 \in \cl A$ onto $\cl A_i$.
These elements are then characterized by the following properties, for every $1 \leq i \neq j \leq s$:
\begin{equation}\label{eq:idem}
1 = \ub_1 + \dots + \ub_s, \qquad \ub_i^2 = \ub_i, \qquad \ub_i \ub_j = 0.
\end{equation}

The orthogonal space of an $\mfm_{\xi}$-primary ideal $I_{\xi}$ has a convenient representation in terms of the vector space 
$$
D_{I_{\xi}}(\zaff) := \vspan{ \omega(\zaff) \in \kk[\zaff] \mid \forall q \in I_{\xi}, \, \omega(\partial_{\xaff})(q)(\xi) = 0 },
$$
which is called the \emph{inverse system of} $I_{\xi}$.

\begin{proposition}[{\cite[Proposition 2.16]{Mourrain2017}}]\label{prop:3.2}
Let $I_{\xi} \subset \kk[\xaff]$ be an $\mfm_{\xi}$-primary ideal. Then
$$
I_{\xi}^{\perp} \simeq \dual{(\kk[\xaff]/I_{\xi})} = D_{I_{\xi}}(\zaff) \eval_{\xi}(\zaff).
$$
\end{proposition}

The above local property implies a global decomposition property on the elements of $\dual{\cl A}$, as given by the following result.

\begin{theorem}[{\cite[Theorem 2.17]{Mourrain2017}}]\label{thm:struct dual}
    Let $I \subset \kk[\xaff]$ be a zero-dimensional ideal, whose minimal primary decomposition is $I = I_{\xi_1} \cap \dots \cap I_{\xi_s}$, where $I_{\xi_i}$ is $\mfm_{\xi_i}$-primary and $\xi_i \in \brev{}\kk\erev{}^n$.
    Then
    $$
    \dual{\cl A} = \left\{ \sum_{i = 1}^s \omega_i(\zaff) \eval_{\xi_i}(\zaff) \in \PolExp \mid \forall 1 \leq i \leq s, \, \omega_i(\zaff) \in D_{I_{\xi_i}}(\zaff) \right\}.
    $$
\end{theorem}

%Finally, we define the (dual) multiplication operators.

\subsection{Hankel operators}

\brev{}In the definitions hereafter, we consider a general $\kk$-algebra $\cl R$, that could be either the symmetric algebra $\cl S$, the ring of polynomials $R$ or the quotient algebra $\cl A=R/I$ for some ideal $I\subset R$.\erev{}

\begin{definition} \label{defn:MultOp}
    Let $\cl R$ be a $\kk$-algebra and $q \in \cl R$.
    We define the \emph{multiplication-by-$q$ operator} in  $\cl R$ as
    $$
    \brev{}\bm M_q\erev{} : \cl R \to \cl R, \quad p \mapsto pq.
    $$
    Its \emph{transposed multiplication operator} is
    $$
    \brev{}\bm M^t_q\erev{} : \cl R^* \to \cl R^*, \quad \varphi \mapsto q \star \varphi.
    $$
\end{definition}

\begin{definition} \label{defn:Hankel}
    Let $\cl R$ be a $\kk$-algebra.
    The \emph{Hankel operator} of a dual element $\varphi \in \dual{\cl R}$ is
    $$ \bm H_{\varphi} : \cl R \to \dual{\cl R}, \quad p \mapsto p \star \varphi. $$
    If the algebra $\cl R = \oplus_{i=0}^{\infty} \cl R_i$ is graded, for $\varphi \in \dual{\cl R_d}$ and $c \leq d$, we define the \emph{Hankel operator in degree} $(d-c,c)$ of $\varphi$ as
    $$ 
    \bm H^{d-c, c}_{\varphi} : \cl R_c \to \dual{\cl R_{d-c}}, \quad p \mapsto p \star \varphi.
    $$
    Similarly, if $\varphi \in \dual{\cl R_{\leq d}}$, we define its \emph{Hankel operator in degree} $(d-c,c)$ as
    $$ 
    \bm H^{d-c, c}_{\varphi} : \cl R_{\leq c} \to \dual{\cl R_{\leq d-c}}, \quad p \mapsto p \star \varphi.
    $$
    In both cases, for any vector spaces $V \subseteq \cl R_c$ and $W \subseteq \cl R_{d-c}$ \brev{}(resp. $V \subseteq \cl R_{\leq c}$ and $W \subseteq \cl R_{\leq d-c}$)\erev{}, we define the \emph{truncated Hankel operator} $\bm H_{\varphi}^{W,V}$ of $\varphi$ on $(V,W)$ as
    $$
        \bm H_{\varphi}^{W,V}:V \to \dual{W}, \quad p \mapsto (p\star \varphi)^{[W]},
    $$
%    obtained by restricting the domain of $\bm H_{\varphi}^{d-c,c}$ to $V$, and 
\brev{}where $(p\star \varphi)^{[W]}$ is the linear functional $p\star \varphi \in \dual{\cl R}$ with domain restricted to $W$.\erev{}
\end{definition}

If $B = \{v_1, \dots, v_r\}$ is a basis of $V$, and $B' = \{w_1, \dots, w_s\}$ is a basis of $W$, the matrix representing $\bm H_{\varphi}^{W,V}$ in $B$ and the dual basis of $B'$ is 
$$
H_{\varphi}^{B',B} := \big( \apply{\varphi}{w_i v_j}  \big)_{\substack{1 \leq i \leq s \\ 1 \leq j \leq r}}.
$$
By convention, we denote $\bm H_{\varphi}^{{B'},{B}} := \bm H_{\varphi}^{\vspan{B'},\vspan{B}}$.

In particular, in our setting, any $f \in \Sd$ defines for every $c \leq d$ \brev{} \emph{a Hankel}\erev{} operator $\bm H_{f^*}^{d-c,c} : \S_c \to \dual{\S_{d-c}}$ and \brev{}a Hankel\erev{} operator $\bm H_{\check{f}}^{d-c,c}: R_{\le c} \to \dual{R_{\le d-c}}$.
By \cref{eq:fcheck}, the matrix of the latter operator in the bases $\{ \xaff^{\alpha} \}_{|\alpha| \le c} \subset R_{\le c}$ and $\{\frac{\zaff^{\alpha}}{\alpha !}\}_{|\alpha| \leq d-c} \subset \dual{R_{\leq d-c}}$ is
$$
H_{\check{f}}^{d-c, c}:= \big( \apply{f^*}{\mathfrak{h}_{0}(\xaff^{\alpha+\beta})} \big)_{|\beta|\le d-c,|\alpha|\le c}.
$$
Moreover, for every $\varphi \in \dual{R}$, the Hankel operator $\bm H_{\varphi}$ is symmetric, namely, for every $p,q \in R$ we have $\bm H_{\varphi}(p)(q)= \bm H_{\varphi}(q)(p) = \varphi(p\,q)$.
One easily verifies that its kernel 
\begin{equation} \label{eq:Iphi}
I_{\varphi} := \ker \bm H_{\varphi} = \{ p\in R \mid p\star \varphi = 0\} = \ann(\varphi)
\end{equation}
is an ideal, hence it defines a quotient algebra
\begin{equation} \label{eq:Aphi}
    \cl A_{\varphi} := R/I_{\varphi}.
\end{equation}
For every $V\subset R_{\le c}$ and $W\subset R_{\le d-c}$ and $\varphi \in \dual{R}$, we also have by definition that
\begin{equation}\label{eq:restrict}
\bm H_{\varphi}^{W,V} = \bm H_{\varphi^{[\le d]}}^{W,V}\brev{}. \erev{}
\end{equation}

Hereafter, we will use the generalized Kronecker theorem: 
\begin{theorem}[{\cite[Theorem 3.1]{Mourrain2017}}]\label{thm:kronecker}
Let $\varphi \in \kk[[\zaff]] \simeq R^*$. Then
$$\rank \bm H_{\varphi}<\infty \quad \iff \quad \varphi \in \PolExp.$$
Moreover, if $\varphi = \sum_{i=1}^{s} \omega_i(\zaff) \eval_{\xi_i} (\zaff)\in \PolExp$, for distinct $\xi_i \in \kk^n$, then 
$\dim \cl A_{\varphi}=\rank \bm H_{\varphi} = \sum_{i=1}^{s} \dim \invsys{\omega_i}$ and 
$$
I_{\varphi} = \cap_{i = 1}^s \vspan{p \in R \mid \forall \theta \in \invsys{\omega_i}, \theta(\partial_{\xaff})(p)(\xi_i)=0}.
$$
\end{theorem}
As a consequence, we have the following lemma.
\begin{lemma} \label{lem:inclusion}
For any $\varphi, \varphi' \in \PolExp$, if $\rank \bm H_\varphi \le \rank \bm H_{\varphi'}$ and $I_{\varphi} \subseteq I_{\varphi'}$, then 
$$
I_{\varphi}=I_{\varphi'} \quad \textnormal{and} \quad \cl A_{\varphi}=\cl A_{\varphi'}.
$$
\end{lemma}
\begin{proof}
Since $I_{\varphi} \subseteq I_{\varphi'}$, we have a canonical \brev{}surjection $\cl A_{\varphi} \twoheadrightarrow \cl A_{\varphi'}$\erev{}, therefore $\dim \cl A_{\varphi'} \le \dim \cl A_{\varphi}$.  
However, by \Cref{thm:kronecker}, we have $\dim \cl A_{\varphi} = \rank \bm H_{\varphi} \le \rank \bm H_{\varphi'} = \dim \cl A_{\varphi'}$.
Thus $\dim \cl A_{\varphi'} = \dim \cl A_{\varphi} < \infty$, hence $\cl A_{\varphi}=\cl A_{\varphi'}$ and $I_{\varphi}= I_{\varphi'}$.
\end{proof}

Finally, we notice that the (transposed) multiplication operator can be read from certain Hankel operators, as follows.
\begin{lemma}[{\cite[Lemma 3.7]{Mourrain2017}}]\label{lem:hankel mult}
For any $\varphi \in \dual{R}$ and $g\in R$, we have
\[
    \bm H_{g \star \varphi} = \bm H_{\varphi} \circ \bm M_{g} = \bm M_{g}^{t} \circ \bm H_{\varphi} \in \brev{}\Hom_{\kk}(R,\dual{R})\erev{}.
\]
%\todo{$\bm M_q$ is not defined as an element of $Hom_k(R,R^*)$? change definition 4.4 and add a comment for $R/I$?}
\end{lemma}
\brev{}We note that, for $\varphi \in I_{\varphi}^\perp \simeq \dual{\cl A_{\varphi}}$, 
the equality also holds on $\Hom_{\kk}(\cl A_{\varphi}, \dual{\cl A_{\varphi}})$, by taking the quotient algebra $\cl A_{\varphi} = R/I_{\varphi}$, where $I_{\varphi} =\ker  \bm H_{\varphi} \subseteq \ker \bm H_{g \star \varphi}$.\erev{}

\subsection{Linear functionals associated with a GAD}

For every $i \in \{1, \dots, s\}$, let us consider the linear form $\ell_i = (\xiprj_i, \xprj)\in \cl S_1$ associated with the point $\xiprj_i \in \kk^{n+1}$.
{\brev{}Since\erev{}} $\kk$ is infinite, up to a generic change of variables, we can assume $\xiprj_{i,0} \neq 0$.
Up to moving these scalars to the corresponding $\omega_i$ in a GAD as in \cref{eq:gad}, we can further assume that $\xiprj_i = (1, \xi_i)$ for $\xi_i \in \kk^n$.
To simplify exposition, we will assume this is always the case.

\begin{definition} \label{defn:evincedbyGAD}
We define the linear functional \emph{associated with the GAD} $\sum_{i=1}^{s} \omega_{i}\, \ell_i^{d-k_i} \in \S_d$ as
\begin{equation} \label{eq:extendedphi}
    \varphi := \sum_{i=1}^{s} {\omega}_i^{d,\ell_i,\xaff}(\bm z) \eval_{\xi_i}(\bm z) \in \kk[[\bm z]] \simeq \dual{R},
\end{equation}
where the ${\omega}_i^{d,\ell_i,\xaff}$ are defined as in \cref{eq:omega v}.
Its \emph{associated ideal} is the \emph{annihilator} of $\varphi$, namely $I_{\varphi} := \ann (\varphi) = \set{p\in R : p\star \varphi = 0 }$, while its \emph{associated quotient algebra} is $\cl A_{\varphi} := R/I_{\varphi}$.
\end{definition}

\begin{remark} \label{rmk:Ixiperp}
    Let $I_{\xi}$ be the ideal associated with $\omega\, \ell^{d-k}$ with $\ell = x_0 + \xi_1 x_1 + \cdots + \xi_n \brev{}x_n\erev{}$.
    By \Cref{prop:3.2} and \Cref{thm:kronecker}, we have
    \[ I_{\xi}^{\perp} %= D(I_{\xi}) \eval_{\xi}(\zaff) 
    = \invsys{{\omega}^{d,\ell,\xaff}(\zaff)} \eval_{\xi}(\zaff). \]
    In particular, $\cl A_{\xi} = R/I_{\xi}$ is a \emph{Gorenstein} algebra, namely its dual is a principal module generated by ${\omega}^{d,\ell,\xaff} e_{\xi}$, and we have
    $$
    \dim_{\kk} (I_{\xi}^{\perp}) = \rank_{\ell}(\brev{}\omega \ell^{d-k}\erev{}).
    $$
\end{remark}

The following theorem shows that, under suitable conditions on a basis of $\cl A_{\varphi}$, the ideal $I_{\varphi}$ is generated by the kernel of \brev{} \emph{a Hankel}\erev{} operator constructed from $\check{f}$.
This is the crucial idea for recovering the GAD of $f$ associated with $I_{\varphi}$, since $\check{f}$ is computed directly from any expression of $f$, regardless of its GADs.
For a given set $B \subseteq R$, we denote $B^{+} := B \cup x_{1} B \cup \cdots \cup x_{n} B$.

\begin{theorem}\label{thm:hankel tnf}
\brev We consider a GAD as in \cref{eq:gad} and the associated functional $\varphi$ as in \cref{eq:extendedphi}. \erev
Let $B$ and $B'$ be bases of $\cl A_{\varphi}$, with $1 \in B$.
Let $c \leq d$ and $A \subset R_{\le c}, A'\subset R_{\le d-c}$ be sets of linearly independent polynomials %\todo{I do not understand why A,A' are not simply taken as vector spaces, instead of their bases. Where are the generators needed?}
such that $B^{+} \subset \vspan{A} \subseteq R_{\le c}$ and $B' \subset \vspan{A'} \subseteq R_{\le d-c}$. 
%Let $A=\{a_{1}, \ldots, a_{s}\}\subset R_{\le c}, A'=\{a'_{1}, \ldots, a'_{t}\}\subset R_{\le d-c}$,  ${B \subset A}, {B'\subset A'}$ be sets of linearly independent polynomials such that $1\in B$ and $B^{+} = B \cup x_{1} B \cup \cdots \cup x_{n} B \subset \vspan{A}$. 
Let $H := H_{\check{f}}^{{A'},{A}}$ be the matrix of the restricted Hankel operator $\bm H := \bm H_{\check{f}}^{A',A}$ with respect to $A,A'$.
Then 
%\todo[inline]{check Jarek vs (2)}
\begin{enumerate}
%\item $I_{\varphi} := \{ p\in R \textup{ s.t. } (\theta(\partial_1, \ldots, \partial_n)p)(\xi_{i})=0 \textup{ for } \theta(\zb)\in \invsys{\check{\omega}_{i}})\}$
%\item $\cl A_{\varphi} = \cl A_{1} \oplus \cdots \oplus \cl A_{s}$ with $\dual{\cl A_{i}}= \invsys{\check{\omega}_{i}(\bm z)} \eval_{\xi_{i}}(\bm z)$
\item\label{thm1} $H_{0} := H_{\check{f}}^{B',B}$ is {invertible}.

\item\label{thm2} $\rank H = \dim \cl A_{\varphi} = \sum_{i=1}^{s} \dim \invsys{{\omega_i}}^{d, \ell_i, \xaff}$.

\item\label{thm3} For every $1 \leq i \leq n$, let 
$H_{i} := H_{\check{f}}^{{B', x_{i} B}}$.
Then
$M_{x_i} = H_{0}^{-1} H_{i}$
is the matrix of the \brev{}multiplication-by-$x_i$ operator on $\cl A_{\varphi}$, in the basis $B$\erev{}.

\item\label{thm4} $\ker \bm H = I_{\varphi}\cap \vspan{A}$,
$(\ker \bm H) = I_{\varphi}$ and
$\textup{im}\, \bm H = I_{\varphi}^{\perp}|_{\vspan{A'}} := \{ \phi|_{\vspan{A'}} \mid \phi \in I_{\varphi}^{\perp} \simeq \cl A_{\varphi}^* \}$.
\end{enumerate}
\end{theorem}
\begin{proof} (\ref{thm1}-\ref{thm2}):
Let $r := \dim \cl A_{\varphi}$.
By \Cref{thm:kronecker}, we have
$$
    r = \rank \bm H_{\varphi} = \sum_{i=1}^{s} \dim \invsys{{\omega_i}^{d, \ell_i, \xaff}},
$$
while by \cref{eq:restrict} and \Cref{lem:w xi dual} we obtain
$$
H_{\varphi}^{A',A} = H_{\check{f}}^{A',A} = H,
$$
hence $\rank H \le \rank H_{\varphi}=r$.
Since $B$ and $B'$ are bases of $\cl A_{\varphi}$, we have the vector space additive decompositions $R = \vspan{B} \oplus I_{\varphi} = \vspan{B'} \oplus I_\varphi$.
As $I_{\varphi}=\ker \bm H_{\varphi}$ and $\bm H_{\varphi}$ is symmetric, we deduce from \cref{eq:restrict}  and \Cref{lem:w xi dual} that 
$$
r = \rank \bm H_{\varphi} = \rank \bm H_{\varphi}^{R,\vspan{B'}} 
 = \rank \bm H_{\varphi}^{\vspan{B'},R} 
 = \rank \bm H_{\varphi}^{{B'},{B}} 
 = \rank H_{\check{f}}^{B',B}.
$$
The size of $B$ and $B'$, bases of $\cl A_{\varphi}$, is $r$.
This  implies that $H_0 = H_{\check{f}}^{B',B}$ is invertible and that $\rank H \ge \rank H_{0} = r$, thus we conclude $\rank H = r$.

(\ref{thm3}): By \Cref{lem:hankel mult}, we have
$$
\bm H_{x_i \star \varphi} = \bm H_{\varphi} \circ \bm M_{x_{i}}.
$$
We choose the basis $B$ of $\cl A_{\varphi}$ and the dual basis of $B'$ in $\dual{\cl A_{\varphi}}$.
Since we assumed $B^+\subset \vspan{A}$, by \cref{eq:restrict} and \Cref{lem:w xi dual}, we obtain the following matrix relations:
$$
H_i
= H^{B',x_i B}_{\check{f}}
= H^{B',x_i B}_{\varphi} 
= H^{B',B}_{x_i \star \varphi} 
= H^{B',B}_{\varphi} M_{x_i} 
= H^{B',B}_{\check{f}} M_{x_i}  = H_0 \, M_{x_i}.
$$
The conclusion follows by inverting $H_0$, which is possible by part (\ref{thm1}).

(\ref{thm4}): We have 
$$
K := \ker \bm H_{\check{f}}^{{A'},{A}} 
= \ker \bm H_{\varphi}^{{A'},{A}} 
= \ker \bm H_{\varphi}^{R,\vspan{A}} = I_{\varphi} \cap \vspan{A},
$$
and consequently $(K) \subseteq I_{\varphi}$.
We now prove the opposite inclusion.
Let $B=\{b_1, \ldots, b_r\}$ with $b_1=1$.
We define $B^{\vspan{\le e}} := \{\sum_{j=1}^{r} q_j b_j \mid q_j \in R_{\le e}\}$.
Since $1 \in B$, then $R = \bigcup_{e=0}^{\infty} B^{\vspan{\le e}}$.
We prove by induction on $e \in \NN$ that $B^{\vspan{\le e}}\cap I_{\varphi} \subset (K)$.
The base step $e=0$ holds, since $B^{\vspan{\le 0}} = \vspan{B}$ and $\vspan{B}\cap I_{\varphi}=\{0\} \subseteq (K)$. 
Now we assume inductively that it holds for a given $e \geq 0$ and let $p\in B^{\vspan{\le e+1}}\cap I_{\varphi}$.
We can write $p$ as
$$
p = \sum_{i=0}^{n} x_i \sum_{j=1}^r q_{i,j} b_j 
= \sum_{j=1}^r \sum_{i=0}^{n} q_{i,j} x_i  b_j,
$$
with $x_0=1$ and $\deg(q_{i,j}) \le e$.
Since $R =  \vspan{B} \oplus I_{\varphi}$, we write $x_i b_j = b_{i,j} + \kappa_{i,j}$ for some $b_{i,j}\in \vspan{B}, \kappa_{i,j}\in K$.
Thus, we have
$$
p 
= \sum_{j=1}^r \sum_{i=0}^{n} q_{i,j} b_{i,j} + \sum_{j=1}^r \sum_{i=0}^{n} q_{i,j}\kappa_{i,j} \in \big(B^{\vspan{\le e}} + (K) \big) \cap I_{\varphi} 
= \big(B^{\vspan{\le e}} \cap I_{\varphi}\big) + (K).
$$
By induction hypothesis, $B^{\vspan{\le e}} \cap I_{\varphi} \subseteq (K)$, therefore $p\in (K)$.
This concludes the induction step, so $B^{\vspan{\le e}}\cap I_{\varphi} \subseteq (K)$ for every $e \in \NN$, which implies $I_{\varphi} \subseteq (K)$ .

To conclude the proof, we observe that
$$
\textup{im}\, \bm H_{\check{f}}^{{A'}, {A}} 
 = \textup{im}\, \bm H_{\varphi}^{{A'}, {A}}
 = \textup{im}\, \bm H_{\varphi}^{{A'}, {R}}
 \subseteq I_{\varphi}^{\perp}|_{A'},
$$
since $R = \vspan{A} + I_{\varphi}$ and $I_{\varphi} = \ker \bm H_{\varphi}$.
Conversely, since 
$$
    \dim \dual{\cl A_{\varphi}} = \dim \cl A_{\varphi} = r = \rank \bm H_{\varphi}^{B',B},
$$
then for every $\phi \in I_{\varphi}^{\perp} \simeq \dual{\cl A_{\varphi}}$, there exists $b \in \vspan{B}$ such that $\phi = \bm H_{\varphi}(b)$.
Therefore, we conclude that  
$\phi|_{\vspan{A'}} = \bm H_{\varphi}^{A',A}(b) = \bm H_{\check{f}}^{{A'},{A}}(b) \in \textup{im}\, \bm H_{\check{f}}^{{A'},{A}}$.
%We conclude that $\textup{im}\, H \equiv (I_{\varphi}^{\perp})_{|\vspan{A'}}$.
\end{proof}

\begin{remark}
Part \eqref{thm4} of \Cref{thm:hankel tnf} implies that $\bm H_{\check{f}}^{A',A}$ is a Truncated Normal Form for $I_{\varphi}$ (see \cite{Mourrain2021,Telen2018,mourrain:hal-05135952}). That is, it yields a normal form modulo $I_{\varphi}$, a key ingredient used in algebraic methods for solving polynomial equations (e.g., Gröbner bases, border bases, resultant-based methods). 
%\Cref{algo:gad} will exploit these techniques to recover a GAD of the given form.
\end{remark}

%Let $\mf{h}_{x_0}(I) \subset \S$ be the homogenization of an ideal $I \subset R$ by $x_0$. \Cref{thm:hankel tnf} requires that $\vspan{A} \subseteq R_{\le c}$ (resp. $\vspan{A'} \subseteq R_{\le d-c}$) contains a basis $B$ (resp. $B'$) of $\cl A_{\varphi}$.
%, which translates to $r( \mf{h}_{x_0}(I_{\varphi}) ) \le c$ (resp. $r( \mf{h}_{x_0}(I_{\varphi}) ) \le d-c$).
%, where $\underline{I}_{\varphi}$ is the homogenization of the ideal $I_{\varphi}$ by the variable $x_0$ and $\reg(\underline{I}_{\varphi})$ is the smallest degree $d$ where $\dim \cl \S_d/(\underline{I}_{\varphi})_d =  \dim\S_{d+1}/(\underline{I}_{\varphi})_{d+1}$.
%\end{remark}

\subsection{The nil-index}

Let $\mathcal{A}_{\xi} = \kk[\xaff]/I_{\xi}$ be a local Artinian algebra defining a point $\xi \in \kk^n$, namely $I_{\xi}$ is $\mathfrak{m}_{\xi}$-primary. 

\begin{definition}
    The \textit{nil-index} of $\mathcal{A}_{\xi}$, denoted by $\mathcal{N}(\mathcal{A}_{\xi})$, is the minimal $N \in \NN$ such that, for every $a(\xaff) \in \mathcal{A}_{\xi}$, we have
    \begin{equation*}
        \big(a(\xaff) - a(\xi)\big)^N = 0 \in \mathcal{A}_{\xi}.
    \end{equation*}
\end{definition}

\begin{remark} \label{rmk:nilindex}
    Equivalently, $\mathcal{N}(\mathcal{A}_{\xi})$ is the minimal $N \in \NN$ such that
    \begin{equation}
        \mathfrak{m}_{\xi}^N = 0 \hspace{0.2cm} \textnormal{ in } \hspace{0.2cm} \mathcal{A}_{\xi} = \kk[\xaff] / I_{\xi},
    \end{equation}
    \brev{}which is also known as the \emph{Loewy length} of $A_\xi$ \cite[p.~54]{SkowronskiYamagata2011}, i.e.,\erev{} it is the minimal degree $N$ such that
    \begin{equation*}
        \mathfrak{m}_{\xi}^N \subseteq I_{\xi} \subseteq \mathfrak{m}_{\xi}.
    \end{equation*}
    \brev{}In other terms, $\mathcal{N}(\mathcal{A}_{\xi})-1 = \max_{N \in \mathbb{N}} \{ \mathfrak{m}_{\xi}^N \neq 0 \},$ which is also known as the (\emph{top}) \emph{socle degree} of $\mathcal{A}_{\xi}$ \cite[p.~526]{MASUTI2018525}.
%, namely, the largest degree of a non-zero graded component of $(0 : I_{\xi})$ in $\cl A_{\xi}$\erev{}.
\end{remark}

The following proposition relates the structure of the terms appearing in a GAD with the nil-index of the associated algebras.

\begin{proposition} \label{prop:nilindex-socle}
Let $\mathcal{A}_{\xi} = \kk[\xaff]/I_{\xi}$ be the local Artinian algebra associated with the GAD $\omega\, \ell^{d-k} \in \mathcal{S}_d$. Then
\begin{equation*}
    \mathcal{N}(\mathcal{A}_{\xi}) - 1 = k = \deg(\omega).
\end{equation*}
\end{proposition}

\begin{proof}
    Let $N \in \NN$ be such that $\mathfrak{m}_{\xi}^N \subseteq I_{\xi}$ for $|\alpha| \geq N$.
    By %\Cref{prop:3.2}
    \Cref{rmk:Ixiperp}, this means
    \begin{equation*}
        (\xaff-\xi)^{\alpha} \in I_{\xi} \iff \forall \varphi \in I_{\xi}^{\perp}, \quad (\xaff-\xi)^{\alpha} \star \varphi = 0 \iff (\xaff-\xi)^{\alpha} \star \invsys{\omega^{d,\ell,\xaff}(\zaff) \, \eval_{\xi}(\zaff)} = 0.
    \end{equation*}
    In particular, we have that $(\xaff-\xi)^{\alpha} \star \omega^{d,\ell,\xaff}(\zaff) \, \eval_{\xi}(\zaff) = 0$.
    By a repeated application of \cref{eq:exprule}, we get
    $$
        (\xaff-\xi)^{\alpha} \star \omega^{d,\ell,\xaff}(\zaff) \, \eval_{\xi}(\zaff) = (\partial_{\zaff} - \xi)^{\alpha} [\omega^{d,\ell,\xaff}(\zaff) \, \eval_{\xi}(\zaff)] = \partial_{\zaff}^{\alpha}\big(\omega^{d,\ell,\xaff}(\zaff)\big) \eval_{\xi}(\zaff).
    $$
    Therefore, the condition $(\xaff-\xi)^{\alpha} \in I_{\xi}$ for $|\alpha| \geq N$ is equivalent to
    \begin{equation*}
        \partial_{\zaff}^{\alpha}\big(\omega^{d,\ell,\xaff}(\zaff)\big) = 0, \quad \forall |\alpha| \geq N,
    \end{equation*}
which means that $\omega^{d,\ell,\xaff}(\zaff)$ is a polynomial of degree $< N$. %, i.e. $\deg(\check{\omega}) \leq N-1$.
The minimal such $N \in \NN$, which equals $\mathcal{N}(\mathcal{A}_{\xi})$ by \Cref{rmk:nilindex}, is then $N = \deg(\omega^{d,\ell,\xaff}) + 1 = k + 1$, where the last equality follows from $\ell \nmid \omega$, by definition of GAD.
\end{proof}

\begin{remark}
    \Cref{prop:nilindex-socle} rephrases, in the notation of this manuscript, the equality between the socle degree of a local Artinian Gorenstein algebra and the degree of its dual generator.
    This is studied particularly for graded algebras \cite[Chap. IV]{Mac}, while the affine version was considered in \cite[Lemma 1.1]{TonyMemoir}.
    Its connection to GADs was also investigated in \cite[Proposition 4.3]{Bernardi2024}.
\end{remark}

\begin{remark}
With the notation of \Cref{defn:MultOp}, we have
$$ \mathcal{N}(\mathcal{A}_{\xi}) = \min_{N \in \NN} \big\{ \brev{}\bm M_p\erev{} = 0 \text{ in } \brev{}\Hom_{\kk}\erev{}(\mathcal{A}_{\xi},\mathcal{A}_{\xi}) \mid \forall p \in \mathfrak{m}_{\xi}^{N} \big\}.
$$
\brev{}Indeed, $\bm M_p=0$ if and only if $p\in I_{\xi}$: one implication follows because $I_{\xi}$ is an ideal, and the converse follows by evaluating $\bm M_p$ at $1$.
The same characterization holds for the transposed multiplication operators $\bm M_p^t$ acting on $\mathcal{A}_{\xi}^*$.\erev{}
\end{remark}

\begin{comment}
\begin{remark}
    We observe that, with the notation of \Cref{defn:MultOp}, we also have
    \[
    \mathcal{N}(\mathcal{A}_{\xi}) = \min_{N \in \NN} \big\{ \brev{}\bm M^t_p\erev{} \equiv 0 \in {\brev{}\Hom_{\kk}\erev{}}(\mathcal{A}_{\xi}, \mathcal{A}_{\xi}), \ \forall p \in \mathfrak{m}_{\xi}^N \big\}.
    \]
    In fact, for every $p \in \mathfrak{m}_{\xi}^{\mathcal{N}(\mathcal{A}_{\xi})} \subseteq I_{\xi}$ and every $q\in \cl A$, we have $\brev{}\bm M^t_p\erev{}(q) = pq \in I_{\xi}$, hence $\brev{}\bm M^t_q\erev{} \equiv 0$. % over $\mathcal{A}_{\xi}$.
    On the other side, for any $p \in I_{\xi} \setminus \mathfrak{m}_{\xi}^{\mathcal{N}(\mathcal{A}_{\xi})-1}$, we have $\brev{}\bm M^t_q\erev{}(1) = p \notin I_{\xi}$, so $\brev{}\bm M^t_q\erev{} \not\equiv 0$. % over $\mathcal{A}_{\xi}$.
\end{remark}
\end{comment}

\subsection{Schur factorization}

Let $I = \cap_{i = 1}^s I_{\xi_i} \subset \kk[\xaff]$ be a zero-dimensional ideal, and consider the decomposition of the Artinian algebra $\mathcal{A} = \kk[\xaff]/I$ as in \Cref{thm:artinian}:
\begin{equation*}
    \mathcal{A} = \mathcal{A}_{\xi_1} \oplus \cdots \oplus \mathcal{A}_{\xi_s},
\end{equation*}
where $A_i \simeq \kk[\xaff]/I_{\xi_i}$ is the local algebra associated with the point $\xi_i \in \kk^n$.

For a given matrix $M \in \End(\cl A)$, its \emph{Schur decomposition} is 
$$
M = U \, T \, U^{-1},
$$
where $U$ is a unitary matrix, while $T$ is upper triangular.

\begin{proposition}\label{prop:3.20}
Let $\mathcal{A} = \bigoplus_{k=1}^s \mathcal{A}_{\xi_k}$ as above, and let %$\lambda = \sum_{i=1}^n \lambda_i \, x_i \in \kk[\xaff]$
$\lambda \in \kk[\xaff]$ be a generic linear form.
Let $M_{\lambda} = U \, T \, U^{-1}$ be the Schur factorization of the multiplication-by-$\lambda$ operator. 
Then, for any polynomial $g \in \kk[\xaff]$, the conjugated multiplication operator $\tilde{M_g} := U^{-1} \, M_g \, U$ is block upper triangular, i.e.,
    \begin{equation*}
     \tilde{M}_g =    \begin{bmatrix}
    \tilde{M}_g^{(1)}       & * & \dots & * \\
    0       & \tilde{M}_g^{(2)} & \dots & * \\
    \vdots       & \vdots & \ddots & \vdots \\
    0      & 0 & \dots & \tilde{M}_g^{(s)}
\end{bmatrix},
    \end{equation*}
where each block $\tilde{M}_g^{(k)}$ represents the multiplication-by-$g$ in $\mathcal{A}_{\xi_k}$.
\end{proposition}

\begin{proof}
    By \cite[Theorem 4.23]{EM07}, the eigenvalues of $M_{\lambda}$ are $\{\lambda(\xi_i)\}_{i \in \{ 1, \ldots, s \}}$, namely the evaluation of $\lambda$ at the roots $\xi_i$, with multiplicities $\mu_i = \dim_{\kk} \mathcal{A}_{\xi_i}$.
    Since $\lambda$ is generic, we may assume that the $\lambda(\xi_i) \in \kk$ are pairwise distinct, i.e., we have well-separated eigenvalues.
    
    Let us define the increasing filtration
    \begin{equation*}
        0 = V_0 \subset V_1 \subset \dots \subset V_s = \mathcal{A},
    \end{equation*}
    where $V_i := \bigoplus_{j=1}^{\brev{}i\erev{}} \mathcal{A}_{\xi_j}$ is the sum of the first $i$ local algebras.
    Each subspace $V_i$ is invariant under $M_{\lambda}$, and we have an isomorphism of $\kk$-vector spaces $V_i/V_{i-1} \simeq \mathcal{A}_{\xi_i}$.
    Hence, the action of $M_{\lambda}$ on $\mathcal{A}$ decomposes into a sequence of actions on the quotients $V_i/V_{i-1}$.
    %, where $M_{\lambda}$ acts via multiplication by $\lambda(x)$ with eigenvalue $\lambda(\xi_k)$.

    Let $M_{\lambda} = U \, T \, U^{-1}$ be the Schur factorization of $M_{\lambda}$.
    Since the eigenvalues $\lambda(\xi_i)$ are distinct, we can order the Schur basis vectors such that, for every $1 \leq i \leq s$, the first $\mu_1 + \dots + \mu_i$ vectors form a basis for $V_i$.
    On this basis, $T$ has a block upper triangular form
\begin{equation*}
        T =  \begin{bmatrix}
    T^{(1)}       & * & \dots & * \\
    0     & T^{(2)}  & \dots & * \\
    \vdots       & \vdots & \ddots & \vdots \\
   0   & 0 & \dots & T^{(s)}
    \end{bmatrix},
\end{equation*}
    where each $T^{(i)} \in \End(\kk^{\mu_i})$ is upper triangular and has all diagonal entries equal to $\lambda(\xi_i)$. 
%\todo{Are the $T^{(i)}$ diagonal?}
    Thus, each $T^{(i)}$ represents the multiplication by $\lambda$ in the quotient $V_i/V_{i-1} \simeq \mathcal{A}_{\xi_i}$.
%    These blocks correspond to invariant subspaces under $M$, and each block $T^{(k)}$ represents the action of $\lambda(x)$ on $\mathcal{A}_{\xi_k}$ in the Schur basis. For each $k$, the action of $M$ on the invariant subspace $V_k$ (spanned by the first $\sum_{j=1}^k \mu_j$ Schur basis vectors) induces a quotient action on $V_k / V_{k-1} \simeq \mathcal{A}_{\xi_k}$ given by $T^{(k)}$.
 %   Thus, the block $T^{(k)}$ is precisely the matrix representation of the multiplication operator by $\lambda(x)$ in the local algebra $\mathcal{A}_{\xi_k}$, written in a Schur basis.
%Now let us analyze the structure of $M_g \in End_{\kk}(\mathcal{A})$, under the same change of basis:
%    \begin{equation*}
%        \tilde{M}_g := U^T \, M_g \, U .
%    \end{equation*}

Since $M_g$ commutes with $M_{\lambda}$, by applying \cite[Proposition 4]{Corless1997} it also preserves the eigenspace filtration $V_0 \subset \dots \subset V_s$, hence it admits the same block decomposition as $T$.
Consequently, $\tilde{M}_g := U^{-1} M_g U$ has the same block triangular form:
    \begin{equation*}
     \tilde{M}_g =    \begin{bmatrix}
    \tilde{M}_g^{(1)}       & * & \dots & * \\
    0       & \tilde{M}_g^{(2)} & \dots & * \\
    \vdots       & \vdots & \ddots & \vdots \\
    0      & 0 & \dots & \tilde{M}_g^{(s)}
\end{bmatrix} .
    \end{equation*}
    Each block $\tilde{M}_g^{(i)}$ acts on $\mathcal{A}_{\xi_i}$, and represents the multiplication operator by $g$ in the local algebra $\mathcal{A}_{\xi_i}$, expressed in the Schur basis.
\end{proof}

%\subsection{Eigencomputation EB}
%Operators of multiplication properties, Local Artinian Algebras

%Diagonal blocks, Nil-indices\\
%Propositions: 
%\begin{itemize}
   % \item The block corresponds to the multiplication in a basis of the local algebra
   % \item Properties relating the degrees $d_{G_i}$ of $G_i$ to nil-indices $NX_i$: 
   % \\$d_{G_i} = NX_i - 1$
%\end{itemize}
%Properties of Schur factorisation, Join Schur factorisation.

\section{Apolar schemes, ranks and minimal decompositions} \label{sec:ApolarSchemes}
In this section, we analyse the apolar schemes associated with a GAD and show that, under regularity hypotheses, a form $f\in \Sd$ has a unique minimal GAD and a unique apolar scheme of minimal length.

Let $\Iprj\subset \S$ be a homogeneous ideal.
We recall that $\Iprj$ is \emph{saturated} if $\big(\Iprj: (x_0, \dots, x_n)\big)=\Iprj$, i.e., $\Iprj$ {\brev{}is the defining ideal of a projective scheme\erev{}}.
For the sake of brevity, we will identify projective schemes with their defining homogeneous saturated ideals.
The dimension of such a scheme is $1$ less than the dimension of its \emph{affine cone} $\V(I) \subseteq \mathbb{A}^{n+1}$.
Hence, a zero-dimensional projective scheme $Z$ is supported at finitely many points in $\mathbb{P}^n(\K)$, and its localizations at these points are Artinian algebras.
We say that such $Z$ is \emph{(locally) Gorenstein} if these localizations are Artinian Gorenstein algebras.
Finally, we say that $\Iprj$ is \emph{apolar} to $f\in \Sd$ if $f^* \in \Iprj_{d}^{\perp}$ or, equivalently, if
$\Iprj \subseteq \ann (f^*)$, by regarding $f^*$ as an element of $\S^*$ (see \cite[Lemma 1.15]{Iarrobino1999}).

%\todo[inline]{Check regularity definition. I believe what we need in our paper is the regularity index $r(I)$. For example: $P_1 = (1,0,0), P_2 = (0,1,0), P_3 = (0,0,1)$, $X = \{P_1,P_2,P_3\}$. We have $$ I_X = (xy,yz,xz), \quad r(I_X) = 1, \quad \reg(\S/I_X) = 2-1=1, \quad \S/I_X = \langle x,y,z \rangle, $$ and the interpolation degree is $1$. In fact, with our method we can compute the GAD x^3+y^3+z^3$, i.e. $d=3, c=2$ in our theorems.}
We recall that the \emph{Castelnuovo-Mumford regularity} $\reg(\S/\underline{I})$ of the $\S$-module $\S/\underline{I}$ is $\max_{i,j} \set{a_{i,j}-i}$ for a minimal free resolution 
$$
0 \rightarrow \oplus_{j} \S(-a_{l,j}) \rightarrow \cdots \rightarrow \oplus_j \S(-a_{1,j}) \rightarrow  \S \rightarrow  \S/\underline{I} \rightarrow  0
$$
(see \cite[chap. 4]{eisenbud2005geometry}).
The \emph{regularity index} $r(\underline{I})$ of the Hilbert function is the minimal degree from which the Hilbert function of $\cl S/\underline{I}$ agrees with its Hilbert polynomial.
When $\underline{I}$ is saturated and (projectively) $0$-dimensional, 
%i.e., it defines a finite non-empty set of projective points $\V(\underline{I}) = \set{\xi_1, \ldots, \xi_s} \subset \PP^{n}$, its Hilbert polynomial is the constant $\deg(\underline{I})$ and
we have, for a generic $v \in \Sone$,
$$
    \reg(\S/\underline{I}) = r(\underline{I}) %= \min_{d \in \N} \set{\dim \S_d/\underline{I}_d = \dim\S_{d+1}/\underline{I}_{d+1}} 
    = \min_{d \in \N} \set{\bm M_{v}: p\in \S_d/\underline{I}_d  \mapsto p\, v \in \S_{d+1}/\underline{I}_{d+1} \textup{ is surjective}}
$$
(see \cite[Theorem 1.69]{Iarrobino1999}).
The regularity $\reg(\S/\underline{I})$ is the minimal degree of a basis in the localization $\cl A$ of $S/\underline{I}$ at $v$. 
When $\underline{I}$ is radical, it is also known as the \emph{interpolation degree} of $\V(\underline{I})$ \cite[chap. 4]{eisenbud2005geometry}.

%it is the minimal degree $d$ such that $\eval_{\xi_1}, \ldots, \eval_{\xi_s} \in \dual{\Sd}$ are linearly independent functionals. It is related to the \emph{Castelnuovo-Mumford regularity} $\reg(S/\underline{I})$ by $\reg(S/\underline{I}) = r(\underline{I})$ \cite[Theorem 1.69]{Iarrobino1999}.

\subsection{Apolar schemes from GADs}
Consider a GAD $f=\sum_{i=1}^{s} \omega_{i}\, \ell_i^{d-k_i} \in \Sd$ as in \cref{eq:gad}.
By a generic change of coordinates, we can assume that $\xiprj_{i,0}\neq 0$ and, by scaling $\omega_i$ and $\ell_i$, we will always assume that $\xiprj_{i,0}=1$.
We define the homogenized operator of \cref{eq:extendedphi} as
$$
\phiprj := \sum_{i=1}^{s} \omega_{i}^{d,\ell_i, \xaff}\, \eval_{\ell_i}\in \kk[[\zprj]] = \dual{\cl S},
$$
and the homogeneous ideal
\begin{equation}\label{eq:I phiprj}
{I}_{\phiprj} := \textup{Ann}(\phiprj) = \vspan{p\in \cl S_k \mid p\star \phiprj =0}_{k \in \NN}.    
\end{equation}

\begin{lemma}\label{lem:ann w exp}
The ideal $I_{\phiprj}$, as defined above, is apolar to $f$.
\end{lemma}
\begin{proof}
For every $p \in (I_{\phiprj})_d$, we have \brev{} $\apply{\phiprj^{[d]}}{p} = \apply{\phiprj}{p} = \apply{p\star\phiprj}{1} = 0$ \erev{} %$\scp{\phiprj^{[d]}, p} = \scp{\phiprj, p} = \scp{p\star\phiprj, 1} = 0$,
hence $\phiprj^{[d]} \in (I_{\phiprj})_{d}^{\perp}$.
By \Cref{lem:w xi dual}, we conclude that
$f^* = \big( \sum_{i=1}^{s} \omega_i^{d,\ell_i,\xaff}(\zprj) \eval_{\ell_i}(\zprj)\big)^{[d]}  = \phiprj^{[d]}\in (I_{\phiprj})_{d}^{\perp}$.
\end{proof}

\begin{proposition}\label{prop:invsys}
For every $\omega \in \Sk \brev{}\setminus \{0\}\erev{}$ with $k\le d$ and $\ell = (\xiprj,\xprj) \in \Sone$ with $\xiprj_0=1$, the ideal $\Ann(\omega^{d,\ell,\xaff} \, \eval_{\ell})$ is saturated, supported on $\xiprj$ and is equal to 
\begin{equation}\label{eq:ann 1}
\Ann(\omega^{d,\ell,\xaff}\, \eval_{\ell}) = \vspan{ p \in \S_i \mid \forall \theta \in \invsys{\omega}^{d, \ell,\xaff}, \ \ \theta(\partial_{\xaff})(p)(\xiprj)= 0}_{i \in \NN}.
\end{equation}
%If also , then we have 
%\begin{equation}\label{eq:ann 2}
%\Ann(\omega\, \eval_{\ell})= \vspan{ p \in \Sd , d \in \NN \mid \forall 
%\theta \in \invsys{\omega^{\xaff}},
%\alpha \in \NN^{n} \textup{ and }\theta=\partial_{\xaff}^{\alpha}(\omega^{\xaff}), 
%\ \  \theta(\partial_{\xaff})(p)(\xiprj)= 0}.
%\end{equation}
\end{proposition}
\begin{proof}
%Let $\vprj= \set{v_0, v_1, \ldots, v_n}$ be a basis of $\Sone$ and  $\ztprj= (\scp{\ell, v_0}_a, \ldots, \scp{\ell, v_n}_a)\in \kk^{n+1}$. By Leibnitz rule, we have $\forall \alpha\in \NN^{n+1}$,
% \begin{align}\label{eq:leibnitz} 
% \vprj^{\alpha} \star (\omega\eval_{\ell}) & = \partial_{\vprj}^{\alpha} (\omega\eval_{\ell}) = \sum_{\beta\ll \alpha}c_{\alpha,\beta} \ztprj^{\beta} \, \partial_{\vprj}^{\alpha-\beta}(\omega) \eval_{\ell}
% \end{align}
%for some coefficients $c_{\alpha,\beta}\in \NN$, and $\beta \ll \alpha$ iff $\beta_i\le \alpha_i$ for $i=1, \ldots, n$.

%Let $\set{v_1, \dots, v_n}\subset \Sone(\ell^{\perp})$ so that $\set{\ell, v_1, \ldots, v_n}$ is a basis of $\Sone$ (It exists since $\ell$ is not isotropic).
% and $\ztprj= (\rho, 0, \ldots, 0)$ with $\rho= \scp{\ell, \ell}$. 
By definition, $p\in \Ann(\omega^{d, \ell,\xaff} \eval_{\ell})_i= \Ann(\omega^{d, \ell,\xaff} \eval_{\ell})\cap \S_i$ means that for every $q \in \S$, we have $\apply{p\star (\omega^{d, \ell,\xaff} \eval_{\ell})}{q}=0$, which is equivalent to $\forall j,l\in \NN$ and $q \in \S_l(\xaff)$ to
\begin{align*}
0 & = \apply{p\star (\omega^{d,\ell, \xaff} \eval_{\ell})}{q\, x_0^{j}}
  = \apply{q \star (\omega^{d,\ell, \xaff} \eval_{\ell})}{p\, x_0^{j}} 
  = \apply{q(\partial_{\xaff}) (\omega^{d,\ell, \xaff} \eval_{\ell})}{p\, x_0^j}.
\end{align*}
By Leibnitz rule, as in \cref{eq:exprule}, we have
$$
\vspan{\partial_{\xaff}^{\alpha}(\omega^{d,\ell, \xaff}\eval_{\ell})}_{\alpha \in \NN^n}
= \vspan{\partial_{\xaff}^{\alpha}(\omega^{d,\ell, \xaff})}_{\alpha \in \NN^n} \, \eval_{\ell}\brev{}, \erev{}
$$ so that 
$q(\partial_{\xaff}) (\omega^{d,\ell, \xaff}  \eval_{\ell})= \theta \,\eval_{\ell}$ where $\theta  \in \invsys{\omega}^{d, \ell,\xaff}$. 
Therefore, we have 
\begin{align*}
\apply{q(\partial_{\xaff}) (\omega^{d,\ell, \xaff}  \eval_{\ell})}{p\, x_0^j}  
&  
 = \apply{\theta \eval_{\ell}}{p\, x_0^j}
 = \theta(\partial_{\xaff})(p \, x_0^j)(\xiprj)
 = \big(\theta(\partial_{\xaff})(p)\, x_0^j\big)(\xiprj)
 = \theta(\partial_{\xaff})(p)(\xiprj).
\end{align*}
This shows the equality \eqref{eq:ann 1}.

Let us prove now that $\ann(\omega^{d,\ell, \xaff} \eval_{\ell})$ is saturated and supported on $\xiprj$.
%By a change of coordinates $\phi \in GL_{n+1}(\kk)$, we can assume that $\xiprj=(1,0,\ldots, 0)$. 
For all $p\in \S_i$ contained in $\big(\ann(\omega^{d,\ell, \xaff} \eval_{\ell}):(x_0, \ldots, x_n)\big)$ %with $d\in \NN$, 
we have $x_j p \in \ann(\omega^{d,\ell, \xaff} \eval_{\ell})$ for $j=0, \ldots, n$. In particular, $x_0\, p\in \ann(\omega^{\brev{}d,\ell, \xaff\erev{}}\eval_{\ell})$. That is, by \cref{eq:ann 1},  for all $\alpha \in \NN^n$ and $\theta = \partial_{\xaff}^{\alpha}(\omega^{d,\ell, \xaff})$, we have
$$
0 = \theta(\partial_{\xaff})(x_0\, p) (\brev{}\xiprj\erev{}) 
= \big(x_0\, \theta(\partial_{\xaff})(p) \big) (\xiprj) 
= \theta(\partial_{\xaff})(p) (\xiprj),
$$
which proves that $p\in \ann(\omega^{d,\ell, \xaff}  \eval_{\ell})$, hence $\ann(\omega \eval_{\ell})$ is saturated.
Since $1 \in \invsys{\omega}^{d, \ell,\xaff}$, then $p(\xiprj)=0$, i.e. $\ann(\omega^{d,\ell, \xaff} \eval_{\ell}) \subseteq (x_1-\xiprj_1 x_0, \ldots, x_n - \xiprj_n x_0)$.
\brev{}Since every $\theta \in \invsys{\omega}^{d,\ell,\xaff}$ has degree at most $k$, \cref{eq:ann 1} also gives
$$
(x_1-\xiprj_1x_0,\ldots,x_n-\xiprj_nx_0)^{k+1}
\subseteq \Ann(\omega^{d,\ell,\xaff}\,\eval_\ell).
$$
This proves that the ideal is supported on $\xiprj$.\erev{}
%which proves that $\ann(\omega^{d,\ell, \xaff}  \eval_{\ell})$ is supported on $\xiprj$.
%This concludes the proof.
\end{proof}

\begin{example}
Let us consider \brev{}$\omega = x_0^3 + x_0^2x_1 + x_0x_1x_2 + x_1x_2^2 + x_2^3 \in \S_3$\erev{} and \brev{}$\ell = x_0+x_1 \in \S_1$\erev{}.
Since 
$$
    \brev{}\omega = (-x_1^2x_2 + x_1x_2^2 + x_2^3) + ( x_1^2 + x_1x_2 ) \ell - 2x_1\ell^2 + \ell^3,\erev{}
$$
for \brev{}$\bm v = \{x_1,x_2\}$\erev{} we obtain
$$
    \brev{}\omega^{3, \ell, {\bm v}} = \frac{1}{6} (-x_1^2x_2 + x_1x_2^2 + x_2^3) + \frac{1}{6}(x_1^2 + x_1x_2) - \frac{2}{3}x_1 + 1 \in \S(\bm v),\erev{}
$$
as defined in \cref{eq:omega v}. The inverse system of $\omega^{3,\ell, \bm v}$ is
$$
    \brev{}\invsys{\omega}^{3,\ell,\bm v} = \Big\langle \omega^{3, \ell, {\bm v}}, 
    -2 x_1x_2 + x_2^2 ,
    -x_1^2 + 2 x_1x_2  + 3 x_2^2, 
    x_2,
    x_1,
    1 \Big\rangle.\erev{}
$$
Thus, $\rank_{\ell}(\omega) = 6$, which means that the rank of the GAD of $f = \omega\, \ell^0$ is $6$.
By computing the kernel of the Hankel matrix in degree $3$, we check that $\underline{I}:=\Ann(\omega^{3,\ell,\xaff} \, \eval_{\ell})$ is
\begin{align*}
    \brev{}\Big(
    (x_0-x_1)^3,\,
    & (3 x_0^2 - 6 x_0x_1 + 3 x_1^2 + x_2^2) x_2,
    (-3 x_0 + 3 x_1 - x_2) x_2^2, \\
    &-8 x_0^2x_1 + 16 x_0x_1^2 + 2 x_0x_1x_2 - 8 x_1^3 - 2 x_1^2x_2 - 2 x_1x_2^2 + 3 x_2^3 \Big),\erev{}
\end{align*}
whose schematic length is $6$. Its Hilbert function is $(1,3,6,6,\ldots)$ so that $\reg(\S/\underline{I})= 2$.
The elements $g \in \underline{I}$ are such that $\forall \alpha, \beta, \gamma \in \N$, 
\[ \brev{}\apply{g \star (\omega^{3, \ell, {\bm v}} \eval_{\ell})}{x_0^\alpha x_1^\beta x_2^\gamma } = \apply{\omega^{3, \ell, {\bm v}} \eval_{\ell}}{g\, x_0^\alpha x_1^\beta x_2^\gamma } = 
\omega^{3, \ell, {\bm v}}(\partial_{x_1},\partial_{x_2})(g\, x_1^\beta x_2^\gamma )(1,1,0)= 0. \erev{}
\]
\end{example}

Let $\varphi = \sum_{i=1}^{s}  {\omega}^{d,\ell_i,\xaff}_{i}(\zaff)\, \eval_{\xi_i}(\zaff)\in \kk[[\zaff]]= \dual{R}$ as in \cref{eq:Iphi}, and let
$$
I_{\varphi}= \textup{Ann}(\varphi)= \set{p\in R\mid p\star \varphi =0}\subset \KKrng = R.
$$
The homogeneous ideal $I_{\phiprj}$ has the following properties:
\begin{lemma}\label{lem:Iphiprj}
The ideal $I_{\phiprj}$ is saturated, zero-dimensional supported on $\set{\xiprj_1, \ldots, \xiprj_s}$, and 
$$
I_{\phiprj} = \mathfrak{h}_{x_0}(I_{\varphi}) = \cap_{i=1}^{s} \Ann(\omega_i^{d,\ell_i,\xaff} \eval_{\ell_i})\brev{}, \erev{}
$$
where $\mf{h}_{x_0}$ is the homogenization in $x_0$.
%and   $\ell_i= (\xiprj_i,\xprj)$.
\end{lemma}
\begin{proof}
For $p\in R$, let $p^h = \mf h_{x_0}(p)$ be its homogenization with respect to $x_0$.
Since for every $\theta \in \kk[\xaff]$ we have $\theta(\partial_{\xaff})(p^h)(\xiprj) =
\theta(\partial_{\xaff})(p)(\xi)$, then
$$
\apply{\phiprj}{p^h} 
= \sum_{i=1}^{s} \apply{\omega_i^{d,\ell_i, \xaff} (\zprj) \eval_{\xiprj_i}(\zprj)}{p^h} 
= \sum_{i=1}^{s} \apply{{\omega}_i^{d,\ell_i, \xaff}(\zaff) \eval_{\xi_i}(\zaff)}{p} 
= \apply{\varphi}{p}\brev{}. \erev{}
$$
We deduce that, for every $q \in R$, we have
$$
\apply{p\star \varphi}{ q} = \apply{\varphi}{p\, q} 
= \apply{\phiprj}{ (pq)^h}
= \apply{ p^h \star \phiprj}{ q^h}.
$$
Thus $p \star \varphi=0 \in R^*${\brev{} if and only if\erev{}} $p^h \star \phiprj=0 \in \S^*$.
In other words, $p \in I_{\varphi}${\brev{} if and only if\erev{}} $p^h \in I_{\phiprj}$, thus $I_{\phiprj} = \mathfrak{h}_{x_0}(I_{\varphi})$.

By \Cref{thm:kronecker}, $p\in I_{\varphi}$ {\brev{}if and only if\erev{}} for every $\theta \in \invsys{\omega_i}^{d,\ell_i,\xaff}$ we have $\theta(\partial_{\xaff})(p)(\xi_i)=0$.
Since, as above, $\theta(\partial_{\xaff})(p)(\xi_i) = \theta(\partial_{\xaff})(p^h)(\xiprj_i)$, we deduce by \Cref{lem:ann w exp} that $p^h\in I_{\phiprj}$ is equivalent to $p^h\in \cap_{i=1}^{s}\Ann(\omega_i^{\brev{}d,\ell_i,\bm x\erev{}} \eval_{\ell_i})$. 

Consequently, by \Cref{prop:invsys}, $I_{\phiprj}$ is the intersection of saturated ideals $\ann(\omega_i^{\brev{}d,\ell_i,\bm x\erev{}} \eval_{\ell_i})$ supported on $\xiprj_i$, hence it is saturated, zero-dimensional, and supported on $\set{\xiprj_1, \ldots, \xiprj_s}$.
\end{proof}

% \begin{remark}
The homogeneous ideal $I_{\underline{\varphi}} = \ann (\underline{\varphi})$ associated with the GAD defines a punctual projective scheme (\Cref{lem:Iphiprj}), which was first considered in \cite{Bernardi2024} and called the scheme \emph{evinced} by the GAD from which $\underline{\varphi}$ is constructed.
Its local components correspond to the \emph{natural apolar schemes}, as defined in  \cite{bernardi2018polynomials}, i.e. the annihilator of $f_{\textnormal{dp}}$ for the contraction (for $\ell = x_0$).
Our definition, as in \cref{eq:I phiprj}, describes this object as the annihilator of a single explicit polynomial exponential series.

\subsection{GAD-rank}
%We recall from \Cref{rmk:reg} that the regularity of a homogeneous zero-dimensional (projective) ideal is the degree at which its Hilbert function stabilizes.

%For an omogeneous ideal ${I}\subset \S$, let $\underline{\cl A}= \S/{I}$ and $\underline{\cl A}_d = \Sd/{I}_d$. Let $I_{\ell^*}$ be its localization in $\S_{\ell^*} \simeq R$ and let $\cl A = \S_{\ell^*}/I_{\ell^*}$ (see \cite[Chap. 3]{Atiyah2018}).
For \brev{}a\erev{} homogeneous ideal $\underline{I} \subset \S$, we consider its graded quotient algebra $\underline{\cl A} = \S/\underline{I}$.
We will need the following (classical) lemma:
\begin{lemma}\label{lem:basisloc} Let $d\ge \reg(\underline{\cl A})$, $v\in \S_1$, and assume that $\bm M_{v,d}: a \in \Aprj_d \mapsto a v  \in \Aprj_{d+1}$ is bijective. 
If $\underline{B} \subset \S_d$ is a basis of $\Aprj_{d}$, then $\underline{B}$ is also a basis of $\S/(\underline{I}, v-1)$.
\end{lemma}
\begin{proof}
Let $\set{v_0, v_1, \ldots, v_n}$ be a basis of $\Sone$, with $v_0=v$.
If $d \ge \reg (\brev{}\Aprj\erev{})$ and $\bm M_{v,d}$ invertible, then 
for all $k \in \NN$, $\bm M_{v,d+k}: a \in \underline{\cl A}_{d+k} \mapsto a v \in \Aprj_{d+k+1}$ is invertible (see e.g. \cite[Theorem 1.10]{Bayer1987}), hence $\bm M_{v^k,d}: a \in \Aprj_d \mapsto a v^k \in \Aprj_{d+k}$ is also bijective. 
Thus, since $\underline{B}$ is a basis of $\Aprj_d$, then 
\begin{equation}\label{eq:basis}
\S_{d+k} = v^{k} \vspan{\underline{B}} \oplus \Iprj_{d+k}. 
\end{equation}

Let $\sigma: \S \to \S/(v-1) \brev{}\simeq \kk[v_1, \dots, v_n]\erev{}, \, p \mapsto p(1, v_1, \dots, v_n)$ be the dehomogenization map, and let $I= \sigma(\underline{I})$ and $B = \sigma(\underline{B})$.
\brev{}The statement is equivalent to proving that $B$ is a basis of $\cl A = \kk[v_1, \dots, v_n]/I$.\erev{}
\brev{}We first prove that $B$ is a generating set for $\cl A$.\erev{}
For any \brev{}$p \in \kk[v_1, \dots, v_n]$\erev{}, we consider its homogenization $q \in \S_{d'}$ in degree \brev{}$d' = \max\{d, \deg p\}$\erev{}.
Then $\sigma(q)=p$ and $q \in v^k \vspan{\underline{B}} + \underline{I}_{d+k}$ for \brev{}$k = d'-d$\erev{}, which implies that $p \in \vspan{B} + I$.

We now show that the polynomials of $B$ are linearly independent in $\cl A$.
For $p \in \vspan{B} \cap I$, \brev{}we have $p = \sigma(b)$, for some $b \in \vspan{\underline{B}}$, and $p = \sigma(F)$, for some $F = \sum_{j=0}^{k} F_j$, with $F_j \in \underline{I}_j$.
Let $q \in \S_{d+k}$ be the homogenization of $p$ in degree $d + k$.
Since the homogenization in a given degree is unique, we have both
\[ q = v^k b \in v^k\vspan{\underline{B}} \quad \textnormal{and} \quad q = \sum_{j=0}^{k} v^{d+k-j} F_j \in \underline{I}_{d+k}, \]
hence we conclude by \cref{eq:basis} that $q \in v^k\vspan{\underline{B}} \cap \underline{I}_{d+k} = 0$, which implies $p=0$.
\erev{}
\iffalse
let $q \in \S_{d'}$ be its homogenization in degree $d'\ge d$.
Since the sum \cref{eq:basis} is direct, we conclude that $q= 0$, therefore $p = \sigma(q)= 0$.
This shows that $B$ is a basis of $\cl A = \Aprj/(v-1)$ and concludes the proof.  
\fi
\end{proof}
The algebra $\cl A = \underline{\cl A}/(v-1)$ is the \emph{localization} of $\underline{\cl A}=\S/\underline{I}$ at $v$, denoted $\underline{\cl A}_{v^*}$, obtained by ``substituting" $v$ by $1$.

We now consider the graded algebra $\S/I_{\phiprj}$ associated with the GAD $f = \sum_{i=1}^{s} \omega_{i}\, \ell_i^{d-k_i}$ (the homogenized analogue of \Cref{defn:evincedbyGAD}, by \Cref{lem:Iphiprj}).
We prove that, if its regularity is small enough, then the considered GAD is the unique minimal one for $f \in \S_d$.
\begin{theorem}\label{thm:rk gad}
Let $1 \leq c \leq d$. If $\reg(\S/I_{\phiprj}) \le \min (d-c, c-1)$, then 
$$
 \gadrank(f) = \rank \bm H_{f^*}^{d-c,c}\brev{}, \erev{}
$$
and $f$ has a unique minimal GAD, up to terms permutation. % given by $f=\sum_{i=1}^{s}\omega_i \ell_i^{d-{k_i}}$.
\end{theorem}
\begin{proof}
Let
$f= \sum_{i=1}^{s} \omega_i \ell_i^{d-k_i}$ \brev{}be a GAD
with $\phiprj := \sum_{i=1}^{s} \omega_{i}^{d,\ell_i, \xaff}\, \eval_{\ell_i}$ such that $\reg(\S/I_{\phiprj}) \le \min (d-c, c-1)$.\erev{}
%i.e., $r_g := \gadrank(f) = \rank(\sum_{i=1}^{s} \omega_i \ell_i^{d-k_i})$.
%By a generic change of coordinates and scaling, we can assume that  $\xiprj_{i,0}=1$ and that $\bm M_{x_0}: \underline{\cl A}_d \rightarrow \underline{\cl A}_{d+1}$ is bijective.   
Since $I_{\phiprj}$ is zero-dimensional (\Cref{lem:Iphiprj}) and $\reg(\S/I_{\phiprj}) \le \min (d-c, c-1)$, the maps $M_{v,d-c}$ and $M_{v,c-1}$ are bijective for a generic $v \in S_1$.
Hence, by \Cref{lem:basisloc}, the localization $(\S_{c-1})_{x_0^*}\simeq R_{\le c-1}$ (resp. $(\S_{d-c})_{x_0^*}\simeq R_{\le d-c}$) contains a basis $B$ (resp. $B'$) of $\cl A_{\varphi} = R/I_{\varphi}\simeq ({\cl S}_d)_{x_0^*}/(I_{\phiprj,d})_{x_0^*}$ with $1 \in B$ and $B^{+}\subset R_{\le c}$.
Let $A$ (resp. $A'$) be a basis of $R_{\le c}$ (resp. $R_{\le d-c}$), so $B^+\subset \vspan{A} = R_{\le c}$.
Therefore, by \Cref{thm:hankel tnf}, we have $r := \rank H^{d-c,c}_{f^*} = \rank H^{A',A}_{\check{f}} = \sum_{i=1}^{s} \dim \invsys{{\omega}^{d, \ell_i,\xaff}_i}$. % by \Cref{defn:invsystrankGAD}.

\brev{}Suppose that $r_g := \gadrank(f)< r$. Then there exists $\phiprj'\in \dual{\cl S}$ such that, for the dehomogenization $\varphi' \in \dual{R}$,
$$
r > r_g = \rank \bm H_{\varphi'} \ge \rank \bm H^{d-c,c}_{(\varphi')^{[\le d]}} = \rank \bm H^{d-c,c}_{\check{f}} = r,
$$
which leads to a contradiction. Thus $r= r_g = \gadrank(f)$ and $f= \sum_{i=1}^{s} \omega_i \ell_i^{d-k_i}$ is a minimal GAD.
\erev{}

Let us now consider another minimal GAD $f = \sum_{i=1}^{s'} \omega_i' (\ell'_i)^{d-k'_i}$.
We know by \cref{eq:fcheckastrunc} that $(\varphi')^{[\le d]}= \check{f}$, hence
%Assume that $f= \sum_{i=1}^{s'} \omega_i' (\ell'_i)^{d-k'_i}$ with $\omega'_i\in \cl S_{k'_i}, \ell'_i= (\xiprj'_i, \xprj) \in \cl S_1$ not dividing $\omega'_{i}$, is a minimal GAD of $f$ so that $r_g= \sum_{i=1}^{s'} \dim \invsys{\omega'_i}_{\ell'_i}$.
%
%By a generic change of coordinates and scaling, we can assume that $\xiprj'_{i,0}=\xiprj_{i,0}=1$.
%Let $\varphi' = \sum_{i=1}^{s'} (\omega_i')^{d,\ell_i',\xaff}(\zaff) \eval_{\xi'_{i}}(\zaff) \in \kk[[\zaff]]= \dual{R}$ where $\xi'_{i}= (\xiprj'_{i,1}, \ldots, \xiprj'_{i,1})\in \kk^{n}$, so that $(\varphi')^{[\le d]}= \check{f}$.
by \cref{eq:restrict} we have
$H_{\varphi'}^{A',A} 
= H_{\check{f}}^{A',A} 
= H_{{f}^*}^{d-c,c}$. 

Thus, by \Cref{thm:hankel tnf}, we have
%Then by \Cref{thm:kronecker}, we have 
%$r_g = \sum_{i=1}^{s'}\dim \invsys{(\omega_i')^{d, \ell_i,\xaff}} = \rank H_{\varphi'} \ge \rank H_{\varphi'}^{A',A} = \rank H^{d-c,c}_{f^*} = r$.
%Thus $r=r_g$ and $\dim \cl A_{\varphi'} = r$ where $I_{\varphi'}$ and $\cl A_{\varphi'}$ are defined respectively in \cref{eq:Iphi} and \cref{eq:Aphi}.
$$ K := \ker H_{\varphi}^{A',A} = \ker H^{A',A}_{\check{f}} = I_{\varphi'}\cap \vspan{A}, $$
%Since $\ker H_{\varphi'}^{A',A} = I_{\varphi'}\cap \vspan{A} = \ker H^{A',A}_{\check{f}} =: K$, we have 
%and $K\subset I_{\varphi'}$.
and $(K)=I_{\varphi}$. Hence, we have $I_{\varphi} \subseteq I_{\varphi'}$, and by \Cref{lem:inclusion}, we deduce that $I_{\varphi} = I_{\varphi'}$ and $\cl A := \cl A_{\varphi} = \cl A_{\varphi'}$.

By \Cref{thm:struct dual}, we have $\set{\xi_1, \ldots, \xi_{s}}= \set{\xi'_1, \ldots, \xi'_{s'}}$ so that $s=s'$ and up to a permutation, $\xi_i= \xi'_i$ for $i=1, \ldots, s=s'$. 
We also deduce that $\dual{\cl A_{\xi_i}}= \invsys{{\omega}_i^{d, \ell_i,\xaff}} \,\eval_{\xi_i}= \invsys{({\omega}'_i)^{d, \ell_i, \xaff}}\,\eval_{\xi_i}$.

Let $\mu_i= \dim \cl A_{\xi_i}$ , $B'_{i}=\set{b'_{i,1}, \ldots, b'_{i,\mu_i}}\subset \vspan{B'}$ be a basis of $\cl A_{\xi_i}$ and  $\bm u_i\in \vspan{B}$
 be the idempotent of $\cl A = \cl A_{\varphi}= \cl A_{\varphi'}$ associated with $\xi_i$ (see \cref{eq:idem}). 
Then by \cite[Lemma 3.16]{Mourrain2017}, 
$\bm u_i \star \varphi  = \omega_i^{d, \ell_i, \xaff} \, \eval_{\xi_i}\in \dual{\cl A_{\xi_i}}$ 
(resp.
$\bm u_i \star \varphi' = ({\omega}'_i)^{d,\ell_i,\xaff}\,  \eval_{\xi_i}\in \dual{\cl A_{\xi_i}}$).
For every $j \in \{1, \ldots, \mu_i\}$ we have
$$
\apply{\bm u_i\star \varphi}{b'_{i,j}} = \apply{\varphi}{\bm u_i \, b'_{i,j}} = \apply{\check{f}}{\bm u_i \, b'_{i,j} } = \apply{\varphi'}{\bm u_i\, b'_{i,j} } = \apply{\bm u_i\star \varphi'}{b'_{i,j}},
$$
thus $\bm u_i \star \varphi = \bm u_i \star \varphi' $ and ${\omega}_i^{d,\ell_i,\xaff} = ({\omega}'_i)^{d, \ell_i, \xaff}$.
This shows that the minimal GAD is unique, up to the permutation used above to associate the corresponding supports.
\end{proof}

% Given an ideal $I$ of $R$, we define its homogenized ideal  $\underline{I}= \vspan{ \mathfrak{h}_{x_0, k}(p), p\in I, \deg(p) = k }\subset \cl S$.  The length of an homogeneous ideal $\underline{I} \subset \cl S$ is $\mathfrak{l}(\underline{I})= \dim \cl S_e/\underline{I}_e$ for $e\gg 0$.

\subsection{Cactus rank}

In this section, we consider projective schemes that are apolar to $f$ and show that, under the same regularity assumption of \Cref{thm:rk gad}, the minimal (by length) of such schemes is unique and of the form $I_{\phiprj}$, where the operator $\phiprj \in \S^*$ arises from the minimal GAD of $f$.
\begin{definition}[Cactus rank]\label{def:cactus}
    We define the \emph{cactus rank} $\cactusrank(f)$ of $f \in \cl S_d$ as the minimal length $\mf l(\underline{I})$
    of an apolar scheme $\Iprj\subset \cl S$, i.e. $f^{*} \in \underline{I}_d^{\perp}$.
    Such an apolar scheme $\underline{I}$ of minimal length is called a \emph{minimal apolar scheme}.
\end{definition}
The name \emph{cactus} was introduced in \cite{Buczyska2013} and \cite{Ranestad2011},
%because of the geometric complexity of the underlying algebraic variety of forms with a fixed cactus rank.
and was previously known in the literature as \emph{minimal scheme length} \cite{Iarrobino1999}.
By \Cref{lem:ann w exp}, we always have $\cactusrank(f) \le \gadrank(f)$.
We show hereafter that these ranks coincide under the regularity condition of \Cref{thm:rk gad}.

\begin{theorem}\label{thm:rk cactus}
Let $1 \leq c \leq d$. If $\reg(\S/I_{\phiprj}) \le \min (d-c, c-1)$, then 
$$
 \cactusrank(f) = \rank \bm H_{f^*}^{d-c,c},
$$
and the scheme $I_{\phiprj}$ arising from the minimal GAD of $f$ is its unique minimal apolar scheme.
\end{theorem}
\begin{proof}
Consider the ideal $I_{\phiprj}$, which is the homogenization of $I_{\varphi}$ by \Cref{lem:Iphiprj}. Its length is $\mathfrak{l}(I_{\phiprj})= \dim R/ I_{\varphi} = r$ where $r := \rank \bm H^{d-c,c}_{f^*}$.
By \Cref{lem:ann w exp}, $I_{\phiprj}$ is apolar to $f$. 
Thus $f^*\in (I_{\phiprj})_{d}^{\perp}$ and 
$r_c := \cactusrank(f)\le \mathfrak{l}(I_{\phiprj})=r$.

We now prove that $r_{c}=r$. Let $\underline{I}$ be a minimal apolar scheme to $f$. 
By \cite[Proposition 2.3]{Buczyska2013}, $\underline{I}$ is zero-dimensional Gorenstein. 
By a generic change of coordinates, we can assume that $x_0 \in \cl S_{1}$ is a non-zero-divisor in $\cl S/\underline{I}$ so that, for $e \ge \reg(\S/\underline{I})$, the multiplication map $\bm M_{x_0} : \S_e/\underline{I}_e \to \S_{e+1}/\underline{I}_{e+1}, \, p \mapsto x_0 p$ is bijective.
By \Cref{lem:basisloc},  $\cl S_e/\underline{I}_e\sim R/I$
where $I$ is the dehomogenization of $\underline{I}$ obtained by substituting $x_0$ by $1$.
By \cite[Chap. 21]{eisenbud1995commutative} and \cite[Definition 2.23]{Mourrain2017}, 
% Then we have $\cl V(\underline{I})=\set{\xiprj'_1, \ldots, \xiprj'_{s}}\subset \PP^{n}$ with $\xiprj'_{i}=(1, \xi'_{i,1}, \ldots, \xi'_{i,n})$ and $\cl V(I)=\set{\xi'_1,\ldots, \xi'_s} \subset \kk^{n}$ where $\xi'_i=(\xi'_{i,1}, \ldots, \xi'_{i,n})$.
%
$\underline{I}$ Gorenstein implies that 
$$
I^{\perp} = \invsys{\varphi'} = \set{ \theta\star \varphi', \theta\in R},
$$
where $I$ is the localization at $x_0$ of $\underline{I}$ and $\varphi'\in \kk[[\zaff]]= \dual{R}$.
Then $f^*\in \underline{I}_d^{\perp}$ implies that $\check{f} \in I_{\le d}^{\perp} = (I^{\perp})^{[\le d]}$ and there exists $\theta_f\in R$ such that $\check{f}= (\theta_f \star \varphi')^{[\le d]}$.

By \Cref{thm:kronecker}, \Cref{lem:hankel mult} and \cref{eq:restrict}, we have
$$\rank \bm H_{\varphi'}=r_c\ge \rank \bm H_{\theta_f\star \varphi'} 
\ge \rank \bm H_{\theta_f\star \varphi'}^{A',A} = \rank \bm H_{\check{f}}^{A',A} = \rank\bm H_{{f}^*}^{d-c,c} = r,
$$
where $A$ (resp. $A'$) is a basis of $R_{\le c}$ (resp. $R_{\brev{}\leq\erev{} d-c}$).
This proves that $r_c=r$.

We conclude using \Cref{lem:inclusion} and \Cref{lem:hankel mult} that $I_{\varphi} = I_{\theta_f\star \varphi'} = I_{\varphi'}$, so that 
$\underline{I}=I_{\phiprj}$ is the unique scheme  apolar to $f$ of minimal length $r_c=r$.
\end{proof}

\Cref{thm:rk cactus} can also be proved by noticing that the scheme $I_{\phiprj}$ defined by the unique minimal GAD of $f$ is a \emph{tight annihilating scheme} for $f$ (see \cite[Definition 5.1]{Iarrobino1999}), by \Cref{thm:rk gad}, and observing that $\reg(\S/I_{\phiprj}) \le \min (d-c, c-1)$ implies $2\,\reg(\S/I_{\phiprj}) +1 \leq d$, hence this minimal scheme is unique by \cite[Theorem 5.3, E.]{Iarrobino1999}.
%, using a different approach and different hypotheses. It is assumed there, that $2\, \reg(\S/\underline{I}_Z)+ 1 \le d$ where $\underline{I}_Z$ is a tight annihilating scheme. 
%\Cref{thm:rk cactus} provides an explicit form for the minimal apolar scheme and shows that if $\reg(\S/I_{\phiprj}) \le \min (d-c, c-1)$ for some $1\le c\le d$,  then $I_{\phiprj}$ is the unique minimal apolar scheme.
We will see in \Cref{sec:expl jarek}, an example where the regularity hypothesis is not satisfied
%hypothesis of \Cref{thm:hankel tnf} are not satisfied
and $\cactusrank(f) < \gadrank(f)$.

\brev{}\begin{remark} 
    \Cref{thm:rk cactus,thm:rk gad} establish a sufficient condition for the uniqueness of minimal GADs and apolar schemes to a given $f$.
    This condition is not necessary in general (see, e.g., \Cref{ex:notconverse}).
    %, although in many cases uniqueness does not hold for higher regularity (\Cref{ex:dWaring,ex:dTangential}).
\end{remark}\erev{}

\section{Algorithm and experimentation} \label{sec:AlgoExp}

We present now an algorithm, based on \Cref{thm:hankel tnf}, to compute the unique minimal GAD of a symmetric tensor \( f \in \Sd \) prescribed by \Cref{thm:rk gad}, under its regularity assumptions.

\subsection{Decomposition Algorithm}
The algorithm proceeds as follows:
\begin{algorithm}[H]
\caption{\texttt{gad\_decompose}\label{algo:gad}}
\begin{algorithmic}[1]

\Function{gad\_decompose}{$f$}
    \State \textbf{Input}: $f\in \S_d$
    \State Perform a random change of coordinates to obtain $\xiprj_{i,0} \neq 0$
    \State Compute a basis $B$ via SVD
    \State Build multiplication matrices $M = \{M_{x_1},\dots,M_{x_n}\}$ (see \texttt{mult\_matrices}) 
    \State Compute multiplicities $\mu_i$ of the points $\xi_i$ (see \texttt{multiplicities})
    \State Compute local blocks of multiplication $M_{x_j}^{(i)}$(see \texttt{local\_mult})
    \State Extract point coordinates $\xi_{i,j}=\frac{1}{\mu_i} \textnormal{Trace}(\tilde{M}_{x_j}^{(i)})$ and define $\ell_i = x_0 + x_1 \xi_{i,1} + \dots + x_n \xi_{i,n}$ 
    \State Compute nil-indices $\nu_i = k_i + 1$ from the nilpotences of $\{ x_j - \xi_{i,j} \}_{1 \leq j \leq n}$ 
    (see \texttt{nil\_index})
    \If { $k_i > d $}
        \State \textbf{error:} a nilpotency index exceeds the degree bound
        \EndIf
%    \State Deduce degrees $k_i$ of the components $\omega_i$
    \State Compute $\omega_i\in \cl S_{k_i}$ by solving the Vandermonde-like system \( f = \sum_i \omega_i \ell_i^{d-k_i} \), using the least-squares method, in the unknown coefficients of $\omega_i$.
    \State \textbf{Output}: $\ell_i, \omega_i, \mu_i$ such that $f = \sum_i \omega_i \ell_i^{d-k_i}$ and $\gadrank(f) = \sum_i \mu_i$.
\EndFunction
\end{algorithmic}
\end{algorithm}

\noindent{} \textbf{Construction of multiplication matrices \( M_{x_j} \) (\texttt{mult\_matrices}):} for each variable \( x_j \in \{x_0, \dots, x_n\} \):
    \begin{itemize}
        \item Construct Hankel matrices $H_j = H^{A', x_j A}$, where $A'$ is the set of monomials up to degree $\lfloor \frac{d-1}{2} \rfloor$ and $A$ is the set of monomials up to degree $d- \lfloor \frac{d-1}{2} \rfloor-1$. 
        \item Compute the linear combination
       $\underline{H} = \sum_{j=0}^n \lambda_j \, H_j$,
        where $\lambda_j$ are sampled from the standard normal probability distribution.
        \item Perform an SVD on $\underline{H}$ to obtain bases $B'$ and $B$ of the row and column polynomial spaces modulo $I_\varphi$, respectively, and compute $M_{x_j}= (H_0^{B',B})^{-1} H_j^{B',B}$.
%Perform an SVD on $\underline{H}$ and compute matrices $M_{x_j}= (H_0^{B,B})^{-1} H_j^{B,B}$.
    \end{itemize}

\noindent{} \textbf{Extraction of multiplicities (\texttt{multiplicities}):}
    to compute the multiplicities $\mu_i$:
    \begin{itemize}
        \item Compute the linear combination:
            $\underline{M} = \sum_{j=1}^n \lambda_j \, M_{x_j}$, 
        where $\lambda_j$ are sampled from the standard normal distribution.
        \item Compute the eigenvalues of $\underline{M}$ by performing Schur factorization $\underline{M} = U \, T \, U^{-1}$.
        \item Apply clustering techniques to compute the multiplicities $\mu_i$ of its eigenvalues, that is the multiplicities $\dim \cl A_{\xi_i}$ of the points $\xi_i$.
    \end{itemize}
 
\noindent{} \textbf{Block decomposition (\texttt{local\_mult}):} to compute the local sub-blocks of the multiplication operators:
    \begin{itemize}
        \item Using the Schur basis, compute the matrices $\tilde{M}_{x_j} = U^{-1} \, M_{x_j} \, U$, for $j=1, \dots, n$.
        \item Extract the diagonal blocks $\tilde{M}_{x_j}^{(i)}$ of size $\mu_i \times \mu_i$, for $i=1,\dots,s$.
    \end{itemize}

%\noindent{} (4) \textbf{Recovery of decomposition points:} %The coordinates of the points $\xiprj_i := \frac{\xi_i}{\underline{x}}$, where $\underline{x}$ is a random linear combination of the $x_i$'s, 
%The coordinates of the points $\xiprj_i$ are obtained directly from the diagonal entries of the local matrix blocks.
%\todo[inline]{not correct. The trace is used. State it as something like "compute $\xi_{i,j} = \frac{1} {\mu_i} trace(M_j^{(i)})$". This can be put directly in the algo and you can remove $(4)$}
    
%\noindent{}(5) \textbf{Construction of linear forms:} For each point \( \xiprj_i \), compute the linear form \( \ell_i = (\xiprj_i, X) \).
    
\noindent{} \textbf{Computation of component degrees (\texttt{nil\_index}):} Compute the nil-indices $\nu_i$ as the minimal integers $m \in \mathbb{N}$ such that $(\tilde{M}_{x_j}^{(i)} - \xi_{i,j} \, Id)^m = 0$ for each $j \in \{1, \dots, n\}$, then set $k_i=\nu_i-1$.

%The degrees \( k_i =\nu_i-1\) of the  components \( \omega_i \) are determined from the nilpotency $\nu_i= \mathcal{N}(\mathcal{A}_{\xi_i})$ of the blocks \( \tilde{M}_{x_j}^ {(i)}\).  
    
% \noindent{} (5) \textbf{System resolution:} The weights \(\omega_i\) are finally computed by solving the Vandermonde-like linear system \( f = \sum_i \ell_i^{d-k_i} \omega_i \), in the unknown coefficients of $\omega_i$.

\vspace{1em}

\begin{theorem} \label{thm:algoworks}
If $f\in \Sd$ has a GAD $f=\sum_{i=1}^{s} \omega_{i}\, \ell_i^{d-k_i}$ such that $\reg(\S/I_{\phiprj})\le \min (d-c, c-1)$, 
%where $\phiprj = \sum_{i=1}^{s} \omega_{i}^{d,\ell_i,\xaff}(\zprj)\, \eval_{\ell_i}(\zprj)$, 
then \Cref{algo:gad} computes this GAD, which is the unique minimal GAD of $f$.
\end{theorem}
\begin{proof}
By a generic change of variables and scaling, we can assume that $\xiprj_{i,0}=1$. By \Cref{lem:polexp trunc}
$f^* = \phiprj^{[d]}$ where $\phiprj=\sum_{i=1}^{s} \omega_i^{d,\ell_i, \xaff}\eval_{\ell_i}$ and
$\check{f} = (\varphi)^{[\le d]}$ 
where $\varphi= \brev{}\sum_{i=1}^{s}\erev{}\omega_i^{d,\ell_i,\xaff}(\zaff)\,\eval_{\xi_i}(\zaff)$ as in \cref{eq:extendedphi}. The assumption $\reg(\S/I_{\phiprj})\le \min (d-c, c-1)$ implies (\Cref{lem:Iphiprj}) that there exist sets $A \subset R_{\le c}$, $A' \subset R_{\le d-c}$ containing bases $B, B'$ of $\cl A_{\varphi}$, such that the hypotheses of \Cref{thm:hankel tnf} hold for the truncated Hankel operator $\bm H := \bm H_{\check{f}}^{A',A}$. In particular, for $H_0 := H_{\check{f}}^{B',B}$ and $H_j := H_{\check{f}}^{B',x_j\, B}$, we have
$\rank H_0= \rank(H) = \dim \cl A_{\varphi}$, and $M_{x_j}= H_0^{-1}H_j$ is the matrix of the multiplication by $x_j$ in the basis $B$ of $\cl A_{\varphi}$ by \Cref{lem:hankel mult}.  
\Cref{algo:gad} builds $H$ (step 3-4) and computes bases $B,B'$. % using SVD. 
By \Cref{prop:3.20}, for a generic linear combination $\underline{M} = \sum_j \lambda_j M_{x_j}$, the Schur factorization produces blocks of size $\mu_i = \dim \cl A_{\xi_i}$, corresponding to the multiple eigenvalues associated with the point $\xi_i$.
For each $x_j$, the blocks $\tilde{M}_{x_j}^{(i)}$ represent the matrix of multiplication by $x_j$ in a basis
of the local algebra $\cl A_{\xi_i}$. 
Its trace is $\mu_i \xi_{i,j}$, which provides the coordinates of the points $\xi_i$ corresponding to the linear forms $\ell_i$.
Within each local algebra $\cl A_{\xi_i}$, the operators of multiplication by $(x_j - \xi_{i,j})$ are nilpotent, and the initial generic change of coordinates ensures that these operators attain the nil-index of each local algebra.
These nil-indices $\cl N(\cl A_{\xi_i})$ determine the degree of the dual polynomial $\cl N(\cl A_{\xi_i}) - 1 = \deg(\omega_i) = k_i$, by \Cref{prop:nilindex-socle}.
By computing the nil-indices from the blocks $\tilde{M}_{x_j}^{(i)}$, \Cref{algo:gad} recovers all the degrees $k_i$ of the weights $\omega_i$.
Once $\ell_i$ and $k_i$ are known, the unknown coefficients of $\omega_i$ are determined from the relation $\check{f} = (\sum_{i=1}^s \omega_i^{d,\ell_i,\xaff} \, \eval_{\xi_i})^{[\le d]}$, by solving a linear system of Vandermonde type. 
By \Cref{thm:rk gad}, this GAD exists and is unique, which ensures that the system is solvable and recovers the $\omega_i$ of such a decomposition.
%By \Cref{lem:Iphiprj}, $I_{\underline{\varphi}}$ is the homogenization of $I_{\varphi}$. Under the regularity assumption, \Cref{thm:rk gad} proves that the rank of the Hankel operator equals the GAD-rank, and that the decomposition is minimal.
%Moreover, \Cref{thm:rk cactus} shows that the apolar scheme $I_{\underline{\varphi}}$ is unique, hence the associated GAD is unique.
\end{proof}

\brev{}\begin{example} \label{ex:dWaring}
    Let $n=2$, $d \geq 2$ and consider $f = x_0^d + x_1^d + x_2^d$.
    The apolar scheme associated with the Waring-GAD is defined by $I_{\phiprj} = (x_0x_1, x_0x_2, x_1x_2)$, hence $\reg(\S/I_{\phiprj}) = 1$.
    When $d \geq 3$, the assumptions of \Cref{thm:rk cactus,thm:rk gad} are satisfied with $c=2$, therefore \Cref{thm:algoworks} ensures that \Cref{algo:gad} recovers the unique minimal GAD, which is the unique minimal Waring decomposition of $f$.
    For $d = 2$, our theorems do not apply. In this case, the Waring and cactus ranks are both $3$, but the minimal GADs (and Waring decompositions) are no longer unique, since $x_0^2 + x_1^2 + x_2^2 = \frac{1}{2}(x_0+x_1)^2 + \frac{1}{2}(x_0-x_1)^2 + x_2^2$.
\end{example}\erev{}

\brev{}\begin{example} \label{ex:dTangential}
    Let $n=1$, $d \geq 2$ and $f = x_0^{d-1} x_1$.
    This is a tangential decomposition \cite{Bernardi20201} $f=\omega \ell^{d-1}$, with $\omega = x_1$ and $\ell = x_0$, and $\invsys{x_1}^{d, x_0, \bm x}= \vspan{1, x_1}$.
    For $d\ge 3$, its GAD-rank is $2$ as well as its cactus rank \cite[Corollary 2]{Ranestad2011}, as we have $I_{\phiprj} = (x_1^2)$, while its Waring rank is $d$ \cite{BBT13,CCG12}.
    Since $\reg(\S/I_{\phiprj}) = 1$, \Cref{algo:gad} recovers the unique minimal non-Waring GAD whenever $d \geq 3$.
    Similarly to \Cref{ex:dWaring}, the GAD-rank does not change for $d = 2$, but the minimal decompositions are not unique: there are $2$ (obvious) local tangential decompositions $(\omega, \ell) = (x_1,x_0)$ and $(\omega, \ell) = (x_0,x_1)$, and infinitely many Waring decompositions $x_0x_1 = \frac{1}{4\alpha\beta}\big( (\alpha x_0 + \beta x_1)^2 - (\alpha x_0 - \beta x_1)^2 \big)$, for $\alpha\beta \neq 0$.
\end{example}\erev{}

\brev{}\begin{example} \label{ex:notconverse}
    Let $n=2$, $d \geq 3$ and $f = x_0^{d-2} x_1 x_2$, whose Waring rank is $2(d-1)$ \cite{BBT13,CCG12} and cactus rank is $4$ \cite[Corollary 2]{Ranestad2011}, with $\invsys{x_1\, x_2}^{d, x_0, \bm x}=\vspan{1, x_1, x_2, x_1\,x_2}$ and $I_{\phiprj} = (x_1^2,x_2^2)$.
    Since $\reg(\S/I_{\phiprj}) = 2$, \Cref{algo:gad} recovers the unique minimal GAD whenever $d \geq 5$.
    For $d=4$, we still have a unique minimal GAD associated with the unique minimal apolar scheme $I_{\phiprj}$, even though our theorems do not apply.
    To see that, we note that the Hilbert function of such a minimal apolar scheme $I$ needs to be $[1,3,4,4,\dots]$, but $\Ann(f^*) = (x_0^3,x_1^2,x_2^2)$.
    Therefore, every ideal $I \subseteq \Ann(f^*)$ with $\dim I_2 = 6 - 4 = 2$ needs to contain $(x_1^2,x_2^2) = I_{\phiprj}$. 
    As it has the same Hilbert function as $I_{\phiprj}$, it must be equal to $I_{\phiprj}$.
    An analogous reasoning shows that, for $d = 3$, there are infinitely many minimal GADs (and apolar schemes $I_{\phiprj}$), parametrized by $\textnormal{Gr}(2,\vspan{x_0^2,x_1^2,x_2^2}) \simeq \PP^2$.
\end{example}\erev{}

\subsection{Computational testing}
We now present the experimental part of this work.
%All experiments are conducted in Julia 1.11.6 using the package \href{https://github.com/enricabarrilli/GAD.jl}{\texttt{GAD.jl}}, which contains the examples presented hereafter.

\brev{}All experiments are conducted in Julia 1.11.6 using the package \href{https://github.com/enricabarrilli/GAD.jl}
{\texttt{GAD.jl}}\footnote{See also its documentation at \href{https://github.com/enricabarrilli/GAD.jl}{https://github.com/enricabarrilli/GAD.jl}.}, which contains the examples presented hereafter.\erev{}

\subsubsection{Waring decomposition}
We consider the following example to test the numerical accuracy of the \texttt{gad\_decompose} algorithm in the Waring case (simple points). 
The goal is to write a symmetric tensor \( f \in \mathcal{S}_4(\mathbb{C}^3) \) as a sum of powers of linear forms, verifying the reconstruction accuracy.

We fix \( n = 2 \), \( d = 4 \), the variables \( x_0, x_1, x_2 \), and consider the tensor \( f \) defined by
\[
f = -3(x_0 + 2x_1 + 2x_2)^4 + (x_0 + x_1 + x_2)^4 + (x_0 + 3x_1 - x_2)^4.
\]
We work in the localization $\S_{x_0^*}$, and take $A' = \{1, x_1, x_2\}$ and $A = \{1, x_1, x_2, x_1^2, x_1x_2, x_2^2\}$, we compute $H = H_{\check{f}}^{A',A}$ of size $3 \times 6$, a submatrix of the catalecticant $H_{\check{f}}^{1,3}$.
We deduce the rank $r=3$ from the SVD of $\underline{H} = \sum_{j=0}^2\lambda_j H^{A',x_j A}$, and compute the multiplication matrices $M_{x_j}$ of size $3 \times 3$.
We then compute a Schur factorization $\underline{M} = Q TQ^t$ of $\underline{M} = \sum_{j=1}^2 \lambda_j M_{x_j}$. From $Q^t M_{x_j} Q$, deduce blocks $M_{x_j}^{(i)}$ of size $1 \times 1$ of local multiplication and we deduce the multiplicities $\mu = [1,1,1]$. The recovered associated linear forms are  
$
  \ell_1 = - 0.1744x_0 - 0.3488x_1 - 0.3488x_2, \,
  \ell_2 = - 0.3255x_0 - 0.3255x_1 - 0.3255x_2, \,
  \ell_3 = - 0.1399x_0 - 0.4196x_1 + 0.1399x_2 $.

We compute the nil-indices [1,1,1], deduce the degrees $[0,0,0]$ of the weights $w_1$, $w_2$ and $w_3$ and solve $f = w_1 \ell_1^4 + w_2 \ell_2^4 + w_3 \ell_3^4$ to get the associated coefficients $w_1 = -3242.32, w_2 = 89.06, w_3 = 2611.85.$

The reconstructed tensor is given by $T = \sum_{i=1}^3 w_i \ell_i^4$,
and the error measured in the Euclidean norm is: $\|f - T\| \approx 9.70 \times 10^{-13}$.
The reconstruction error is consistent with double-precision floating-point arithmetic accuracy.
This example supports the efficiency and numerical reliability of the \texttt{gad\_decompose} algorithm for computing Waring decompositions of symmetric tensors, under the regularity assumption.

\subsubsection{Multiple points case}
We consider the following example to test the numerical accuracy of the \texttt{gad\_decompose} algorithm in the case of a GAD involving multiplicities greater than one. 

Fix \( n = 2 \), \( d = 5 \), and consider the tensor \( f \in \S_5 \) defined by
\[
f = x_0^3 x_1 x_2 + (x_0 + \frac 1 2 x_1 + 2 x_2)^4 (x_0 + x_2).
\]
We work in the localization $\S_{x_0^*}$, take $A' = A = \{1, x_1,x_2, x_1^2, x_1 x_2, x_2^2\}$, and compute $H = H_{\check{f}}^{A',A}$ of size $6 \times 6$, a submatrix of $H_{\check{f}}^{2,3}$.
We compute the rank from the SVD of $\underline{H} = \sum_{j=0}^2 \lambda_j H^{A', x_j A}$ and get $r=6$.
We deduce the multiplication matrices $M_{x_j}$ of size $6 \times 6$.
We compute a Schur factorization $\underline{M} = QTQ^t$ of $\underline{M} = \sum_{j=1}^2 \lambda_j M_{x_j}$. From $Q^t M_{x_j} Q$, we find the blocks $M_{x_j}^{(i)}$ of local multiplication of size $2 \times 2$ and $4 \times 4$, hence deduce the multiplicities $\mu = [2,4]$.
The associated linear forms are
$ \ell_1 = 13.21x_0 + 6.61x_1 + 26.43x_2, \,
  \ell_2 = - 3.62x_0 - 2.04 \times 10^{-15}x_1 - 3.76 \times 10^{-14}x_2$.
We find the nil-indices $2$ and $3$, deduce the degrees $1$ and $2$ of $w_1$ and $w_2$, and solve $f = w_1 \ell_1^4 + w_2 \ell_2^3$ to get $ w_1 = 3.2810 \times 10^{-5}(x_0 + x_2) - 1.7949 \times 10^{-19}x_1, \,
  w_2 = -1.5305 \times 10^{-15}x_2^2 - 0.0211x_1 x_2 - 8.6353 \times 10^{-18}x_1^2 - 7.7109 \times 10^{-16}x_0 x_2 - 9.0234 \times 10^{-17}x_0 x_1 - 1.2588 \times 10^{-16}x_0^2$.
The reconstructed tensor is $T = \sum_{i=1}^2 w_i \, \ell_i^{d-\deg(w_i)}$, while the Euclidean norm of the reconstruction error is $\|f - T\| \approx 7.60 \times 10^{-14}$.
The reconstruction error supports the effectiveness of the algorithm in decomposing symmetric tensors even in the presence of multiple points.

\subsubsection{Random Tests}

Hereafter, we provide a first coarse analysis of the numerical stability of the algorithm \texttt{gad\_decompose}, under small random perturbations of random GADs of a given format, leaving deeper investigations for future work. In the experiments, we use agglomerative hierarchical clustering (AHC) with single linkage, based on the Euclidean distance between samples, following the classical scheme introduced by Johnson \cite{johnson1967hierarchical} (implemented in the package {\texttt{Clustering.jl}}), and supply as input the rank and the number of clusters. 
\brev{}In the general version of the algorithm, however, these are recovered automatically: the rank from a truncated SVD of $\underline{H}$ with threshold \(10^{-3}\), and the clusters by grouping consecutive vectors whose Euclidean distance is at most $10^{-2}$.
\erev{}
%For the clustering of points, we use agglomerative hierarchical clustering (AHC) with single linkage, based on the Euclidean distance between samples, following the classical scheme introduced by Johnson \cite{johnson1967hierarchical}, implemented in the package {\texttt{Clustering.jl}}. In these experiments we give in input the rank and the number of clusters.
%\vspace{1em}

%We perform a series of randomized numerical experiments designed to probe the sensitivity of the decomposition to small perturbations. 
For a fixed decomposition type specified by parameters \((n, d, \mathbf{k})\), where \(\mathbf{k}\) is a vector of the degrees \(k_i\) of the coefficient polynomials \(\omega_i\), we construct a base tensor \(f_0 \in \mathcal{S}_d(\mathbb{R}^{n+1})\) having a GAD of the form $f_0 = \sum_{i=1}^s \omega_i \ell_i^{d-k_i},$
where each \(\omega_i\) is a polynomial of degree \(k_i\), and \(\ell_i = (\xiprj_i, \xprj)\) is a general linear form with fixed leading coefficient \(\xi_{i,0} = 1\). 

We then perturb this tensor as follows:
$f = f_0 + \varepsilon {\brev{}f_1\erev{}}$,
where ${\brev{}f_1\erev{}}$ is a random degree-\(d\) homogeneous polynomial in $n+1$ variables, with coefficients drawn from a standard normal distribution and normalized with respect to the apolar norm, i.e., \(\|{\brev{}f_1\erev{}}\|_a = 1\). 

The resulting tensor \(f\) does not generally lie on the same osculating variety, and the goal is to use \texttt{gad\_decompose} to recover an approximate decomposition and measure the relative reconstruction error
\(
\delta(\varepsilon) := \frac{\|T - f\|_a}{\|f\|_a},
\)
where \(T\) denotes the tensor reconstructed from the output decomposition.

%\vspace{1em}

The perturbation magnitude $\varepsilon$ is sampled logarithmically over the range $\varepsilon \in [10^{-14}, 10^0]$, with steps of 0.5 in the exponent, producing 29 perturbation levels. 

For each base tensor $f_0$ and each perturbation level $\varepsilon$, we generate 10 random perturbations. The perturbed tensors are decomposed using the \texttt{gad\_decompose} algorithm, and for each case we compute the relative reconstruction error. 

%Decompositions that fail are excluded by removing  NaN values from the data vectors. \todo{How often does it fail?}
For each perturbation level $\varepsilon$, we then report the median relative reconstruction error over 10 trials removing the NaN values (blue plot), together with the worst-case (maximum red curve) and best-case (minimum green curve) reconstruction errors.

\vspace{1em}

%In the regime of small perturbations, the expected behavior is: \(\delta(\varepsilon) \approx \kappa\, \varepsilon,\) for some $\kappa >0$, which corresponds to a straight line of slope 1 in log-log coordinates. Deviations from this linear behavior can therefore be interpreted as breakdowns in local stability, sensitivity to multiplicities, or numerical limitations.

%\vspace{1em}

%Geometrically, the perturbation moves the tensor \(f_0\) off the model variety. The decomposition algorithm attempts to project the perturbed tensor back onto it by identifying approximate points and multiplicities. 

%\vspace{1em}

%We now detail the behavior of the algorithm in specific representative cases:

\begin{figure}[!htbp]
    \centering
    \captionsetup{font=small} 

    \begin{subfigure}[b]{0.45\textwidth}
        \centering
        \includegraphics[scale=0.2]{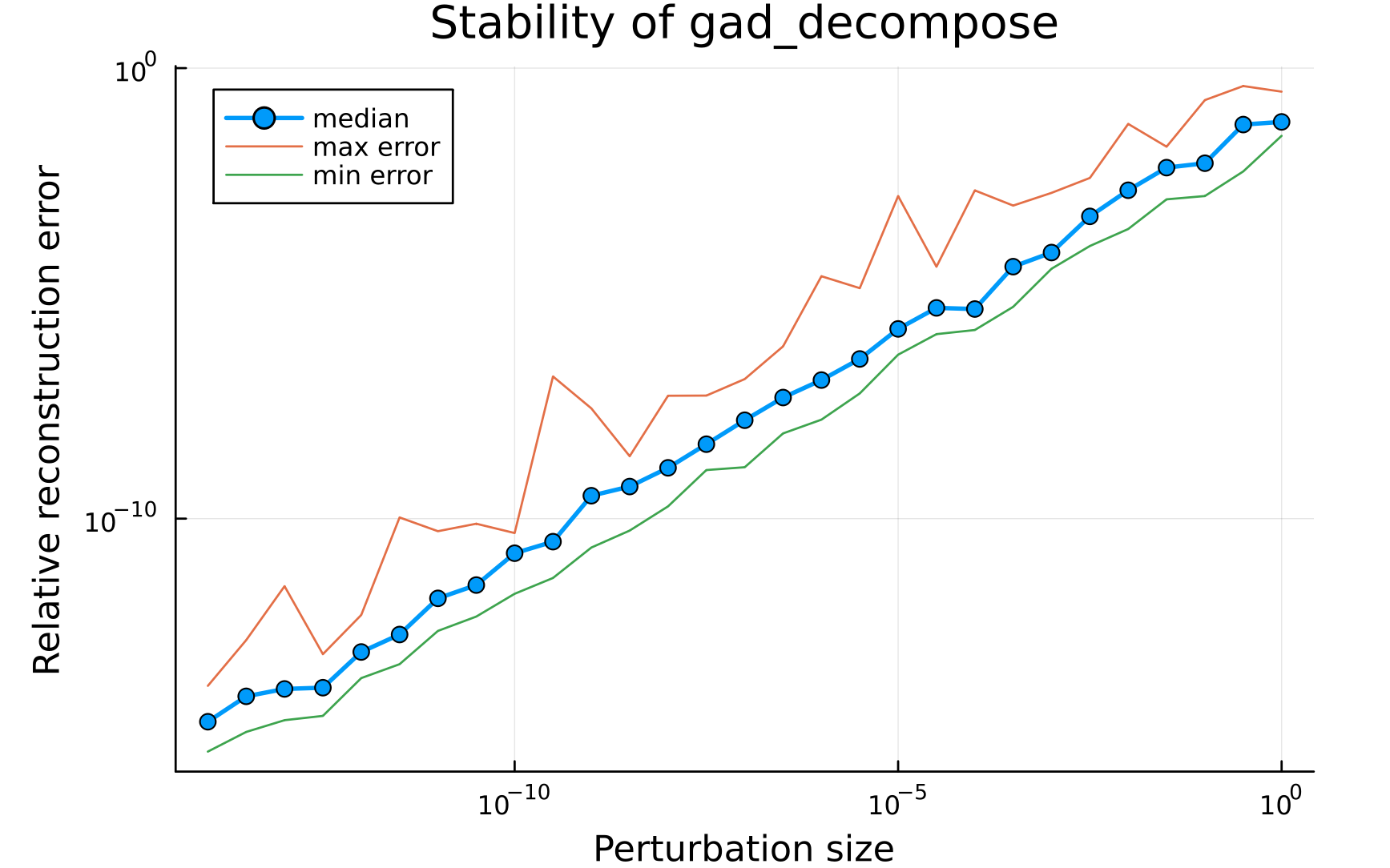}
        \caption{Case \((n, d, \mathbf{k}) = (9, 3, [0,0,0,0,0])\): Waring decomposition - Five simple points.}
        \label{fig:test1}
    \end{subfigure}
    \hfill
    \begin{subfigure}[b]{0.45\textwidth}
        \centering
        \includegraphics[scale=0.2]{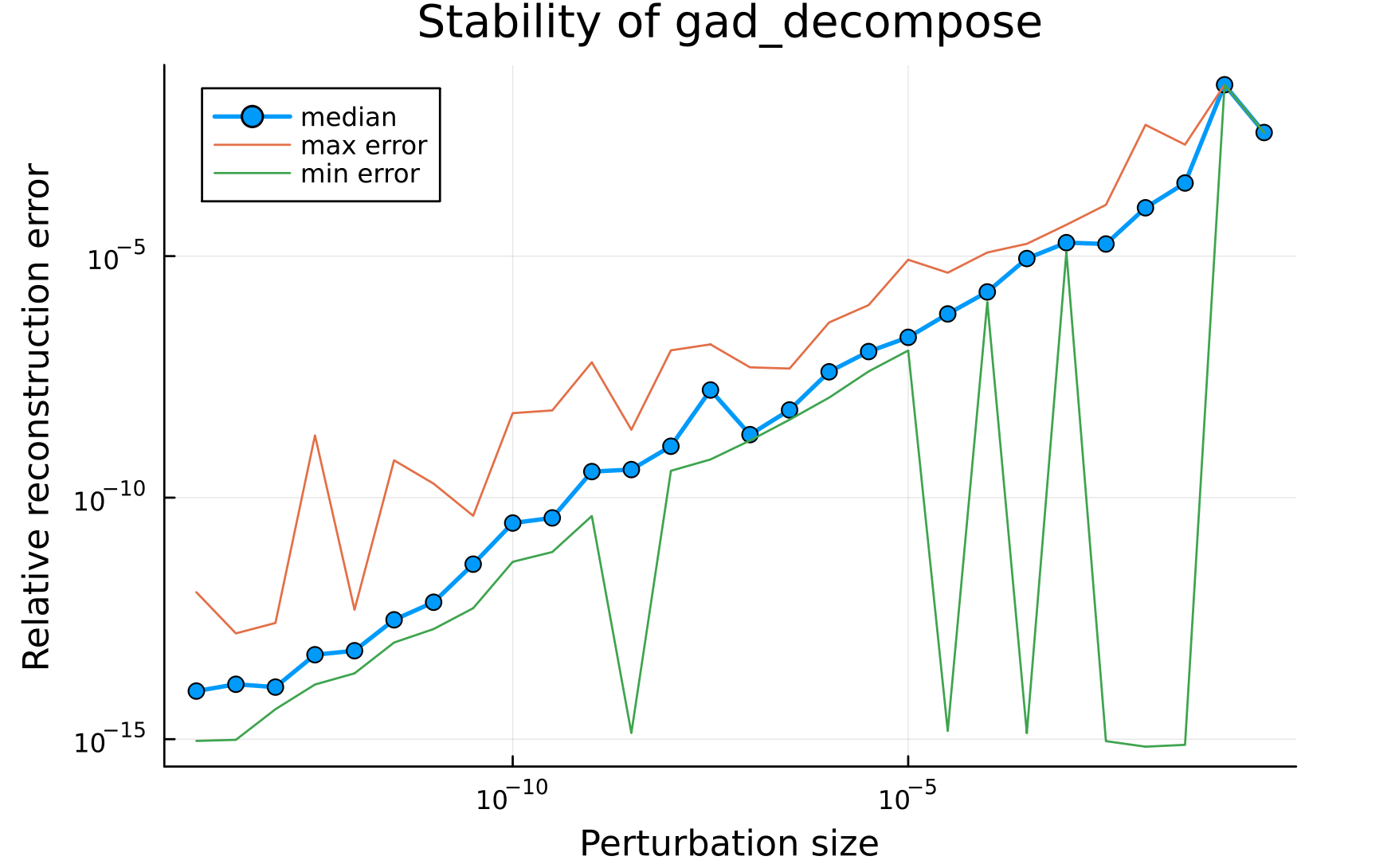}
        \caption{Case \((n, d, \mathbf{k}) = (2, 5, [1,1,0])\): Two points of multiplicity 2 and one simple point.}
        \label{fig:test3}
    \end{subfigure}
\end{figure}

\begin{comment}
    \begin{figure}[!htbp]
    \centering
    \makebox[0pt]{\includegraphics[scale=0.20]{n9 d3 ks00000.png}}
    \caption{\label{fig:test1} Case \((n, d, \mathbf{k}) = (9, 3, [0,0,0,0,0])\): Waring decomposition into five simple points.}
\end{figure}
\end{comment}

In Figure (\subref{fig:test1}), we analyze the perturbations of Waring decompositions of rank $5$ of the form $f = \sum_{i=1}^5 \omega_i \ell_i^3$ with $\omega_i \in \kk$ in the variables $x_0, \ldots, x_9$.
The relative reconstruction error exhibits clean linear growth across nearly the entire range of \(\varepsilon\), exhibiting approximately linear reconstruction-error growth \(\delta(\varepsilon) \approx \kappa\, \varepsilon\) for some $\kappa >0$. The algorithm performs robustly in all trials, with tight bounds between minimum and maximum errors. 

The decomposition process benefits from stable clustering and multiplicity detection in the case of simple points.
In particular, when the eigenvalue vector contains distinct and well-separated values, the clustering function identifies groups more efficiently, as the pairwise distance matrix displays a clear separation between clusters. This leads to a stable clustering structure.
Conversely, when many input values are very close or nearly identical, the algorithm must evaluate many small, similar distances, making cluster formation less clear and potentially leading to inefficient or unstable merges.
The experiment therefore suggests that \texttt{gad\_decompose} maintains stable and predictable performance in the simple-point regime, with reduced sensitivity to random fluctuations.

\begin{comment}
    
\paragraph{Test 2: One point of multiplicity 5 and one simple point}
\begin{figure}[!htbp]
    \centering
    \makebox[0pt]{\includegraphics[scale=0.2]{n3 d5 ks20 2.png}}
    \caption{\label{fig:test2} Case \((n, d, \mathbf{k}) = (3, 5, [2,0])\): One point of multiplicity 5 and one of multiplicity 1.}
\end{figure}

This case includes a higher multiplicity component. In Figure~\ref{fig:test2}, the linear trend persists in the low-noise regime: the reconstruction error remains low and increases smoothly with $\varepsilon$ up to about 
$10^{-5}$, indicating good stability in that regime. However, for larger perturbations, there is a sharp increase and visible irregularities in the error. The blue median curve also exhibits a spike, likely caused by occasional failures in correctly recovering the decomposition when the noise becomes too large. The algorithm still recovers the decomposition in most trials, though occasional outliers emerge.

\end{comment} 

In Figure (\subref{fig:test3}), we consider random perturbations of a GAD of the form $f = \omega_1 \ell_1^4 + \omega_2 \ell_2^4 + \omega_3 \ell_3^{5}$, where $\deg (\omega_1)=\deg(\omega_2)= 1$ and $\deg(\omega_3)=0$ in the variables $x_0, x_1, x_2$.

Here we also observe that the reconstruction error exhibits a linear scaling with respect to the perturbation size, but with some irregularities. 
The median curve grows almost linearly on the log–log scale, suggesting that the algorithm maintains stable behavior on average. 
However, the maximum error shows pronounced fluctuations, with occasional peaks that deviate from the general trend.
These outliers are attributable to instabilities in internal steps of the decomposition, such as the estimation of multiplicities and the computation of nil-indices, which are sensitive to randomization and conditioning.
The minimum error curve, although more stable, also displays oscillations in the large-perturbation regime, suggesting that certain perturbations may produce configurations with a very low relative reconstruction error. 
%that are either particularly well or poorly conditioned. 
Overall, this experiment shows that \texttt{gad\_decompose} is stable in median performance, but can produce sporadic instabilities depending on the specific realization of the perturbation.

\subsubsection{Nil-index obstruction}\label{sec:expl jarek}

We consider the form \( f \in \S_3(\C^6) \) from (\cite[Example 4.6]{Bernardi2024}):
\begin{align*}
f = {} & 24\, x_0^3 + 70\, x_0^2 x_1 + 75\, x_0^2 x_2 + 70\, x_0^2 x_3 + 180\, x_0^2 x_4 + 10\, x_0^2 x_5 + 10\, x_0 x_1^2 \\ 
& + 70\, x_0 x_2^2 + 360\, x_0 x_2 x_3 
       + 120\, x_0 x_2 x_4 + 60\, x_0 x_3^2 + 60\, x_2^3 + 60\, x_2^2 x_3
\end{align*}
Applying \Cref{algo:gad} with the usual localization at $x_0=1$, and employing $H_{\check{f}}^{1,2}$, we find that $r=6$ and $\vspan{B} = \vspan{B'} = \vspan{1, x_1, \ldots, x_5}$.
We compute the multiplication matrices $M_{x_j}= (H_{\check{f}}^{B,B})^{-1} H_{\check{f}}^{B,x_j B}$ for $j=1, \ldots, 5$.
Their eigenvalues are (approximately) $0$, thus the apolar scheme should be supported at the single point $\xiprj=(1, 0, \ldots, 0)$ and a GAD of $f$ should be of the form $\omega\, x_0^{3-k}$ with  $0\le k\le 3$ and $\omega \in \Sk$. 
However, \brev{}the maximum nilpotency index of the matrices\erev{} $M_{x_j}$ is $5$, which would imply that $\deg(\omega)=4$. This is a contradiction since $\deg(f)=3$,
and the algorithm stops (at line 10).

% By applying the decomposition algorithm, the reconstruction from the computed decomposition yields a polynomial \( T \) with a large reconstruction error:
% \[
% \|f - T\| \approx 4,4 \times 10^2.
% \]

% This significant error suggests that the internal computation of the \texttt{nilx} parameter is suboptimal for this input. The correct nil-index is 4, but the nil-index computed by the algorithm is $nilx = 5$.

We check that the \brev{}matrices\erev{} $M_{x_j}$ are commuting and of dimension $6 \brev{}\times 6\erev{}$. Thus the associated scheme $\underline{I}$ %$= \set{p \in \S\mid p(\textup{Id}, M_{x_1}, \ldots, M_{x_n})=0}$\todo{explain} of 
has length $\mf l(\underline{I})=6$.
It turns out that 
$\underline{I}$ is generated by $K = \ker H^{1,2}_{f^*}$ (see \cite[Example 4.6]{Bernardi2024}):
\begin{align*}
K  = \big\langle & 
x_1^2 - x_0 x_5, x_1 x_2, x_1 x_3, x_1 x_4, x_1 x_5, -x_0x_3 + x_2^2,
%x_2^2 + \frac{420}{13} x_0x_5 - \frac{21}{13} x_0x_4 - \frac{6}{13} x_0 x_3, \\
 x_2 x_3 - x_0x_ 4, x_2 x_4 - 6 x_0 x_5, \\
 & -6 x_0 x_5 + x_3^2,
%x_2 x_5, x_3^2 - 13\, x_0 x_5, 
x_2 x_5, x_3 x_4, x_3 x_5, x_4^2, x_4x_5, x_5^2 \big\rangle\brev{}. \erev{}
\end{align*}
This minimal apolar scheme is of the form $I_{\phiprj}$ with $\phiprj = g^{4,x_0,\xaff}(\zprj) \eval_{x_0}(\zprj)$ where
\begin{align*}
 g = & \ 24\, x_0^4
 + \frac{280}{3}\,x_0^3 x_1 + 100\,x_0^3 x_2 + \frac {280} 3\,x_0^3 x_3 + 240\,x_0^3 x_4 + \frac{40}{3}\,x_0^3 x_5 + 20\,x_0^2 x_1^2 \\
 & \ + 140\,x_0^2 x_2^2 + 720\,x_0^2 x_2 x_3 + 240\,x_0^2 x_2 x_4 + 120\,x_0^2 x_3^2
 + 240\,x_0 x_2^3 + 240\,x_0 x_2^2 x_3 + 20\, x_2^4.
\end{align*}
We deduce that $\cactusrank(f) = \dim \invsys{g}^{4,x_0,\xaff}=6$.
By analyzing the Hilbert function of $\ann(f^*)$, we conclude that $\gadrank(f) > 6$.
Since $\dim \invsys{{f}}^{3,x_0, \xaff}=7$, we deduce that $\gadrank(f)=7$ and a minimal GAD of $f$ is $f = f\, x_0^0$.
\Cref{thm:rk gad} does not apply since the regularity of the quotient by the ideal $I_{\phiprj}$ associated with this GAD is 
$$\reg(\S/I_{\phiprj}) = 2 > \min \set{3-2,2-1}=1.$$

On the other hand, $\reg (\S/I_{\phiprj}) = 1$ and \Cref{thm:rk gad,thm:rk cactus} apply: $\gadrank(g) = \cactusrank(g) = 6$.
We verify that 
$$
\check{g} = \check{f} + \frac{5}{6} z_2^4,  
$$
so that $\check{g}$ is an extension in degree $4$ of 
$\check{f}$, i.e. $\check{g}^{[\le 3]}= \check{f}$. The cactus rank of $f$ can also be recovered using the extension algorithm in \cite{Bernardi20201}, which generalizes the Waring extension algorithm in \cite{Brachat2010}.

\section*{Conclusions and Limitations}

We provide a novel and global characterization of the zero-dimensional scheme associated with a Generalized Additive Decomposition of a form $f$, whose structure can be explicitly deduced from the GAD.
When the regularity of this apolar scheme is small enough, this is the unique minimal apolar scheme, and such a decomposition is the unique minimal GAD.
In this case, we demonstrate how to compute this minimal decomposition, using classical linear algebra tools, such as Schur factorization and eigen computations.

The numerical experiments demonstrate:
\begin{itemize}
  \item Near-optimal linear error growth in the low-perturbation regime.
  \item Consistent recovery of multiplicities and decomposition structure under small noise.
  \item Increased sensitivity for higher-order components and in higher-dimensional settings.
\end{itemize} 
Critical parts of the algorithm include the computation of the rank via SVD and the clustering of eigenvalues, which directly affect the stability and accuracy of the decomposition. 
We have identified the following limitations:
\begin{itemize}
  \item Breakdown of stability when the perturbation pushes the tensor too far from the suitable secant of the osculating varieties.
  \item Occasional failure in identifying correct point multiplicities when the considered tensor has a high degree or a large ambient space.
\end{itemize}
These insights motivate future improvements, such as conditioning-aware modifications, eigenstructure stabilization, and alternative formulations to handle critical cases more robustly.

When the regularity of a minimal GAD is not small enough, the approach cannot be applied directly. We plan to investigate further moment extension techniques, such as \cite{Brachat2010, Bernardi2013, Bernardi20201}, and to analyze the associated moment extension variety in order to provide an efficient method for computing the GAD and Cactus ranks.

\medskip
\paragraph{\emph{Acknowledgements.}}
This work has been supported by European Union’s HORIZON–MSCA-2023-DN-JD programme under the Horizon Europe (HORIZON) Marie Skłodowska-Curie Actions, grant agreement 101120296 (TENORS), by the Research Foundation - Flanders (FWO mandate: 12ZZC23N), and by the BOF project C16/21/002 by the Internal Funds KU Leuven.

\bibliographystyle{siam}
\bibliography{Biblio}

\end{document}